%% file: main.tex
\documentclass[final,3p,times]{elsarticle}

\usepackage{amssymb, amsmath, amsfonts}
\usepackage{amsthm}

\numberwithin{equation}{section}

\usepackage{graphicx}
\usepackage{tikz}
\usetikzlibrary{shapes.geometric, shapes.symbols, arrows.meta}
\usetikzlibrary{calc,angles}
\tikzset{block/.style={draw,rectangle, align=center, inner sep=5pt},
oval/.style={draw,ellipse, align = center, inner sep=10pt},
decision/.style={draw,diamond, aspect = 2, align=center,inner sep=2pt}}

\usepackage{hyperref}
\hypersetup{
    colorlinks=true,
    linkcolor=blue,
    filecolor=magenta,      
    urlcolor=cyan,
}
\usepackage{cleveref}

\usepackage[font=small, labelfont=bf]{caption}
\usepackage{enumitem}

\makeatletter
\def\namedlabel#1#2{\begingroup
   \def\@currentlabel{#2}%
   \label{#1}\endgroup
}
\makeatother

\journal{}

\input{style1}

\input{style2}

\begin{document}

\begin{frontmatter}



\title{Error analysis of the space-time interface-fitted finite element method for \\
an inverse source problem for an advection-diffusion equation \\
with moving subdomains}


\author[1]{Thi Thanh Mai Ta}
\ead{mai.tathithanh@hust.edu.vn, tathithanhmai@gmail.com}

\author[1]{Quang Huy Nguyen}
\ead{huy.nguyenquang1@hust.edu.vn}

\author[2]{Dinh Nho H{\` a}o\corref{cor}}
\ead{hao@math.ac.vn}

\address[1]{Faculty of Mathematics and Informatics, Hanoi University of Science and Technology, 11657 Hanoi, Vietnam}

\address[2]{Institute of Mathematics, Vietnam Academy of Science and Technology, 11307 Hanoi, Vietnam}

\cortext[cor]{Corresponding author}

\begin{abstract}
    A space-time interface-fitted approximation of an inverse source problem for the advection-diffusion equation with moving subdomains is investigated. The problem is reformulated as an optimization problem using Tikhonov regularization. A space-time interface-fitted method is employed to discretize the advection-diffusion equation, where two second-order a priori error estimates are established with respect to the $\Ls^2$-norms. Additionally, the regularized source is discretized sequentially using the variational approach, the element-wise constant discretization, and finally, the post-processing strategy. Optimal error estimates are achieved for the first two methods, while superlinear convergence is obtained for the third. Furthermore, a priori choices for the regularization parameter are proposed, depending on the mesh size and noise level. These choices ensure that the discrete and post-processing solutions strongly converge to the exact source as the mesh size and noise level tend to zero.
\end{abstract}




\end{frontmatter}


\section{Introduction}
\label{sec: introduction}

Let $\Omega$ be a bounded Lipschitz domain in $\mathbb{R}^d\ \left(d=1\text{ or } 2\right)$ with boundary $\partial\Omega$. The domain $\Omega$ is split into two time-dependent subdomains $\Omega_1(t)$ and $\Omega_2(t)$ by an interface $\Gamma(t)$, for all $t\in [0,T]$ with $T>0$. The interface $\Gamma(t)$ is transported by a velocity field $\vb = \vb\left(\xb, t\right)\in \Cs\left([0, T], \CCs^2(\Omega)\right)$. This velocity field satisfies the incompressibility condition $\nabla \cdot \vb(\xb, t) = 0$ for all $(\xb, t) \in \Om \times [0, T]$ \cite[Remark 2.2]{VR2018}. We denote by $Q_T:= \Omega\times (0,T)$ the space-time domain and 
$$
Q_i := \bigcup_{t\in (0,T)} \Omega_i(t)\times \left\{t\right\} \qqqq (i=1,2)
$$
two space-time subdomains separated by the space-time interface $\Gamma^\ast := \bigcup_{t\in (0,T)} \Gamma(t)\times \left\{t\right\}$. Assume that $\Gamma^\ast$ is a $\Cs^2$-regular hypersurface in $\mathbb{R}^{d+1}$, and that $\Gamma(t)\cap \partial\Omega = \varnothing$ for all $t\in [0,T]$. Consider the following problem
\begin{equation}
    \label{eq: state equation}
    \left\{\begin{array}{ll}
         \partial_t U + \vb \cdot \nabla U - \nabla \cdot\left(\kappa \nabla U\right) = F & \text{in } Q_T, \\
         \left[U\right] = 0 & \text{on } \Gamma^\ast,\\
         \left[\kappa \nabla U \cdot \mathbf{n}\right] = 0 & \text{on } \Gamma^\ast,\\
         U = 0 & \text{on } \partial\Omega \times (0,T), \\
         U\left(\cdot, 0\right) = U_0 & \text{in } \Omega,
    \end{array}\right.
\end{equation}
where $F$ refers to the source term, $U_0$ denotes the initial value, and $\mathbf{n}$ is the unit normal on $\Gamma(t)$, pointing from $\Omega_1(t)$ into $\Omega_2(t)$. The notation $\left[U\right] := U_{1 \, \mid\, \Gamma(t)} - U_{2 \, \mid\, \Gamma(t)}$ denotes the jump of $U$ across $\Gamma(t)$, where $U_{i \, \mid\,  \Gamma(t)}$ is the limiting value from $\Om_i(t)$ of $U\ (i=1,2)$. For simplicity, the diffusion coefficient $\kappa$ is assumed to be a positive constant within each subdomain
$$
\kappa := 
\begin{cases}
    \kappa_1 > 0 & \text{in}\q Q_1, \\
    \kappa_2 > 0 & \text{in}\q Q_2.
\end{cases}
$$ 

In many practical scenarios, the objective is to determine the source term $F$ on the right-hand side of \eqref{eq: state equation}, using additional observations of $U$. This is referred to as the inverse source problem. While numerous studies have addressed inverse source problems for parabolic equations (see the classical books \cite{Isakov1990, Hao1998}, the recent articles \cite{HHOT2017, HQS2021}, and the references therein), very few papers have examined such problems for advection-diffusion equations with moving subdomains. Related to our problem setting, Bellassoued and Yamamoto \cite{BY2006} investigated an inverse source problem for a parabolic transmission equation. Assuming that $F$ takes the form $F\left(\xb, t\right) = \ell\left(\xb, t\right)f\left(\xb\right)$ for all $\left(\xb,t\right)\in Q_T$, they established a conditional stability result for determining the spatial component $f\left(\xb\right)$ from a single measurement of $U$ on a fixed subdomain, where $\ell$ is a given smooth function. More recently, Chen et al. \cite{CJW2022} utilized a partial interior observation of $U$ to simultaneously reconstruct the initial value $U\left(\cdot, 0\right)$ and the spatial component $f\left(\xb\right)$. They derived a conditional stability result and proposed an iterative thresholding algorithm to solve the problem. In another related study, Zhang et al. \cite{ZLW2020} examined a distributed optimal control problem for a parabolic interface system. They conducted an error analysis of the finite element approximation for the problem and derived optimal error estimates for the control, state, and adjoint. However, it is important to note that in all these studies, the interface was assumed to be fixed, which is a special case of our problem setting.

Moreover, physical requirements often impose further constraints on the source term, the observations, or both. For example, in engineering applications and physical processes such as mass transport \cite{GR2011}, heat transfer \cite{Slodicka2021}, and electromagnetics \cite{LSV2021a}, it is evident that the source term is nonnegative. Let $\omega\subset \Omega_2(t)$ for all $t\in \left[0,T\right]$ be a nonempty fixed subdomain, and define $\omega_T := \omega \times \left(0,T\right)$ (see Figure \ref{fig: model}). Given two functions, $\ell\in \Ls^\infty\left(Q_T\right)$ and $g\in \Ls^2\left(Q_T\right)$, that satisfy $\ell\left(\xb,t\right)\ge \Ls >0$ and $g\left(\xb,t\right)\ge 0$ for almost every $\left(\xb,t\right) \in Q_T$, respectively, along with the partial interior data $U_d\in \Ls^2\left(\omega_T\right)$. This paper aims to present the numerical analysis of the following inverse source problem (see Subsection \ref{subsec: the inverse source problem} for details): 

\begin{description}
    \item[Inverse problem:] In Problem \eqref{eq: state equation}, assume that the source $F$ has the form $F\left(\xb,t\right) = \ell\left(\xb,t\right) f\left(\xb,t\right) + g\left(\xb,t\right)$ for all $\left(\xb,t\right)\in Q_T$, with the initial value $U_0 = 0$. Determine the component $f\in \Ls^2\left(Q_T\right)$ such that $U_{\mid\, \omega_T} = U_d$ and additionally, $f\left(\xb,t\right)\ge 0$ for almost every $\left(\xb,t\right)\in Q_T$.
\end{description}

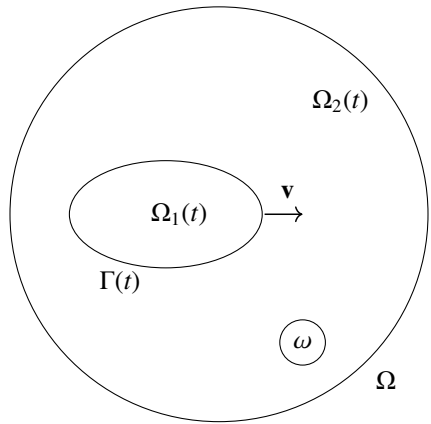
\begin{figure}[http]
    \centering
    \begin{tikzpicture}
        \draw (0.5,0) circle (2.75cm);
        \draw[->, line width=0.2mm] (1.1,0) -- (1.6,0);
        \node at (1.4, 0.3) {$\vb$};
        \node at (-0.8, -0.9) {$\Gamma(t)$};
        \node at (2.1, 1.5) {$\Omega_2(t)$}; 
        \node at (2.7, -2.2) {$\Omega$};
        \draw (1.6,-1.7) circle (0.3cm);
        \node at (1.6,-1.7) {$\omega$};
        \node[oval] at (-0.2, 0) {$\quad\Omega_1(t)$};
    \end{tikzpicture}
    \caption{The interface $\Gamma(t)$, which envolves by a velocity $\vb$, devides the domain $\Omega$ into two subdomains $\Omega_1(t)$ and $\Omega_2(t)$. The subdomain $\omega\subset\Omega_2(t)$ is fixed for all $t\in \left[0,T\right]$, considering $d=2$.}
    \label{fig: model}
\end{figure}

Since this problem is ill-posed, regularization is necessary to construct stable numerical approximations. In \cite{ZH2021}, Zhang and Hofmann proposed two iterative regularization methods to deal with an ill-posed linear operator equation under nonnegative constraints. Their problem is an abstract framework for our \textbf{Inverse Problem}. When the forward operator is injective, they establish a convergence result in terms of the discretization level and noise level. In contrast, we adopt Tikhonov regularization and reformulate the problem as a minimization problem. Subsequently, we analyze the existence, uniqueness, and optimality conditions of the minimizer, followed by its stability with respect to the noise.

On the other hand, discretizing Problem \eqref{eq: state equation} is essential. The primary challenge arises from a discontinuity of the gradient $\nabla U$ across the time-dependent interface $\Gamma(t)$, which results in suboptimal convergence rates for traditional finite element methods \cite{Babuka1970, CZ1998}. Various approaches have been proposed in the literature to address this difficulty, including the extended finite element method \cite{Zunino2013, LR2013}, the aggregation finite element method \cite{BDV2023}, and the arbitrary Lagrangian-Eulerian method \cite{BMF+2018, LRS2020}. Notably, all these studies employed classical time-stepping or time-discontinuous Galerkin schemes in combination with finite element methods. An alternative is the space-time method, which treats the time variable as an additional spatial variable. This approach discretizes the space-time domain using unstructured meshes and solves the problem by a space-time formulation, as demonstrated in \cite{LMN2016, LMS2019, GGS2024}, and more recently in \cite{NLPT2024}. In recent years, space-time methods have gained increasing prominence, particularly with advancements in parallel computing capabilities.

In this paper, we discretize the state and adjoint problems using the space-time interface-fitted method \cite{NLPT2024}. This approach facilitates efficient interface approximation without requiring expensive re-meshing algorithms or the computation of interaction integrals over the time-dependent interface. Furthermore, a priori error estimates for this method are provided when the solution exhibits globally low but locally high regularity, differing from the space-time methods presented in \cite{LST+2021, LST+2021b, LSY2024}, which require a global $\Hs^2$-smoothness condition to derive a priori error estimates. Specifically, we utilize element-wise linear functions to discretize the advection-diffusion equation. Notably, this low-order finite element approximation is well-suited to our problem, as the noise data typically exhibit low smoothness, which limits the adjoint regularity and, consequently, the finite element order. Furthermore, we extend the numerical analysis from \cite{NLPT2024} by deriving two second-order a priori error estimates with respect to the $\Ls^2$-norms under appropriate assumptions. 

Next, we focus on discretizing the regularized source. Among the numerous methods available in the literature, our focus is on the element-wise constant discretization, the post-processing strategy, and the variational approach. The element-wise constant discretization is the most classical method, first studied for elliptic optimal control problems \cite{Falk1973, CT2002}. This method is straightforward to implement and well-suited to a wide range of problems. However, it suffers from its slow convergence, being only of first order with respect to the $\Ls^2$-norm. The post-processing strategy constructs an improved approximation by projecting a term involving the discrete adjoint onto an admissible set. This approach achieves quadratic convergence in the $\Ls^2$-norm. Meyer and R{\" o}sch \cite{MR2004} originally introduced this core concept for elliptic optimal control problems, and it was subsequently applied in \cite{MV2008, TH2018} to parabolic optimal control problems. The third method, the variational approach, was developed by Hinze \cite{Hinze2005} for elliptic optimal control problems and also achieves second-order convergence. Recently, Zhang et al. \cite{ZLW2020} employed this strategy to discretize an optimal control problem for a parabolic interface system, resulting in optimal error estimates for the control, state, and adjoint. Unlike the two previous methods, the variational approach implicitly discretizes the control using a projection formula. 

In this work, we sequentially combine the space-time interface-fitted method with the variational approach, the element-wise constant discretization, and the post-processing strategy, to discretize the \textbf{Inverse problem}. The error analyses of these combinations are presented, yielding optimal error estimates for the first two methods and superlinear convergence for the third with respect to the $\Ls^2$-norm. Finally, we derive the overall error from regularization and discretization in terms of the regularization parameter, mesh size, and noise level. Following our approach in \cite{HLO2024}, we suggest a priori choices for the regularization parameter to ensure strong convergence of both the discrete and post-processing solutions to the exact source. Furthermore, we establish the corresponding convergence rates. To the best of our knowledge, such convergence rates have not been previously reported in the literature on inverse source problems for advection-diffusion equations with moving subdomains.

The remainder of this paper is organized as follows. In Section \ref{sec: optimization}, we introduce the function spaces and formulate our \textbf{Inverse problem}. Since this problem is ill-posed, we apply Tikhonov regularization to address this issue in Section \ref{sec: tikhonov regularization}. In Section \ref{sec: discretization}, we discretize the regularized problem. Under appropriate conditions, in Section \ref{subsec: auxiliary results}, we derive optimal error estimates for the state and adjoint in various norms. Section \ref{sec: error estimates} presents a numerical analysis of the three source discretization methods. The main results are summarized in Theorem \ref{theo: overall convergence rate} and Corollary \ref{coro: a priori choice for lambda}. Finally, we provide perspectives and remarks about future work.

\section{Problem setting}
\label{sec: optimization} 

In this section, we introduce the function spaces and the variational formulation of Problem \eqref{eq: state equation}. After that, we formulate our \textbf{Inverse problem}, discuss the concept of its solutions, and address its ill-posedness.

\subsection{Function spaces}
\label{subsec: functional setting}

For $s\ge 0, q\in [1, \infty]$ and a Lipschitz domain $\Omega$, the notation $\Ws^{s,q}\left(\Omega\right)$ refers to the classical Sobolev space, where $\Ws^{0,q}\left(\Omega\right)$ corresponds to the Lebesgue space $\Ls^q\left(\Omega\right)$ \cite[Section 2.2]{Ern2021}. The scalar product in $\Ls^2\left(\Omega\right)$ is denoted by $\left(\cdot,\cdot\right)_{\Ls^2\left(\Omega\right)}$, and the norm in $\Ls^q\left(\Omega\right)$ is denoted by $\norm{\cdot}_{\Ls^q\left(\Omega\right)}$. Among Sobolev spaces, only $\Ws^{s,2}\left(\Omega\right)$ forms a Hilbert space, which we specifically denote by $\Hs^s\left(\Omega\right)$. For the norm in $\Hs^s\left(\Omega\right)$, we use the notation $\norm{\cdot}_{\Hs^s(\Om)}$.

In this paper, $C > 0$ is a generic constant independent of the parameter $\lambda$, the mesh size $h$, the noise level $\varepsilon$, the data $U_d$, the function $g$, and the exact source $f_+$. However, it may depend on the space-time domain $Q_T$, the position of the space-time interface $\Gamma^\ast$, the norm $\norm{\vb}_{\LLs^\infty\left(Q_T\right)}$, the function $\ell$, and the coefficient $\kappa$. Additionally, let $\tau\ge 0$ denote an exponent of the parameter $\lambda$. Different values of $C$ and $\tau$ in different contexts are allowed. 

For $l,k\in \mathbb{N}$ and a space-time domain $Q_T$, we introduce an anisotropic Sobolev space
$$
\Hs^{l, k}\left(Q_T\right) = \left\{v\in \Ls^2\left(Q_T\right)\mid \partial_{\xb}^{\boldsymbol{\alpha}} \partial_t^r v \in \Ls^2\left(Q_T\right)\text{ for all } 0\le \abs{{\boldsymbol{\alpha}}}\le l,\ r=0,1,\ldots,k\right\},
$$
endowed with the norm $\norm{\cdot}_{\Hs^{l, k}\left(Q_T\right)}$ \cite[Section 1.4.1]{WYW2006}. When $l=k$, we obtain the standard Sobolev space $\Hs^k\left(Q_T\right)$. We denote by $\Hs^{1,0}_0\left(Q_T\right)$ the closure of $\Cs_0^1\left(Q_T\right)$ with respect to the norm $\norm{\cdot}_{\Hs^{1,0}\left(Q_T\right)}$. For convenience, we use the compact notation $\Ys := \Hs^{1,0}_0\left(Q_T\right)$, furnished with an equivalent norm
$$
\norm{v}_{\Ys}^2 := \int\limits_0^T \int\limits_\Om \kappa \abs{\nabla v}^2 \dx \dt \qqqq \forall v\in \Ys.
$$
The equivalence results from the Poincar{\' e}–Steklov inequality
\begin{equation}
    \label{eq: compare L2 and the Y norm}
    \norm{v}_{\Ls^2\left(Q_T\right)}\le C_P\norm{\nabla v}_{\LLs^2\left(Q_T\right)}\qqqq \forall v\in \Ys,
\end{equation}
where the constant $C_P>0$ depends only on the space-time domain $Q_T$ \cite[Lemma 3.27]{Ern2021}. The dual space of $\Ys$ is denoted by $\Ys^\prime$, and the duality pairing between $\Ys^\prime$ and $\Ys$ is denoted by $\inprod{\cdot, \cdot}$. Let us introduce the spaces
$$
\Xs := \brac{v \in \Ys \mid \partial_t v \in \Ys^\prime},\qqq \Xs_0 := \brac{v \in \Xs \mid v\left(\cdot, 0\right) = 0},\qqq \Xs_T := \brac{v \in \Xs \mid v\left(\cdot, T\right) = 0},
$$
equipped with the norm 
$$
\norm{v}^2_{\Xs} := \norm{v}_{\Ys}^2 + \norm{\partial_t v}_{\Ys^\prime}^2 \qqqq \forall v\in \Xs.
$$
Recall from \cite[Lemma 64.40]{Ern2021c} that in $\Xs$, the trace operator $v\in \Xs\rightarrow v\left(\cdot,t\right)\in \Ls^2\left(\Omega\right)$ is bounded for almost every $t\in \left[0,T\right]$. Specifically, the following inequality holds
\begin{equation}
    \label{eq: time trace inequality}
    \sup_{t\in \left[0,T\right]} \norm{v\left(\cdot,t\right)}_{\Ls^2\left(\Omega\right)}\le C\norm{v}_{\Xs}\qqqq \forall v\in \Xs.
\end{equation}

Besides these spaces, the space $\Hs^{1,0}\left(Q_1\cup Q_2\right)$ is also essential, as the trace operators $\gamma_i: \Hs^{1,0}\left(Q_i\right)\rightarrow \Ls^2\left(\Gamma^\ast\right)\ (i=1,2)$ are well-defined under mild assumptions on $\Gamma^\ast$ \cite[Theorem 2.1]{Lions1957}. In the subsequence, we will frequently use the space $\Ws := \Hs^{1}\left(Q_T\right)\cap\Hs^s\left(Q_1\cup Q_2\right)$, where $s>\frac{d+3}{2}$ is given. This space is equipped with the norm 
$$
\norm{v}^2_{\Ws}:= \norm{v}^2_{\Hs^1\left(Q_T\right)} + \norm{v}^2_{\Hs^s\left(Q_1\cup Q_2\right)} \qqqq \forall v\in \Ws.
$$

Let us further present the Stein extension operators \cite[Section VI.3.1]{Stein1971}, which are crucial for handling functions with global low but local high regularity. For any fixed $s\ge 0$ and a function $v \in \Hs^s\left(Q_1 \cup Q_2\right)$, denote by $v_i:=v_{\, \mid\, Q_i} \in \Hs^s\left(Q_i\right)$ the restriction of $v$ to the subdomain $Q_i\ \left(i=1,2\right)$. Assume that $\Gamma^\ast$ is a Lipschitz continuous hypersurface in $\mathbb{R}^{d+1}$, then there exist smooth extensions $\Es_i: \Hs^s\left(Q_i\right)\to \Hs^s\left(Q_T\right)$ such that
\begin{equation}
    \label{eq: extension operator}
    \Es_i v=v_i \qq \text{in } Q_i, \qqqq \norm{\Es_i v}_{\Hs^s\left(Q_T\right)} \leq C\norm{v_i}_{\Hs^s\left(Q_i\right)}\qq (i=1,2).
\end{equation}

\subsection{The advection-diffusion problem}
\label{subsec: the weak advection-diffusion problem}

Let $F\in \Ls^2\left(Q_T\right)$ and $U_0 \in \Hs^1_0\left(\Omega\right)$. Denote by $u_0 \in \Xs$ an extension of $U_0 \in \Hs^1_0\left(\Omega\right)$. Following \cite{NLPT2024}, the solution to Problem \eqref{eq: state equation} is defined as $U = \overline{u} + u_0 \in \Xs$, where $\overline{u}\in \Xs_0$ satisfies
$$
a\left(\overline{u}, \vphi\right) = \left(F, \vphi\right)_{\Ls^2\left(Q_T\right)} - a\left(u_0, \vphi\right)\qqqq \forall\vphi \in \Ys,
$$
with the bilinear form $a : \Xs\times \Ys \to \R$ given by
$$
a\left(u, \vphi\right) := \inprod{\partial_t u, \vphi}+\int\limits_0^T\int\limits_\Omega  \left(\vb\cdot\nabla u\right)\vphi + \kappa \nabla u \cdot \nabla \vphi \dx \dt.
$$
By the Banach-Ne{\v c}as-Babu{\v s}ka theorem \cite[Theorem 25.9]{Ern2021b}, this problem admits a unique solution $\overline{u}\in \Xs_0$ satisfying
\begin{equation}
    \label{eq: priori estimate}
    \norm{\overline{u}}_{\Xs} \leq C \left(\norm{F}_{\Ls^2\left(Q_T\right)} + \norm{u_0}_{\Xs}\right),
\end{equation}
where the constant $C>0$ is independent of $F$ and $u_0$. Consequently, for any choice of $u_0\in \Xs$, the following problem is well-posed
\begin{equation}
    \label{eq: general weak problem}
    a\left(U, \vphi\right) = \left(F, \vphi\right)_{\Ls^2\left(Q_T\right)}\qqqq \forall\vphi \in \Ys.
\end{equation}

\subsection{The inverse source problem}
\label{subsec: the inverse source problem}

Given $\ell\in \Ls^\infty\left(Q_T\right)$ and $g\in \Ls^2\left(Q_T\right)$ such that $\ell \ge \Ls >0$ and $g \ge 0$ at almost everywhere in $Q_T$, respectively. Let $F=\ell f+g$ with $f\in \Ls^2\left(Q_T\right)$ and $U_0 = 0$ in \eqref{eq: general weak problem}. In this case, $u_0=0$ and hence $U=\overline{u}\in \Xs_0$. Note that $\overline{u}\in \Xs_0$ can be separated as $\overline{u} = u + u^{\ast}$, where $u\in X_0$ solves the problem
\begin{equation}
    \label{eq: weak state equation}
    a\left(u, \vphi\right) = \left(\ell f, \vphi\right)_{\Ls^2\left(Q_T\right)}\qqqq \forall\vphi \in \Ys,
\end{equation}
and $u^{\ast}\in \Xs_0$ solves the problem
\begin{equation}
    \label{eq: weak state equation with g}
    a\left(u^{\ast}, \vphi\right) = \left(g, \vphi\right)_{\Ls^2\left(Q_T\right)}\qqqq \forall\vphi \in \Ys.
\end{equation}
Since $u^{\ast}$ is uniquely determined, our \textbf{Inverse problem} reduces to finding $f\in F_+$ in \eqref{eq: weak state equation} from the data $U_d\in \Ls^2\left(\omega_T\right)$ in the subdomain $\omega_T$, where the admissible set $F_+$ is defined as
\begin{equation}
    \label{eq: admissible set}
    F_{+} = \left\{f\in \Ls^2\left(Q_T\right)\mid f\ge 0 \text{ at almost everywhere in } Q_T\right\}.
\end{equation}
This set is nonempty, closed, and convex. Mathematically, we solve an operator equation with a priori information
\begin{equation}
    \label{eq: the inverse problem}
    Af = z_d, \qqqq f\in F_{+},
\end{equation}
where $A$ is a bounded linear operator, defined via \eqref{eq: weak state equation}
$$
\begin{aligned}
    A: \Ls^2\left(Q_T\right) &\to \Ls^2\left(\omega_T\right),\\
    f &\mapsto u\left(f\right)_{\, \mid \, \omega_T}.
\end{aligned}
$$
Here, we use the notation $u\left(f\right)$ to emphasize the dependence of $u$ in \eqref{eq: weak state equation} on $f$. The notation $z_d:=U_d - u^{\ast}_{\, \mid \, \omega_T}\in \Ls^2\left(\omega_T\right)$ represents the exact data. To avoid ambiguity, we interpret $z_d \equiv 0$ in $Q_T\setminus \overline{\omega_T}$ so that it is well-defined in $\Ls^2\left(Q_T\right)$.

Let us discuss the concept of solutions to Problem \eqref{eq: the inverse problem}. Firstly, this problem may not have solutions, since $z_d$ can be outside the restricted range $A\left(F_{+}\right)$. On the other hand, we can construct examples where two solutions of Problem \eqref{eq: weak state equation} coincide in $Q_2$, and hence in $\omega_T$, but exhibit different behavior in $Q_1$. Therefore, the operator $A$ is generally not injective, and Problem \eqref{eq: the inverse problem} may have many solutions. Consequently, it is essential to recall from \cite[Section 5.4]{EHN1996} the following definition:

\begin{definition}
    \label{defi: minimum norm solution}
    Let $F_{+}$ be the admissible set in \eqref{eq: admissible set}. An element $f_+ \in F_{+}$ is called the exact source if among all $f\in F_{+}$ that solve Problem \eqref{eq: the inverse problem}, it has the minimal $\Ls^2\left(Q_T\right)$-norm, i.e., 
    $$
    \norm{f_+ }_{\Ls^2\left(Q_T\right)} \le \norm{f}_{\Ls^2\left(Q_T\right)}.
    $$
\end{definition}

Clearly, $f_+ \in F_{+}$ is uniquely determined. Regarding the ill-posedness of Problem \eqref{eq: the inverse problem}, note that despite $A$ being a linear operator, this problem is nonlinear, due to the presence of the inequality constraint. Therefore, the ill-posedness criterion for linear problems (please refer to, for example, \cite[Section 1.5]{IVT2002}) does not apply. Instead, we adopt the local ill-posedness concepts in \cite[Definition 3]{HP2018} for nonlinear problems. The compact embedding $\Hs^{1,0}\left(Q_T\right)\hookrightarrow \Ls^2\left(Q_T\right)$ \cite[Theorem 2.35]{Ern2021} implies that $A$ is a compact operator. Thus, combined with the arguments in \cite[Section 1]{ZH2021}, we conclude that Problem \eqref{eq: the inverse problem} is locally ill-posed at every point in $F_{+}$.

\section{Tikhonov regularization}
\label{sec: tikhonov regularization}

The ill-posedness of Problem \eqref{eq: the inverse problem} indicates that its approximated solution does not depend continuously on the data. Therefore, regularization is required to overcome this challenge and derive a stable solution. In this work, we employ Tikhonov regularization: We approximate Problem \eqref{eq: the inverse problem} by the following problem
\begin{equation}
    \label{eq: problem formulation}
    \begin{aligned}
        & \min_{f\in F_{+}} J^{\varepsilon}_{\lambda}\left(f\right):= \dfrac{1}{2}\norm{u\left(f\right) - z_d^{\varepsilon}}^2_{\Ls^2\left(\omega_T\right)} + \dfrac{\lambda}{2}\norm{f}^2_{\Ls^2\left(Q_T\right)}, \\
        & \text{ subject to \eqref{eq: weak state equation}}.
    \end{aligned}
\end{equation}
Given a noise level $\varepsilon>0$, we denote by $U_d^{\varepsilon}\in \Ls^2\left(\omega_T\right)$ the imprecise observation of $U_d\in \Ls^2\left(\omega_T\right)$ that satisfies
\begin{equation}
    \label{eq: noise level}
    \norm{U_d^{\varepsilon} - U_d}_{\Ls^2\left(\omega_T\right)}\le \varepsilon,
\end{equation}
and $z_d^{\varepsilon}:= U_d^{\varepsilon} - u^{\ast}_{\, \mid \, \omega_T} \in \Ls^2\left(\omega_T\right)$ is the noisy data. As with Problem \eqref{eq: the inverse problem}, the data $z_d^{\varepsilon}$ is well-defined in $\Ls^2\left(Q_T\right)$ by interpreting $z^\varepsilon_d \equiv 0$ out side of $\omega_T$. Finally, $\lambda >0$ denotes the regularization parameter.

\subsection{Existence of solutions and optimality conditions}

To demonstrate the effectiveness of Tikhonov regularization, we first prove that Problem \eqref{eq: problem formulation} admits a unique solution. While the following theorem is based on classical results, we revisit it here for clarity and completeness within the context of regularizing an inverse problem with moving subdomains.

\begin{theorem}
    \label{theo: existence of the solution}
    For any fixed $\lambda>0$, Problem \eqref{eq: problem formulation} has a unique solution $f_{\lambda}^{\varepsilon}\in F_{+}$.
\end{theorem}

\begin{proof}
    Recall that the set $F_{+}$ is nonempty. Together with $J^{\varepsilon}_{\lambda}\left(f\right)\ge 0$ on $F_{+}$, we deduce that $\inf\limits_{f\in F_{+}} J^{\varepsilon}_{\lambda}\left(f\right)$ is finite. Hence, there exists a sequence $\left\{f_n\right\}_{n \in \mathbb{N}}\subset F_{+}$ such that 
    \begin{equation}
        \label{eq: existence of the solution 3}
        \lim_{n \rightarrow \infty} J^{\varepsilon}_{\lambda}\left(f_n\right) = \inf\limits_{f\in F_{+}} J^{\varepsilon}_{\lambda}\left(f\right).
    \end{equation}
    The inequality $\norm{f_n}^2_{\Ls^2\left(Q_T\right)} \le \frac{2}{\lambda}J^{\varepsilon}_{\lambda}\left(f_n\right)$ for all $n\in \mathbb{N}$ implies that the sequence $\left\{f_n\right\}_{n \in \mathbb{N}}$ is bounded in $\Ls^2\left(Q_T\right)$, which allows us to extract a (not relabeled) weakly convergent subsequence $\left\{f_n\right\}_{n \in \mathbb{N}}$ such that $f_n \rightharpoonup f_{\lambda}^{\varepsilon}$ in $\Ls^2\left(Q_T\right)$, with $f_{\lambda}^{\varepsilon} \in \Ls^2\left(Q_T\right)$. Moreover, there exists a sufficiently large $\rho > 0$ such that
    $$
    \left\{f_n\right\}_{n \in \mathbb{N}}\subset F_{+} \cap \overline{B}\left(\rho\right),
    $$
    where $\overline{B}\left(\rho\right)$ denotes a closed ball with radius $\rho >0$ in $\Ls^2\left(Q_T\right)$. Since $F_{+} \cap \overline{B}\left(\rho\right)$ is a closed, bounded, and convex subset of $\Ls^2\left(Q_T\right)$, it is weakly sequentially compact \cite[Theorem 2.10]{Troltzsch2010}. This gives us $f^\varepsilon_\lambda\in F_{+}$. 
    
    Consider the variational problem: Find $u_n := u\left(f_n\right)\in \Xs_0$ that satisfies
    \begin{equation}
        \label{eq: existence of the solution 1}
        a\left(u_n, \vphi\right) = \left(\ell f_n, \vphi\right)_{\Ls^2\left(Q_T\right)}\qqqq \forall\vphi \in \Ys.
    \end{equation}
    This problem is well-posed. From \eqref{eq: priori estimate}, we know that 
    $$
    \norm{u_n}_{\Xs} \le C \norm{f_n}_{\Ls^2\left(Q_T\right)}\qqqq \forall n \in \mathbb{N},
    $$
    which implies that the sequence $\left\{u_n\right\}_{n \in \mathbb{N}}$ is bounded in $\Xs$. Therefore, there exists $u_{\lambda}^{\varepsilon}\in \Xs$ and a (not relabeled) weakly convergent subsequence $\left\{u_n\right\}_{n \in \mathbb{N}}$ such that $u_n \rightharpoonup u_{\lambda}^{\varepsilon}$ in $\Xs$. Thus, for all $\vphi\in \Ys$, we have
    $$
    \lim_{n\to \infty} \left[\inprod{\partial_t u_n, \vphi}+ \int\limits_0^T\int\limits_\Omega \left(\vb\cdot \nabla u_n\right) \vphi + \kappa\nabla u_n \cdot \nabla \vphi \dx\dt\right] = \inprod{\partial_t u^\varepsilon_\lambda, \vphi} + \int\limits_0^T\int\limits_\Omega \left(\vb\cdot \nabla u^\varepsilon_\lambda\right) \vphi + \kappa\nabla u^\varepsilon_\lambda\cdot \nabla \vphi \dx\dt.
    $$
    Taking the limit in \eqref{eq: existence of the solution 1}, we obtain
    \begin{equation}
        \label{eq: existence of the solution 2}
        a\left(u_{\lambda}^{\varepsilon}, \vphi\right) = \left(\ell f_{\lambda}^{\varepsilon}, \vphi\right)_{\Ls^2\left(Q_T\right)}\qqqq \forall\vphi \in \Ys.
    \end{equation}
    Next, we will prove that $u^\varepsilon_\lambda \in \Xs_0$. To do this, we choose $\vphi\in \Xs_T$ in \eqref{eq: existence of the solution 1} and apply the integration by parts formula with $u_n\in \Xs_0$ and $\vphi\in \Xs_T$. One obtains
    $$
    -\inprod{\partial_t u_n, \vphi} + \int\limits_0^T\int\limits_\Omega \left(\vb\cdot\nabla u_n\right)\vphi + \kappa\nabla u_n \cdot\nabla \vphi\dx\dt = \left(\ell f_n, \vphi\right)_{\Ls^2\left(Q_T\right)}.
    $$
    We take $n\to \infty$ to get
    $$
    -\inprod{\partial_t u^\varepsilon_\lambda, \vphi} + \int\limits_0^T\int\limits_\Omega \left(\vb\cdot\nabla u^\varepsilon_\lambda\right)\vphi + \kappa\nabla u^\varepsilon_\lambda \cdot\nabla \vphi\dx\dt = \left(\ell f^\varepsilon_\lambda, \vphi\right)_{\Ls^2\left(Q_T\right)}.
    $$
    On the other hand, by integrating \eqref{eq: existence of the solution 2} by parts with $u^\varepsilon_\lambda\in \Xs$ and $\vphi\in \Xs_T$, we deduce that
    $$
    -\inprod{\partial_t u^\varepsilon_\lambda, \vphi} + \int\limits_0^T\int\limits_\Omega \left(\vb\cdot\nabla u^\varepsilon_\lambda\right)\vphi + \kappa\nabla u^\varepsilon_\lambda \cdot\nabla \vphi\dx\dt = \left(\ell f^\varepsilon_\lambda, \vphi\right)_{\Ls^2\left(Q_T\right)} + \int\limits_\Omega u^\varepsilon_\lambda\left(\xb, 0\right)\vphi\left(\xb, 0\right)\dx.
    $$
    From the last two identities, we infer that $u^\varepsilon_\lambda\left(\cdot,0\right)=0$, and hence $u^\varepsilon_\lambda\in \Xs_0$. Combining this with \eqref{eq: existence of the solution 2}, we conclude that $u^\varepsilon_\lambda = u\left(f^\varepsilon_\lambda\right)$. 
    
    Finally, we apply \eqref{eq: existence of the solution 3} and the lower semi-continuity of the $\Ls^2$-norms to get
    $$
    \begin{aligned}
        \inf\limits_{f\in F_{+}} J^{\varepsilon}_{\lambda}\left(f\right)=\liminf_{n \rightarrow \infty} J_{\lambda}^{\varepsilon}\left(f_n\right)&= \liminf_{n \rightarrow \infty} \dfrac{1}{2} \norm{u\left(f_n\right)-z^\varepsilon_d}^2_{\Ls^2\left(\omega_T\right)} + \liminf_{n \rightarrow \infty} \dfrac{\lambda}{2} \norm{f_n}^2_{\Ls^2\left(Q_T\right)}\\
        &\ge \dfrac{1}{2} \norm{u_{\lambda}^{\varepsilon}-z^\varepsilon_d}^2_{\Ls^2\left(\omega_T\right)} + \dfrac{\lambda}{2} \norm{f_{\lambda}^{\varepsilon}}^2_{\Ls^2\left(Q_T\right)}\\
        &= J_{\lambda}^{\varepsilon}\left(f_{\lambda}^{\varepsilon}\right),
    \end{aligned}
    $$
    which indicates that $f_{\lambda}^{\varepsilon}\in F_{+}$ is a minimizer. Uniqueness follows from the strict convexity of the functional $J_{\lambda}^{\varepsilon}$. The proof is complete.
\end{proof}

Next, we derive the optimality conditions for Problem \eqref{eq: problem formulation}. To do so, let us introduce the following adjoint problem: Find $p\left(f\right)\in \Xs_T$ such that
\begin{equation}
    \label{eq: weak adjoint equation}
    a^{\prime}\left(p\left(f\right),\phi\right) = \left(\chi_{\omega_T}\left(u\left(f\right)-z_d^{\varepsilon}\right), \phi\right)_{\Ls^2\left(Q_T\right)}\qqqq \forall \phi\in \Ys,
\end{equation}
where the bilinear form $a^\prime: \Xs \times \Ys \rightarrow \mathbb{R}$ is defined as
$$
a^{\prime}\left(p,\phi\right) := -\inprod{\partial_t p, \phi}+\int\limits_0^T\int\limits_\Omega  -\left(\vb\cdot\nabla p\right)\phi + \kappa \nabla p \cdot \nabla \phi \dx \dt,
$$
and $\chi_{\omega_T}$ is the characteristic function of the subdomain $\omega_T$. By changing the time and velocity field directions and applying \cite[Theorem 2.1]{NLPT2024}, we conclude the well-posedness of this problem.

\begin{theorem}
    \label{theo: optimality conditions}
    The unique solution $f^{\varepsilon}_{\lambda} \in F_{+}$ to Problem \eqref{eq: problem formulation}, together with the corresponding state $u^{\varepsilon}_{\lambda} \in \Xs_0$ and adjoint $p^{\varepsilon}_{\lambda} \in \Xs_T$, satisfies the following optimality conditions
    \begin{equation}
        \label{eq: weak state equation optimality conditions}
        a\left(u^{\varepsilon}_{\lambda}, \vphi\right) = \left(\ell f^{\varepsilon}_{\lambda}, \vphi\right)_{\Ls^2\left(Q_T\right)}\qqqq \forall\vphi \in \Ys,
    \end{equation}
    and 
    \begin{equation}
        \label{eq: weak adjoint equation optimality conditions}
        a^{\prime}\left(p^{\varepsilon}_{\lambda},\phi\right) = \left(\chi_{\omega_T}\left(u^{\varepsilon}_{\lambda}-z^{\varepsilon}_d\right),\phi\right)_{\Ls^2\left(Q_T\right)}\qqqq \forall \phi\in \Ys,
    \end{equation}
    and the variational inequality 
    \begin{equation}
        \label{eq: variational inequality}
        \left(\ell p^{\varepsilon}_{\lambda} + \lambda f^{\varepsilon}_{\lambda}, f- f^{\varepsilon}_{\lambda}\right)_{\Ls^2\left(Q_T\right)}\ge 0 \qqqq \forall f\in F_{+}.
    \end{equation}
    This inequality is equivalent to the projection formula
    \begin{equation}
        \label{eq: projection formula of f}
        f^{\varepsilon}_{\lambda}=\operatorname{Proj}_{F_{+}}\left(-\dfrac{1}{\lambda} \ell p^{\varepsilon}_{\lambda}\right),
    \end{equation}
    where the operator $\operatorname{Proj}_{F_{+}}: \Ls^2\left(Q_T\right)\rightarrow F_{+}$ is defined for all $v  \in \Ls^2\left(Q_T\right)$ as
    $$
    \operatorname{Proj}_{F_{+}}\left(v \right)\left(\xb, t\right) := \max \left\{0, v \left(\xb, t\right)\right\} \qqqq \text{for almost every } \left(\xb, t\right)\in Q_T.
    $$
\end{theorem}

\begin{proof}
    Following the classical arguments in \cite[Lemma 2.21]{Troltzsch2010}, we prove that the functional $J^\varepsilon_\lambda$ defined in \eqref{eq: problem formulation} is Fr{\' e}chet differentiable and its gradient $\nabla J^\varepsilon_\lambda\left(f\right)$ at $f\in F_{+}$ is given by
    $$
    \nabla J^\varepsilon_\lambda\left(f\right)= \ell p\left(f\right) + \lambda f,
    $$
    where $p\left(f\right)\in \Xs_T$ is the solution to Problem \eqref{eq: weak adjoint equation}. \\
    To proceed, let us take a small variation $\delta f \in \Ls^2\left(Q_T\right)$ of $f\in F_{+}$; we have 
    $$
    \begin{aligned}
        &J^\varepsilon_\lambda\left(f + \delta f \right) - J^\varepsilon_\lambda\left(f \right)=\\
        &=\dfrac{1}{2}\norm{u\left(f + \delta f \right)-z^\varepsilon_d}^2_{\Ls^2\left(\omega_T\right)} - \dfrac{1}{2}\norm{u\left(f\right) -z^\varepsilon_d}^2_{\Ls^2\left(\omega_T\right)} + \dfrac{\lambda}{2}\norm{f+ \delta f}^2_{\Ls^2\left(Q_T\right)} - \dfrac{\lambda}{2}\norm{f}^2_{\Ls^2\left(Q_T\right)}\\
        &=\dfrac{1}{2}\norm{u\left(f + \delta f \right) - u\left(f\right)}^2_{\Ls^2\left(\omega_T\right)} + \left(u\left(f + \delta f \right) - u\left(f\right), u\left(f\right)-z^\varepsilon_d\right)_{\Ls^2\left(\omega_T\right)} + \dfrac{\lambda}{2}\norm{\delta f}^2_{\Ls^2\left(Q_T\right)} + \lambda \left(f, \delta f\right)_{\Ls^2\left(Q_T\right)}\\
        &=\dfrac{1}{2}\norm{u\left(\delta f\right)}^2_{\Ls^2\left(\omega_T\right)} + \left(u\left(\delta f\right), u\left(f\right)-z^\varepsilon_d\right)_{\Ls^2\left(\omega_T\right)} + \dfrac{\lambda}{2}\norm{\delta f}^2_{\Ls^2\left(Q_T\right)} + \lambda \left(f, \delta f\right)_{\Ls^2\left(Q_T\right)}.
    \end{aligned}
    $$
    Here, $u\left(\delta f\right) \in \Xs_0$ is the solution to the problem
    \begin{equation}
        \label{eq: state equation difference}
        a\left(u\left(\delta f\right), \vphi\right) = \left(\ell \delta f, \vphi\right)_{\Ls^2\left(Q_T\right)}\qqqq \forall\vphi \in \Ys.
    \end{equation}
    Thanks to \eqref{eq: compare L2 and the Y norm} and \eqref{eq: priori estimate}, we have $\norm{u\left(\delta f\right)}_{\Ls^2\left(\omega_T\right)} \le C \norm{u\left(\delta f\right)}_{\Xs} \le C \norm{\delta f}_{\Ls^2\left(Q_T\right)}$, resulting in
    $$
    \norm{u\left(\delta f\right)}^2_{\Ls^2\left(\omega_T\right)} = o\left(\norm{\delta f}_{\Ls^2\left(Q_T\right)}\right) \qqqq \text{as } \norm{\delta f}_{\Ls^2\left(Q_T\right)}\rightarrow 0.
    $$
    Thus, we imply
    \begin{equation}
        \label{eq: difference of J}
        J^\varepsilon_\lambda\left(f + \delta f \right) - J^\varepsilon_\lambda\left(f \right) = \left(u\left(\delta f\right), u\left(f\right)-z^\varepsilon_d\right)_{\Ls^2\left(\omega_T\right)} + \lambda \left(f, \delta f\right)_{\Ls^2\left(Q_T\right)} + o\left(\norm{\delta f}_{\Ls^2\left(Q_T\right)}\right).
    \end{equation}
    To derive the functional gradient, we rewrite the first term on the right-hand side of \eqref{eq: difference of J} as a scalar product in the solution space. Let $p\left(f\right)\in \Xs_T$ be the solution to Problem \eqref{eq: weak adjoint equation}. We choose $\phi = u\left(\delta f\right)\in \Xs_0$ in \eqref{eq: weak adjoint equation} to arrive at
    \begin{equation}
        \label{eq: optimality conditions 1}
        \begin{aligned}
            \left(u\left(\delta f\right), u\left(f\right)-z^\varepsilon_d\right)_{\Ls^2\left(\omega_T\right)} &= \left(u\left(\delta f\right), \chi_{\omega_T}\left(u\left(f\right)-z^\varepsilon_d\right)\right)_{\Ls^2\left(Q_T\right)} \\
            &= -\inprod{\partial_t p\left(f\right), u\left(\delta f\right)}+\int\limits_0^T\int\limits_\Omega  -\left[\vb\cdot\nabla p\left(f\right)\right]u\left(\delta f\right) + \kappa \nabla p\left(f\right) \cdot \nabla u\left(\delta f\right) \dx \dt.
        \end{aligned}
    \end{equation}
    By integrating by parts with $p\left(f\right)\in \Xs_T$ and $u\left(\delta f\right) \in \Xs_0$, we rewrite $\inprod{\partial_t p\left(f\right), u\left(\delta f\right)}$ as
    $$
    \begin{aligned}
        \inprod{\partial_t p\left(f\right), u\left(\delta f\right)} &= \int\limits_\Omega p\left(f\right)\left(\xb,T\right)u\left(\delta f\right)\left(\xb,T\right)\dx - \int\limits_\Omega p\left(f\right)\left(\xb,0\right)u\left(\delta f\right)\left(\xb,0\right)\dx - \inprod{\partial_t u\left(\delta f\right), p\left(f\right)} \\
        &= - \inprod{\partial_t u\left(\delta f\right), p\left(f\right)}.
    \end{aligned}
    $$
    To handle the advection part on the right-hand side of \eqref{eq: optimality conditions 1}, we invoke the divergence theorem, the incompressibility condition $\nabla\cdot \vb \left(\xb, t\right)= 0$ for all $\left(\xb, t\right)\in \Omega\times [0, T]$, and the homogeneous Dirichlet boundary condition. We obtain
    \begin{equation}
        \label{eq: optimality conditions 2}
        \begin{aligned}
            \int\limits_0^T\int\limits_\Omega \left[\vb\cdot\nabla p\left(f\right)\right] u\left(\delta f\right) \dx\dt &= \int\limits_0^T\int\limits_\Omega \nabla\cdot\left[p\left(f\right)u\left(\delta f\right) \vb\right] - \left[\vb\cdot\nabla u\left(\delta f\right)\right]p\left(f\right) - p\left(f\right)u\left(\delta f\right)\left(\nabla\cdot \vb\right)\dx\dt \\
            &= \int\limits_0^T\int\limits_{\partial\Omega} p\left(f\right)u\left(\delta f\right) \vb \cdot \mathbf{n}_{\Omega} \ds\dt - \int\limits_0^T\int\limits_\Omega \left[\vb\cdot\nabla u\left(\delta f\right)\right]p\left(f\right) \dx\dt\\
            &= - \int\limits_0^T\int\limits_\Omega \left[\vb\cdot\nabla u\left(\delta f\right)\right]p\left(f\right) \dx\dt,
        \end{aligned}
    \end{equation}
    where $\mathbf{n}_\Omega$ is the outward normal to $\partial\Omega$. Hence, \eqref{eq: optimality conditions 1} becomes
    $$
    \left(u\left(\delta f\right), u\left(f\right)-z^\varepsilon_d\right)_{\Ls^2\left(\omega_T\right)} = \inprod{\partial_t u\left(\delta f\right), p\left(f\right)} + \int\limits_0^T\int\limits_\Omega \left[\vb\cdot\nabla u\left(\delta f\right)\right]p\left(f\right) + \kappa \nabla u\left(\delta f\right) \cdot \nabla p\left(f\right) \dx\dt,
    $$
    which leads to
    \begin{equation}
        \label{eq: optimality conditions 3}
        \left(u\left(\delta f\right), u\left(f\right)-z^\varepsilon_d\right)_{\Ls^2\left(\omega_T\right)} = a\left(u\left(\delta f\right), p\left(f\right)\right) = \left(\ell p\left(f\right),\delta f\right)_{\Ls^2\left(Q_T\right)},
    \end{equation}
    where $\vphi =  p\left(f\right)\in \Xs_T$ was chosen in \eqref{eq: state equation difference} for the first equality. By substituting this into \eqref{eq: difference of J}, we conclude the Fr{\' e}chet differentiability of the functional $J^\varepsilon_\lambda$, together with its gradient.
\end{proof}

\subsection{Convergence}
\label{subsec: convergence}

This subsection investigates some convergence properties of Tikhonov regularization. We begin by showing that the solution $f^{\varepsilon}_{\lambda} \in F_{+}$ to Problem \eqref{eq: problem formulation} is stable with respect to the perturbations in the data $z^{\varepsilon}_d\in \Ls^2\left(\omega_T\right)$. The following theorem is a constrained variant of the results presented in \cite[Theorem 2.1]{EKN1989} and \cite[Theorem 2.2]{JLW2020}.

\begin{theorem}
    \label{theo: stable convergence}
    For any fixed $\lambda > 0$, let $\left\{z_n\right\}_{n\in \mathbb{N}}\subset \Ls^2\left(\omega_T\right)$ be a sequence that strongly converges to $z_d^{\varepsilon}$ in $\Ls^2\left(\omega_T\right)$, and let $\left\{f_n\right\}_{n\in \mathbb{N}}\subset F_{+}$ denote the sequence of solutions to the corresponding problems
    \begin{equation}
        \label{eq: stable convergence}
        \begin{aligned}
            & \min_{f\in F_{+}} \dfrac{1}{2}\norm{u\left(f\right) - z_n}^2_{\Ls^2\left(\omega_T\right)} + \dfrac{\lambda}{2}\norm{f}^2_{\Ls^2\left(Q_T\right)},\qqqq n\in \mathbb{N},\\
            & \text{ subject to \eqref{eq: weak state equation}}.
        \end{aligned}
    \end{equation}
    Then, $\left\{f_n\right\}_{n\in \mathbb{N}}$ strongly converges to the solution $f^{\varepsilon}_{\lambda} \in F_{+}$ of Problem \eqref{eq: problem formulation} in $\Ls^2\left(Q_T\right)$.
\end{theorem}

\begin{proof}
    By Theorem \ref{theo: existence of the solution}, for each $n\in \mathbb{N}$, there exists a unique solution $f_n \in F_{+}$ to Problem \eqref{eq: stable convergence}. For any $f\in F_{+}$, we have
    $$
    \dfrac{1}{2}\norm{u\left(f_n\right) - z_n}^2_{\Ls^2\left(\omega_T\right)} + \dfrac{\lambda}{2}\norm{f_n}^2_{\Ls^2\left(Q_T\right)}\le \dfrac{1}{2}\norm{u\left(f\right) - z_n}^2_{\Ls^2\left(\omega_T\right)} + \dfrac{\lambda}{2}\norm{f}^2_{\Ls^2\left(Q_T\right)} \qqqq \forall n \in \mathbb{N},
    $$
    which implies the boundedness of the sequence $\left\{f_n\right\}_{n \in \mathbb{N}}$ in $\Ls^2\left(Q_T\right)$. By applying the technique in Theorem \ref{theo: existence of the solution}, we conclude the existence of an element $f^\varepsilon_\lambda\in F_{+}$ and a (not relabed) weakly convergent subsequence $\left\{f_n\right\}_{n \in \mathbb{N}}$ such that $f_n \rightharpoonup f^\varepsilon_\lambda$ in $\Ls^2\left(Q_T\right)$ and $u\left(f_n\right) \rightharpoonup u\left(f^\varepsilon_\lambda\right)$ in $\Xs$. This, together with the strong convergence of the sequence $\left\{z_n\right\}_{n\in \mathbb{N}}$ to $z^\varepsilon_d$ in $\Ls^2\left(\omega_T\right)$, gives us $u\left(f_n\right) - z_n \rightharpoonup u\left(f^\varepsilon_\lambda\right) - z^\varepsilon_d$ in $\Ls^2\left(\omega_T\right)$, and hence
    \begin{equation}
        \label{eq: stable convergence 3}
        \liminf_{n\to \infty} \norm{u\left(f_n\right) - z_n}^2_{\Ls^2\left(\omega_T\right)} \ge \norm{u\left(f^\varepsilon_\lambda\right) - z^\varepsilon_d}^2_{\Ls^2\left(\omega_T\right)}.
    \end{equation}
    By invoking the lower semi-continuity of the $\Ls^2$-norms, we deduce that
    \begin{equation}
        \label{eq: stable convergence 1}
        \begin{aligned}
        J^{\varepsilon}_{\lambda}\left(f\right) = \dfrac{1}{2}\norm{u\left(f\right) - z_d^{\varepsilon}}^2_{\Ls^2\left(\omega_T\right)} + \dfrac{\lambda}{2}\norm{f}^2_{\Ls^2\left(Q_T\right)} &\ge \lim\limits_{n\to \infty}\left(\dfrac{1}{2}\norm{u\left(f\right) - z_n}^2_{\Ls^2\left(\omega_T\right)} + \dfrac{\lambda}{2}\norm{f}^2_{\Ls^2\left(Q_T\right)}\right)\\
        &\ge \liminf\limits_{n\to \infty}\left(\dfrac{1}{2}\norm{u\left(f_n\right) - z_n}^2_{\Ls^2\left(\omega_T\right)} + \dfrac{\lambda}{2}\norm{f_n}^2_{\Ls^2\left(Q_T\right)}\right)\\
        &\ge \dfrac{1}{2}\norm{u\left(f^\varepsilon_\lambda\right) - z^\varepsilon_d}^2_{\Ls^2\left(\omega_T\right)} + \dfrac{\lambda}{2}\norm{f^\varepsilon_\lambda}^2_{\Ls^2\left(Q_T\right)} = J^{\varepsilon}_{\lambda}\left(f^\varepsilon_\lambda\right),
    \end{aligned}
    \end{equation}
    for all $f\in F_{+}$. Therefore, $f^\varepsilon_\lambda \in F_{+}$ solves Problem \eqref{eq: problem formulation}.
    
    Next, we will prove that the sequence $\left\{f_n\right\}_{n\in \mathbb{N}}$ strongly converges to $f^{\varepsilon}_{\lambda}$ in $\Ls^2\left(Q_T\right)$. By contradiction, suppose that $\lim\limits_{n\to \infty} \norm{f_n}_{\Ls^2\left(Q_T\right)} \ne \norm{f^\varepsilon_\lambda}_{\Ls^2\left(Q_T\right)}$, then
    \begin{equation}
        \label{eq: stable convergence 2}
        \theta := \limsup\limits_{n\to \infty} \norm{f_n}_{\Ls^2\left(Q_T\right)} > \liminf\limits_{n\to \infty} \norm{f_n}_{\Ls^2\left(Q_T\right)} \ge \norm{f^\varepsilon_\lambda}_{\Ls^2\left(Q_T\right)}.
    \end{equation}
    Since there exists a (not relabeled) subsequence $\left\{f_{n}\right\}_{n\in \mathbb{N}}$ such that $\lim\limits_{n\to \infty}\norm{f_{n}}_{\Ls^2\left(Q_T\right)} = \theta$, we choose $f= f^\varepsilon_\lambda \in F_{+}$ in \eqref{eq: stable convergence 1} to have
    $$
    J^\varepsilon_\lambda \left(f^\varepsilon_\lambda\right) = \liminf\limits_{n \to \infty}\left(\dfrac{1}{2}\norm{u\left(f_{n}\right) - z_{n}}^2_{\Ls^2\left(\omega_T\right)} + \dfrac{\lambda}{2}\norm{f_{n}}^2_{\Ls^2\left(Q_T\right)}\right)= \liminf\limits_{n \to \infty}\dfrac{1}{2}\norm{u\left(f_{n}\right) - z_{n}}^2_{\Ls^2\left(\omega_T\right)} + \dfrac{\lambda}{2}\theta^2.
    $$
    Combine this with \eqref{eq: stable convergence 2}, we arrive at
    $$
    \dfrac{1}{2}\norm{u\left(f^\varepsilon_\lambda\right) - z^\varepsilon_d}^2_{\Ls^2\left(\omega_T\right)} = \liminf\limits_{n \to \infty}\dfrac{1}{2}\norm{u\left(f_{n}\right) - z_{n}}^2_{\Ls^2\left(\omega_T\right)} + \dfrac{\lambda}{2}\left(\theta^2 - \norm{f^\varepsilon_\lambda}^2_{\Ls^2\left(Q_T\right)}\right) > \liminf\limits_{n \to \infty}\dfrac{1}{2}\norm{u\left(f_{n}\right) - z_{n}}^2_{\Ls^2\left(\omega_T\right)},
    $$
    which contradicts \eqref{eq: stable convergence 3}. The proof is finished. 
\end{proof}

Regarding the error in regularizing the exact source using Tikhonov regularization, let us recall the following result from \cite[Theorems 5.16 and 5.19]{EHN1996} for general linear inverse problems with convex constraints:

\begin{lemma}
    \label{lem: convergence rate}
    Let $f_+ \in F_{+}$ be the exact source and $f^{\varepsilon}_{\lambda}\in F_{+}$ be the solution to Problem \eqref{eq: problem formulation}. Assume that there exists $\zeta\in \Ls^2\left(\omega_T\right)$ with the minimal $\Ls^2\left(\omega_T\right)$-norm such that $f_+ =\operatorname{Proj}_{F_{+}}\left(A^\ast\zeta\right)$, where $A^\ast: \Ls^2\left(\omega_T\right) \to \Ls^2\left(Q_T\right)$ denotes the adjoint of the operator $A$ in \eqref{eq: the inverse problem}. Then, the following inequality holds
    $$
    \norm{f_+ - f^{\varepsilon}_{\lambda}}_{\Ls^2\left(Q_T\right)}\le \sqrt{\lambda}\norm{\zeta}_{\Ls^2\left(\omega_T\right)} + \dfrac{\varepsilon}{\sqrt{\lambda}}.
    $$
\end{lemma}

\section{Discretization}
\label{sec: discretization}

Assume that $\Omega$ is a polyhedron in $\mathbb{R}^{d}\ \left(d=1\text{ or } 2\right)$. The space-time domain $Q_T = \Omega\times \left(0, T\right)$ is partitioned into shape-regular simplicial finite elements by a family of interface-fitted meshes $\left\{\mathcal{T}_h\right\}_{h\in \left(0,h_\ast\right)}$, where $h\in \left(0,h_\ast\right)$ denotes the mesh size, with $h_\ast>0$ be given. More precisely, each element $K\in \mathcal{T}_h$ is a triangle (when $d=1$) or a tetrahedron (when $d=2$), and it can be classified into one of the following cases:

\begin{enumerate}
    \item $K$ lies entirely within $\overline{Q_1}$;
    \item $K$ lies entirely within $\overline{Q_2}$;
    \item $K$ intersects $\Gamma^\ast$ at $d+1$ vertices.
\end{enumerate}
For further details, please refer to \cite{NLPT2024}. Additionally, assume that the mesh $\mathcal{T}_h$ is quasi-uniform. Let $\Gamma_h^\ast$ denote the linear approximation of $\Gamma^\ast$, consisting of all edges (or faces) with nodes lying on $\Gamma^\ast$. The discrete interface $\Gamma_h^\ast$ divides $Q_T$ into two subdomains $Q_{1,h}$ and $Q_{2,h}$, which are the approximated counterparts of $Q_1$ and $Q_2$, respectively.

\subsection{Space-time interface-fitted method for state and adjoint problems}
\label{subsec: space-time interface-fitted method}

This subsection recalls the space-time interface-fitted method from \cite{NLPT2024}: On the mesh $\mathcal{T}_h$, we define the following space-time finite element spaces
$$
\Vs_{h,0} := \left\{\vphi_h \in \Cs\left(\overline{Q_T}\right)\mid \vphi_{h\, \mid \, K}\in \mathbb{P}_1\left(K\right) \text{ for all } K\in \mathcal{T}_h \text{, and }\vphi_h=0 \text{ on } \partial\Omega\times(0,T) \text{ and }\Omega\times\left\{0\right\}\right\},
$$
and
$$
\Vs_{h,T} := \left\{\phi_h \in \Cs\left(\overline{Q_T}\right)\mid \phi_{h\, \mid \, K}\in \mathbb{P}_1\left(K\right) \text{ for all } K\in \mathcal{T}_h \text{, and }\phi_h=0 \text{ on } \partial\Omega\times(0,T) \text{ and }\Omega\times\left\{T\right\}\right\},
$$
where for $m\in \mathbb{N}$ and $K\in \mathcal{T}_h$, the space $\mathbb{P}_m\left(K\right)$ consists of all polynomials of degree $m$ on $K$. Clearly, $\Vs_{h,0}\subset \Xs_0$ and $\Vs_{h,T}\subset \Xs_T$. 

We first consider the discrete state problem: For $\ell\in \Ls^{\infty}\left(Q_T\right)$ and $f\in \Ls^{2}\left(Q_T\right)$, find $u_h\left(f\right)\in \Vs_{h,0}$ that satisfies
\begin{equation}
    \label{eq: discrete state equation}
    a_h\left(u_h\left(f\right), \vphi_h\right) = \left(\ell f,\vphi_h\right)_{\Ls^2\left(Q_T\right)}\qqqq \forall \vphi_h\in \Vs_{h,0}.
\end{equation}
The bilinear form $a_h: \Xs_0 \times \Ys \to \mathbb{R}$ is defined as
$$
a_h\left(u, \vphi\right) := \inprod{\partial_t u, \vphi} + \int\limits_0^T\int\limits_\Omega \left(\vb\cdot \nabla u\right)\vphi +  \kappa_h \nabla u \cdot \nabla \vphi \dx \dt,
$$
where $\kappa_h$ approximates $\kappa$ by means of
$$
\kappa_h := 
\begin{cases}
    \kappa_{1}>0 & \text{in} \q Q_{1,h}, \\ 
    \kappa_{2}>0 & \text{in} \q Q_{2,h}.
\end{cases}
$$
Regarding the numerical analysis, we introduce the seminorm
$$
\norm{v}_{h}^2 : = \sum_{i=1}^2\int\limits_{Q_{i,h}} \kappa_i \abs{\nabla v}^2 \dx \dt \qqqq \fa v \in \Hs^{1,0}\left(Q_{1,h}\cup Q_{2,h}\right).
$$
When $v\in\Ys \subset \Hs^{1,0}\left(Q_{1,h}\cup Q_{2,h}\right)$, the right-hand side of this expression simplifies to
$$
\sum_{i=1}^2\int\limits_{Q_{i,h}} \kappa_i \abs{\nabla v}^2 \dx \dt =  \int\limits_0^T\int\limits_\Omega \kappa_h \abs{\nabla v}^2 \dx \dt,
$$
which defines an equivalent norm in $\Ys$, also denoted by $\norm{v}_{h}$. This norm, incorporating the coefficient $\kappa_h$, provides advantages for analyzing discrete problems compared to the norm $\norm{v}_{\Ys}$. Moreover, for $v\in \Xs_0$, it follows that
$$
\begin{aligned}
    \inprod{\partial_t v, v} +\int\limits_0^T\int\limits_\Omega \left(\vb\cdot \nabla v\right) v + \kappa_h \nabla v \cdot \nabla v \dx \dt &=\dfrac{1}{2}\norm{v\left(\cdot, T\right)}^2_{\Ls^2\left(\Omega\right)} - \dfrac{1}{2}\norm{v\left(\cdot, 0\right)}^2_{\Ls^2\left(\Omega\right)} - \int\limits_0^T\int\limits_\Omega \left(\vb\cdot \nabla v\right) v \dx\dt + \norm{v}_{h}^2\\
    &=\dfrac{1}{2}\norm{v\left(\cdot, T\right)}^2_{\Ls^2\left(\Omega\right)} - \dfrac{1}{2}\norm{v\left(\cdot, 0\right)}^2_{\Ls^2\left(\Omega\right)} + \norm{v}_{h}^2\\
    &\ge \norm{v}_{h}^2,
\end{aligned}
$$
owing to the integration by parts formula and the technique in \eqref{eq: optimality conditions 2}. Therefore, we obtain
\begin{equation}
    \label{eq: coercivity of ah}
    a_h\left(v,v\right) \ge \norm{v}_{h}^2 \qqqq \forall v \in \Xs_0.
\end{equation}
This inequality will be used later. On the space $\Hs^{1}\left(Q_{1,h}\cup Q_{2,h}\right)$, we introduce the seminorm
$$
\norm{v}_{h,0}^2 := \norm{v}_{h}^2 + \norm{\xi_h\left(v\right)}_{h}^2\qqqq \forall v \in \Hs^{1}\left(Q_{1,h}\cup Q_{2,h}\right),
$$
where $\xi_h\left(v\right)\in \Vs_{h,0}$ is the unique solution to the problem
$$
\int\limits_0^T\int\limits_\Omega \kappa_h \nabla \xi_h\left(v\right) \cdot \nabla \psi_h \dx \dt = \sum_{i=1}^2\int\limits_{Q_{i,h}} (\partial_t v) \psi_h \dx\dt \qqqq \forall \psi_h\in \Vs_{h,0}.
$$
By choosing $\psi_h = \xi_h\left(v\right)$ and applying \eqref{eq: compare L2 and the Y norm}, we end up with
\begin{equation}
    \label{eq: compare L2 and the X norm}
    \norm{\xi_h\left(v\right)}_{h}\le C \norm{\partial_t v}_{\Ls^2\left(Q_{1,h}\cup Q_{2,h}\right)}\qqqq \forall v \in \Hs^{1}\left(Q_{1,h}\cup Q_{2,h}\right).
\end{equation}
The following discrete stability condition holds
\begin{equation}
    \label{eq: stability}
    \sup_{\vphi_h \in \Vs_{h,0}\setminus \{0\}} \dfrac{a_h\left(u_h\left(f\right), \vphi_h\right)}{\norm{\vphi_h}_{h}} \ge C_s \norm{u_h\left(f\right)}_{h,0} \qqqq \fa u_h\left(f\right) \in \Vs_{h,0},
\end{equation}
ensuring that Problem \eqref{eq: discrete state equation} is well-posed, see \cite[Lemma 3.1]{NLPT2024}. Here, $C_s>0$ is a constant that depends only on the space-time domain $Q_T$, the norm $\norm{\vb}_{\LLs^\infty\left(Q_T\right)}$, and the coefficient $\kappa$.

Next, we present the discretization of Problem \eqref{eq: weak adjoint equation}. Similarly, we determine $p_h\left(f\right)\in \Vs_{h,T}$ such that
\begin{equation}
    \label{eq: discrete adjoint equation}
    a_h^{\prime}\left(p_h\left(f\right), \phi_h\right) = \left(\chi_{\omega_T}\left(u\left(f\right)-z_{d}^{\varepsilon}\right),\phi_h\right)_{\Ls^2\left(Q_T\right)}\qqqq \forall \phi_h\in \Vs_{h,T},
\end{equation}
where the bilinear form $a_h^{\prime}: \Xs_T \times \Ys \to \mathbb{R}$ is given by
$$
a_h^{\prime}\left(p, \phi\right) := -\inprod{\partial_t p, \phi} + \int\limits_0^T\int\limits_\Omega -\left(\vb\cdot \nabla p\right)\phi +  \kappa_h \nabla p \cdot \nabla \phi \dx \dt.
$$
On the space $\Hs^1\left(Q_{1,h}\cup Q_{2,h}\right)$, we introduce another seminorm
$$
\norm{v}_{h,T}^2 := \norm{v}_{h}^2 + \norm{\xi_h^\prime\left(v\right)}_{h}^2\qqqq \forall v \in \Hs^{1}\left(Q_{1,h}\cup Q_{2,h}\right),
$$
where $\xi_h^\prime\left(v\right)\in \Vs_{h,T}$ denotes the unique solution to the problem
$$
\int\limits_0^T\int\limits_\Omega \kappa_h \nabla \xi_h^\prime\left(v\right) \cdot \nabla \psi_h^\prime \dx \dt = \sum_{i=1}^2\int\limits_{Q_{i,h}} (\partial_t v) \psi_h^\prime \dx\dt \qqqq \forall \psi_h^\prime\in \Vs_{h,T}.
$$
Employing the technique for \eqref{eq: stability}, we establish the following stability condition
\begin{equation}
    \label{eq: stability adjoint}
    \sup_{\phi_h \in \Vs_{h,T}\setminus \{0\}} \dfrac{a^{\prime}_h\left(p_h\left(f\right), \phi_h\right)}{\norm{\phi_h}_{h}} \ge C_s \norm{p_h\left(f\right)}_{h,T} \qqqq \fa p_h\left(f\right) \in \Vs_{h,T},
\end{equation}
and concludes the well-posedness of Problem \eqref{eq: discrete adjoint equation}. 

\subsection{Discretization of the regularized source}
\label{subsec: discretization of the source}

For a subspace $F^d \subseteq \Ls^2\left(Q_T\right)$, which may be either finite-dimensional or the entire space $\Ls^2\left(Q_T\right)$, we denote by $F_{+}^h := F^d \cap F_{+}$ the discrete admissible set. We call an element $f^{\varepsilon}_{\lambda, h}\in F^h_{+}$ the discrete source if, together with the corresponding state $u^{\varepsilon}_{\lambda, h} \in \Vs_{h,0}$ and adjoint $ p^{\varepsilon}_{\lambda, h}\in \Vs_{h, T}$, it solves the system
\begin{equation}
    \label{eq: discrete state optimality conditions}
    a_h\left(u^{\varepsilon}_{\lambda, h}, \vphi_h\right) = \left(\ell f^{\varepsilon}_{\lambda, h}, \vphi_h\right)_{\Ls^2\left(Q_T\right)} \qqqq \forall\vphi_h \in \Vs_{h,0},
\end{equation}
and
\begin{equation}
    \label{eq: discrete adjoint optimality conditions}
    a^{\prime}_h\left(p^{\varepsilon}_{\lambda, h},\phi_h\right) = \left(\chi_{\omega_T}\left(u^{\varepsilon}_{\lambda, h} - z_{d,h}^{\varepsilon}\right),\phi_h\right)_{\Ls^2\left(Q_T\right)} \qqqq \forall \phi_h\in \Vs_{h,T},
\end{equation}
and the variational inequality
\begin{equation}
    \label{eq: discrete variational inequality}
    \left(\ell p^{\varepsilon}_{\lambda, h} + \lambda f^{\varepsilon}_{\lambda, h}, f_h- f^{\varepsilon}_{\lambda, h}\right)_{\Ls^2\left(Q_T\right)} \ge 0 \qqqq \forall f_h\in F^h_{+},
\end{equation}
where $z^\varepsilon_{d,h}:= U^\varepsilon_d - u^\ast_{h\, \mid\, \omega_T}\in \Ls^2\left(\omega_T\right)$ denotes the discrete data. Similar to Problem \eqref{eq: the inverse problem}, $z^\varepsilon_{d,h}$ can be interpreted as an element of $\Ls^2\left(Q_T\right)$. Here, $u^\ast_h \in \Vs_{h,0}$ approximates $u^\ast \in \Xs_0$ in \eqref{eq: weak state equation with g}, also using the space-time interface-fitted method. It can be defined analogously to $u_h\left(f\right)\in \Vs_{h,0}$ in \eqref{eq: discrete state equation}. 

This paper will not discuss the existence and uniqueness of solutions to Problem \eqref{eq: discrete state optimality conditions}--\eqref{eq: discrete variational inequality}. For simplicity, we assume that the discrete source exists uniquely.

The way we choose $F^d$ leads to the corresponding discretization method. Our work focuses on the following three approaches:

\begin{enumerate}
    \item Variational approach: We choose $F^d$ to be the whole space $\Ls^2\left(Q_T\right)$, which means that the two sets $F_{+}$ and $F^h_{+}$ coincide \cite{Hinze2005}. By invoking a projection form of \eqref{eq: discrete variational inequality}, we convert the discretization of the regularized source into the discrete treatment for a term involving the adjoint. We will address the error of this method in Subsection \ref{subsec: error variational}.
    \item Element-wise constant discretization: For simplicity, assume that the mesh $\mathcal{T}_h$ is used for discretizing the regularized source. We choose $F^d$ to be the space of element-wise constant functions \cite{Falk1973}, i.e.,
    \begin{equation}
        \label{eq: step functions space}
        F^d=\left\{f_h \in \Ls^{\infty}\left(Q_T\right)\mid f_{h\, \mid\, K}\in \mathbb{P}_0\left(K\right) \text{ for all } K\in \mathcal{T}_h\right\}.
    \end{equation}
    The error of this approach will be estimated in Subsection \ref{subsec: error constant}.
    \item Postprocessing strategy: $F^d$ is chosen as in the element-wise constant discretization. After computing the element-wise constant solution $f^{\varepsilon}_{\lambda, h}$ to Problem \eqref{eq: discrete state optimality conditions}--\eqref{eq: discrete variational inequality}, we project the term $\ell p^{\varepsilon}_{\lambda, h}$ onto $F_{+}$ to obtain a better approximation $f_h^\dagger$ \cite{MR2004}. Here, $p^{\varepsilon}_{\lambda, h}$ is defined in \eqref{eq: discrete adjoint optimality conditions}, and $f_h^\dagger$ is called the post-processing solution
    \begin{equation}
        \label{eq: post-processing solution}
        f_h^\dagger :=\operatorname{Proj}_{F_{+}}\left(-\dfrac{1}{\lambda} \ell p^{\varepsilon}_{\lambda, h}\right).
    \end{equation}
    The error analysis of this strategy will be discussed in Subsection \ref{subsec: error postprocessing}.
\end{enumerate}

\section{A priori error estimates for advection-diffusion problems}
\label{subsec: auxiliary results}

In this section, we focus on discretization errors of the state and adjoint, along with the stability of the discrete state. First, we present an auxiliary result regarding the mismatch between each space-time subdomain $Q_i$ and its approximated counterpart $Q_{i,h}\ \left(i=1,2\right)$. Let us define
$$
S_h^1 := Q_{1,h}\setminus \overline{Q_1 }= Q_2\setminus \overline{Q_{2,h}}\qqq \text{and} \qqq S_h^2 := Q_{2,h}\setminus \overline{Q_2} = Q_1\setminus \overline{Q_{1,h}},
$$
and $S_h := S_h^1\cup S_h^2$ (see Figure \ref{fig: mismatch region}). We denote by $\mathcal{T}_h^\ast := \left\{K\in \mathcal{T}_h\mid K\cap \Gamma^\ast \ne \varnothing\right\}$ the set of all interface elements.

\begin{figure}[!htp]
    \centering
    \includegraphics[width=0.35\textwidth]{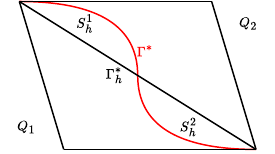}
    \caption{The region $S_h=S_h^1\cup S_h^2$ lies between the space-time interface $\Gamma^\ast$ (red) and the discrete one $\Gamma_h^\ast$ (black), considering $d=1$.}
    \label{fig: mismatch region}
\end{figure}

\begin{lemma}
    Assume that $\Gamma^\ast$ is a $\Cs^2$-continuous hypersurface in $\mathbb{R}^{d+1}\ \left(d=1\text{ or } 2\right)$. Then, for all $K\in \mathcal{T}_h^\ast$, we have 
    \begin{equation}
        \label{eq: discrepancy region}
        \abs{K \cap S_h} \le C h^{d+2}.
    \end{equation}
    There holds for the cardinality of $\mathcal{T}_h^\ast$ that
    \begin{equation}
        \label{eq: cardinality of Th-ast}
        \sum_{K\in \mathcal{T}_h^\ast} 1 \le C h^{-d}.
    \end{equation}
    The area (or volume) of the region formed by all interface elements is bounded by 
    \begin{equation}
        \label{eq: Th^ast area}
        \sum_{K\in \mathcal{T}_h^\ast}\abs{K}\le C h.
    \end{equation}
\end{lemma}

\begin{proof}
    When $d=1$, the proof of the first inequality can be found in \cite[Section 2]{CZ1998}. The second inequality follows by combining \eqref{eq: discrepancy region} with \cite[Lemma 3.3.4]{Fei1986}. All the arguments can be extended to the case $d=2$ without essential changes. As a consequence of \eqref{eq: cardinality of Th-ast} and the inequality $\abs{K}\le C h^{d+1}$ for all $K\in \mathcal{T}_h$, we obtain \eqref{eq: Th^ast area}.
\end{proof}

Let $v\in \Hs^1\left(Q_T\right)\cap \Hs^2\left(Q_1 \cup Q_2\right)$. By the Sobolev embedding theorem \cite[Theorem 2.35]{Ern2021}, if $v \in \Hs^2\left(Q_1 \cup Q_2\right)$, then $v\in \Cs\left(\overline{Q_1}\right)\cap \Cs\left(\overline{Q_2}\right)$. Additionally, if $v\in \Hs^1\left(Q_T\right)$, it follows from \cite[Theorem 18.8]{Ern2021} that $\gamma_1 v - \gamma_2 v = 0$, where $\gamma_1$ and $\gamma_2$ are the trace operators defined in Subsection \ref{subsec: functional setting}. Consequently, $v \in \Cs\left(\overline{Q_T}\right)$. 

We now study the approximability of the Lagrangian interpolant. Let $I_{h}: \Cs\left(\overline{Q_T}\right) \rightarrow \Vs_{h,0}$ be the nodal interpolation operator \cite[Section 19.3]{Ern2021}. For the case $Q_T \subset \mathbb{R}^2$, Chen and Zou established the interpolation estimate in \cite[Lemma 2.1]{CZ1998}. However, the order of convergence was only nearly optimal, up to a factor of $\sqrt{\abs{\log h}}$, where $h$ denotes the mesh size. 

This paper imposes an additional condition on $v$ and extends Chen and Zou's approach to recover an optimal error estimate for both cases $Q_T \subset \mathbb{R}^2$ and $Q_T \subset \mathbb{R}^3$. We obtain the following result:

\begin{lemma}
    \label{lem: interpolation error}
    For all $v\in \Ws$, the interpolation operator $I_h$ satisfies the following inequality
    $$
    \norm{v - I_h v}_{\Ls^2\left(Q_T\right)} + h \norm{\Ds\left(v - I_h v\right)}_{\LLs^2\left(Q_T\right)} \le C h^2 \norm{v}_{\Hs^s\left(Q_1\cup Q_2\right)},
    $$
    where $\Ds:=\left(\nabla, \partial_t \right)^\top$ denotes the space-time gradient operator.
\end{lemma}

\begin{proof}
    Let us present the estimate of $\norm{v - I_h v}_{\Ls^2\left(Q_T\right)}$. The result for $\norm{\Ds\left(v - I_h v\right)}_{\LLs^2\left(Q_T\right)}$ can be proved with the same technique. The idea is to first estimate the interpolation error on an arbitrary element $K\in\mathcal{T}_h$, then sum over all elements to obtain the desired result. 
    
    Under the regularity assumption, we have $v\in \Hs^2\left(K\right)$ on any $K\notin \mathcal{T}_h^\ast$. The classical interpolation theory \cite[Theorem 11.13]{Ern2021} gives us
    \begin{equation}
        \label{eq: interpolation error 1}
        \norm{v - I_h v}_{\Ls^2\left(K\right)} \le C h^2 \norm{v}_{\Hs^2\left(K\right)}.
    \end{equation}
    
    Next, consider an arbitrary element $K\in \mathcal{T}_h^\ast$. Without loss of generality, suppose that $K\cap S_h \subset Q_1$ and $K\setminus \overline{S_h} \subset Q_2$. For $v \in \Hs^s\left(Q_1 \cup Q_2\right)$ with $s>\frac{d+3}{2}$, note that $\Es_i v \in \Hs^s\left(Q_T\right)\subset \Ws^{1,\infty}\left(Q_T\right)\ (i=1,2)$ \cite[Theorem 2.31]{Ern2021}, where $\Es_1$ and $\Es_2$ are the extension operators in \eqref{eq: extension operator}. By applying \eqref{eq: discrepancy region} and the classical interpolation theory, we have
    $$
    \begin{aligned}
        \norm{v - I_h v}^2_{\Ls^2\left(K\right)} &= \norm{\Es_1 v - I_h\left(\Es_1 v\right)}^2_{\Ls^{2}\left(K\cap S_h\right)} + \norm{\Es_2 v - I_h\left(\Es_2 v\right)}^2_{\Ls^{2}\left(K\setminus \overline{S_h}\right)}\\
        &\le \abs{K\cap S_h}\norm{\Es_1 v - I_h\left(\Es_1 v\right)}^2_{\Ls^{\infty}\left(K\cap S_h\right)} + \norm{\Es_2 v - I_h\left(\Es_2 v\right)}^2_{\Ls^2\left(K\setminus \overline{S_h}\right)}\\
        &\le C h^{d+2}\norm{\Es_1 v - I_h\left(\Es_1 v\right)}^2_{\Ls^{\infty}\left(K\right)} + \norm{\Es_2 v - I_h \left(\Es_2 v\right)}^2_{\Ls^2\left(K\right)}\\
        &\le C h^{d+4}\norm{\Ds\left(\Es_1 v\right)}^2_{\LLs^{\infty}\left(K\right)} + C h^4\norm{\Ds^2\left(\Es_2 v\right)}^2_{\mathbb{L}^2\left(K\right)},
    \end{aligned}
    $$
    where $\Ds^2$ refers to the space-time Hessian operator. We then sum over all elements $K\in \mathcal{T}_h^\ast$. By using \eqref{eq: cardinality of Th-ast}, the Sobolev embedding $\HHs^{s-1}\left(Q_T\right)\hookrightarrow \LLs^{\infty}\left(Q_T\right)$ for $s>\frac{d+3}{2}$ \cite[Theorem 2.31]{Ern2021}, and \eqref{eq: extension operator} again, we get
    \begin{equation}
        \label{eq: interpolation error 2}
        \begin{aligned}
            &\sum_{K\in \mathcal{T}_h^\ast} \norm{v - I_h v}^2_{\Ls^2\left(K\right)} \le\\
            &\le C h^{d+4} \max{\left\{\norm{\Ds\left(\Es_1 v\right)}_{\LLs^{\infty}\left(Q_T\right)}^2, \norm{\Ds\left(\Es_2 v\right)}_{\LLs^{\infty}\left(Q_T\right)}^2\right\}} \left(\sum_{K\in \mathcal{T}_h^\ast} 1\right) + Ch^4\left(\sum_{i=1}^2\norm{\Ds^2\left(\Es_i v\right)}^2_{\mathbb{L}^2\left(Q_T\right)}\right)\\
            &\le C h^4 \max{\left\{\norm{\Ds\left(\Es_1 v\right)}_{\HHs^{s-1}\left(Q_T\right)}^2, \norm{\Ds\left(\Es_2 v\right)}_{\HHs^{s-1}\left(Q_T\right)}^2\right\}} + C h^4 \norm{v}^2_{\Hs^2\left(Q_1\cup Q_2\right)}\\
            &\le C h^4 \norm{v}^2_{\Hs^s\left(Q_1\cup Q_2\right)}.
        \end{aligned}
    \end{equation}
    
    From \eqref{eq: interpolation error 1} and \eqref{eq: interpolation error 2}, we obtain $\norm{v - I_h v}_{\Ls^2\left(Q_T\right)} \le C h^2 \norm{v}_{\Hs^s\left(Q_1\cup Q_2\right)}$. By following the same arguments, we can prove that $\norm{\Ds\left(v - I_h v\right)}_{\LLs^2\left(Q_T\right)} \le Ch\norm{v}_{\Hs^s\left(Q_1\cup Q_2\right)}$. The proof is complete.
\end{proof}

For ease of presentation, we will neglect the dependence of both the state and adjoint on $f$ in this section. For instance, the state is denoted by $u$ instead of $u\left(f\right)$. Note that the mismatch between $\mathcal{T}_h$ and $Q_T$ at the interface leads to the inconsistency of $a_h\left(\cdot,\cdot\right)$. In particular, we have the following lemma:

\begin{lemma}
    \label{lem: discrepancy between a and ah}
    Let $u\in \Xs_0$ and $u_h\in \Vs_{h,0}$ be the solutions to Problems \eqref{eq: weak state equation} and \eqref{eq: discrete state equation}, respectively. Then, for all $\vphi_h \in \Vs_{h,0}$, the following equality holds
    $$
    a_h\left(u-u_h,\vphi_h\right) = \int\limits_{S_h} \left(\kappa_h - \kappa\right) \nabla u \cdot\nabla \vphi_h \dx\dt.
    $$
\end{lemma}

\begin{proof}
    For any $\vphi_h \in \Vs_{h,0}$, we invoke \eqref{eq: weak state equation} and \eqref{eq: discrete state equation} to have
    $$
    \begin{aligned}
        a_h\left(u,\vphi_h\right) &= \inprod{\partial_t u, \vphi_h} + \int\limits_0^T\int\limits_\Omega \left(\vb\cdot \nabla u\right)\vphi_h +  \kappa_h \nabla u \cdot \nabla \vphi_h \dx \dt\\
        &= a\left(u,\vphi_h\right) + \int\limits_0^T\int\limits_\Omega \left(\kappa_h - \kappa\right) \nabla u \cdot\nabla \vphi_h \dx\dt\\
        &=\left(\ell f, \vphi_h \right)_{\Ls^2\left(Q_T\right)} + \int\limits_{S_h} \left(\kappa_h - \kappa\right) \nabla u \cdot\nabla \vphi_h \dx\dt\\
        &= a_h\left(u_h,\vphi_h\right) + \int\limits_{S_h} \left(\kappa_h - \kappa\right) \nabla u \cdot\nabla \vphi_h \dx\dt,
    \end{aligned}
    $$
    observed that $\kappa_h - \kappa$ vanishes everywhere outside of $S_h$.
\end{proof}

Next, we present a priori estimates for the state error in various norms. The following result estimates that error in the norm $\norm{\cdot}_{h,0}$ (please see \cite[Corollary 3.1]{NLPT2024} for more details). 

\begin{lemma}
    \label{lem: state error estimate discrete norm}
    Let $u\in \Xs_0$ and $u_h\in \Vs_{h,0}$ be the solutions to Problems \eqref{eq: weak state equation} and \eqref{eq: discrete state equation}, respectively. Assume that $u\in \Ws$, then we have the following estimate
    $$
    \norm{u-u_h}_{h,0} \le C h \norm{u}_{\Hs^s\left(Q_1\cup Q_2\right)}.
    $$
\end{lemma}

We proceed by estimating the state error in the $\Ls^2\left(\Omega\right)$-norm at $t=T$. In doing so, we apply Nitsche's trick \cite{Nitsche1968}. For any $\eta\in \Ls^2\left(Q_T\right)$ and any $y_T\in \Hs^1_0\left(\Omega\right)$, it follows from the arguments for Problem \eqref{eq: general weak problem} (after reversing the time and velocity field directions) that there exists a solution $y \in \Xs$ to the problem 
\begin{equation}
    \label{eq: Nitsche's trick equation}
    a^{\prime}\left(y, \varsigma\right) = \left(\eta, \varsigma\right)_{\Ls^2\left(Q_T\right)} \qqqq \forall \varsigma \in \Ys,
\end{equation}
along with the final condition $y\left(\cdot, T\right) = y_T$. In this context, Problem \eqref{eq: Nitsche's trick equation} is also called an adjoint problem. It should not be confused with Problem \eqref{eq: weak adjoint equation}. Let us now introduce the following assumption:

\begin{assumption}
    \label{assum: Nitsche's trick assumption}
    For any $\eta\in \Ls^2\left(Q_T\right)$ and any $y_T\in \Hs^1_0\left(\Omega\right)$, assume that the solution $y\in \Xs$ to Problem \eqref{eq: Nitsche's trick equation} exhibits the regularity $y\in \Ws$. Moreover, it satisfies the smoothing property
    $$
    \norm{y}_{\Ws} \le C\left(\norm{\eta}_{\Ls^2\left(Q_T\right)} + \norm{y_T}_{\Ls^2\left(\Omega\right)}\right),
    $$
    where the constant $C>0$ does not depend on $\eta$ and $y_T$.
\end{assumption}

Similar assumptions were also used in \cite[Theorem 3.15]{LR2013} and \cite[Theorem 2]{LR2017} to derive second-order error estimates at the final time. Likewise, we consider Assumption \ref{assum: Nitsche's trick assumption} to be reasonable; however, we are not aware of any literature explicitly establishing its validity.

\begin{lemma}
    \label{lem: state error estimate L2 T norm}
    Let $u\in \Xs_0$ and $u_h\in \Vs_{h,0}$ be the solutions to Problems \eqref{eq: weak state equation} and \eqref{eq: discrete state equation}, respectively. Assume that $u\in \Ws$ and Assumption \ref{assum: Nitsche's trick assumption} is satisfied. Then, there exists $h_\ast >0$ such that for all $h\in \left(0, h_\ast\right)$, we have
    $$
    \norm{\left(u-u_h\right)\left(\cdot, T\right)}_{\Ls^2\left(\Omega\right)} \le C h^2 \norm{u}_{\Hs^s\left(Q_1\cup Q_2\right)}.
    $$
\end{lemma}

\begin{proof}
    In \eqref{eq: Nitsche's trick equation}, let $\eta = 0$ and $y_T = \norm{\left(u-u_h\right)\left(\cdot, T\right)}^{-1}_{\Ls^2\left(\Omega\right)}\left(u-u_h\right)\left(\cdot,T\right)\in \Hs^1_0\left(\Omega\right)$. Choose $\varsigma = u-u_h \in \Xs_0$. Applying the integration by parts formula for $u-u_h\in \Xs_0$ and $y\in \Xs$, together with the technique in \eqref{eq: optimality conditions 2}, we get
    $$
    \begin{aligned}
        &\norm{\left(u-u_h\right)\left(\cdot, T\right)}_{\Ls^2\left(\Omega\right)} = \\
        &= \int\limits_\Omega \left(u-u_h\right)\left(\xb, 0\right) y  \left(\xb, 0\right)\dx + \inprod{\partial_t \left(u-u_h\right), y  } + \int\limits_0^T\int\limits_\Omega \left[\vb\cdot \nabla \left(u-u_h\right) \right] y   +  \kappa \nabla \left(u-u_h\right) \cdot \nabla  y  \dx \dt\\
        &= a_h\left(u-u_h,  y  \right) + \int\limits_{S_h}\left(\kappa -\kappa_h\right)\nabla \left(u-u_h\right) \cdot \nabla  y  \dx \dt,
    \end{aligned}
    $$
    using the fact that $\kappa -\kappa_h=0$ outside of $S_h$. Since $y \in \Ws$, we are able to invoke the interpolant $I_hy$. We denote $e := y - I_h y$ for convenience. We choose $\vphi_h = I_h y   \in \Vs_{h,0}$ in Lemma \ref{lem: discrepancy between a and ah} to obtain
    \begin{equation}
        \label{eq: state error estimate L2 T norm 1}
        \begin{aligned}
            \norm{\left(u-u_h\right)\left(\cdot, T\right)}_{\Ls^2\left(\Omega\right)} &= a_h\left(u-u_h,  y  \right) + \int\limits_{S_h}\left(\kappa -\kappa_h\right)\nabla \left(u-u_h\right) \cdot \nabla  y  \dx \dt\\
            &= a_h\left(u-u_h,  e\right) + a_h\left(u-u_h, I_h y  \right) +\int\limits_{S_h}\left(\kappa -\kappa_h\right)\nabla \left(u-u_h\right) \cdot \nabla  y  \dx \dt\\
            &= a_h\left(u-u_h,  e \right) + \int\limits_{S_h}\left(\kappa_h -\kappa \right)\nabla u \cdot \nabla \left(I_h y  \right) \dx \dt + \int\limits_{S_h}\left(\kappa -\kappa_h\right)\nabla \left(u-u_h\right) \cdot \nabla  y  \dx \dt\\
            &= a_h\left(u-u_h,  e \right) + \int\limits_{S_h}\left(\kappa -\kappa_h \right)\left[\nabla u \cdot\nabla  e  - \nabla u \cdot \nabla  y + \nabla\left(u-u_h\right)\cdot\nabla  y  \right]\dx\dt.
        \end{aligned}
    \end{equation}
    
    We first estimate the discrete bilinear term $a_h\left(u-u_h,  e \right)$. To do so, we integrate by parts again with $u-u_h\in \Xs_0$ and $e\in \Hs^{1}\left(Q_T\right)$. Subsequently, we apply \eqref{eq: compare L2 and the Y norm} and employ Lemmas \ref{lem: state error estimate discrete norm} and \ref{lem: interpolation error}. It follows that
    $$
    \begin{aligned}
        &a_h\left(u-u_h,  e \right) =\\
        &= \int\limits_\Omega \left(u-u_h\right)\left(\xb, T\right) e \left(\xb, T\right)\dx + \int\limits_0^T\int\limits_\Omega -\left(u-u_h\right)\left(\partial_t  e \right) + \left[\vb\cdot \nabla \left(u-u_h\right) \right] e  +  \kappa_h \nabla \left(u-u_h\right) \cdot \nabla  e  \dx \dt\\
        &\le \norm{\left(u-u_h\right)\left(\cdot, T\right)}_{\Ls^2\left(\Omega\right)}\norm{ e \left(\cdot, T\right)}_{\Ls^2\left(\Omega\right)} + \norm{u-u_h}_{\Ls^2\left(Q_T\right)}\norm{\Ds  e }_{\LLs^2\left(Q_T\right)} + C\norm{\nabla\left(u-u_h\right)}_{\LLs^2\left(Q_T\right)}\left(\norm{ e }_{\Ls^2\left(Q_T\right)} + \norm{\Ds e }_{\LLs^2\left(Q_T\right)}\right)\\
        &\le \norm{\left(u-u_h\right)\left(\cdot, T\right)}_{\Ls^2\left(\Omega\right)}\norm{ e \left(\cdot, T\right)}_{\Ls^2\left(\Omega\right)} + C\norm{\nabla\left(u-u_h\right)}_{\LLs^2\left(Q_T\right)}\left(\norm{ e }_{\Ls^2\left(Q_T\right)} + \norm{\Ds e }_{\LLs^2\left(Q_T\right)}\right)\\
        &\le \norm{\left(u-u_h\right)\left(\cdot, T\right)}_{\Ls^2\left(\Omega\right)}\norm{ e \left(\cdot, T\right)}_{\Ls^2\left(\Omega\right)} + C h \norm{u}_{\Hs^s\left(Q_1\cup Q_2\right)} h \norm{y}_{\Hs^s\left(Q_1\cup Q_2\right)}.
    \end{aligned}
    $$
    By using \eqref{eq: time trace inequality}, \eqref{eq: compare L2 and the Y norm}, and Lemma \ref{lem: interpolation error}, we end up with
    \begin{equation}
        \label{eq: state error estimate L2 T norm 2}
        \begin{aligned}
            \norm{e\left(\cdot,T\right)}^2_{\Ls^2\left(\Omega\right)}\le C \norm{e}^2_{\Xs} \le C\left(\norm{e}^2_{\Ls^2\left(Q_T\right)} + \norm{\Ds e}^2_{\LLs^2\left(Q_T\right)}\right)&\le C\left(\norm{\nabla e}^2_{\LLs^2\left(Q_T\right)} +\norm{\Ds e}^2_{\LLs^2\left(Q_T\right)}\right)\\
            &\le C\norm{\Ds e}^2_{\LLs^2\left(Q_T\right)} \le C \norm{e}^2_{\Hs^1\left(Q_T\right)}\le C h^2 \norm{y}^2_{\Hs^s\left(Q_1\cup Q_2\right)}.
        \end{aligned}
    \end{equation}
    It follows from Assumption \ref{assum: Nitsche's trick assumption} that when $\eta = 0$ and $y_T = \norm{\left(u-u_h\right)\left(\cdot, T\right)}^{-1}_{\Ls^2\left(\Omega\right)}\left(u-u_h\right)\left(\cdot,T\right)$, we have
    \begin{equation}
        \label{eq: state error estimate L2 T norm 5}
        \norm{y}_{\Hs^s\left(Q_1\cup Q_2\right)} \le C\norm{\norm{\left(u-u_h\right)\left(\cdot, T\right)}^{-1}_{\Ls^2\left(\Omega\right)}\left(u-u_h\right)\left(\cdot,T\right)}_{\Ls^2\left(\Omega\right)} = C.
    \end{equation}
    We see that the constant $C>0$ here is $h$-independent. Therefore, we get
    \begin{equation}
        \label{eq: state error estimate L2 T norm 3}
        a_h\left(u-u_h,  e \right) \le Ch \norm{\left(u-u_h\right)\left(\cdot, T\right)}_{\Ls^2\left(\Omega\right)} + C h^2 \norm{u}_{\Hs^s\left(Q_1\cup Q_2\right)}.
    \end{equation}
    
    Next, consider the second integral on the right-hand side of \eqref{eq: state error estimate L2 T norm 1}, denoted by $I_1$ for short. By using the Cauchy-Schwarz inequality, together with Lemmas \ref{lem: interpolation error} and \ref{lem: state error estimate discrete norm}, we bound $I_1$ by
    $$
    \begin{aligned}
        I_1 &:= \int\limits_{S_h}\left(\kappa -\kappa_h \right)\left[\nabla u \cdot\nabla  e  - \nabla u \cdot \nabla y + \nabla\left(u-u_h\right)\cdot\nabla y \right]\dx\dt\\
        &\le C\left(\norm{\nabla u}_{\LLs^2\left(S_h\right)}\norm{\Ds  e }_{\LLs^2\left(Q_T\right)} + \norm{\nabla u}_{\LLs^2\left(S_h\right)}\norm{\nabla y }_{\LLs^2\left(S_h\right)}+ \norm{\nabla \left(u-u_h\right)}_{\LLs^2\left(Q_T\right)}\norm{\nabla y }_{\LLs^2\left(S_h\right)}\right)\\
        &\le C \left(\norm{\nabla u}_{\LLs^2\left(S_h\right)}h \norm{y}_{\Hs^s\left(Q_1\cup Q_2\right)} +\norm{\nabla u}_{\LLs^2\left(S_h\right)}\norm{\nabla y }_{\LLs^2\left(S_h\right)} + h \norm{u}_{\Hs^s\left(Q_1\cup Q_2\right)}\norm{\nabla y}_{\LLs^2\left(S_h\right)}\right).
    \end{aligned}
    $$
    We follow the arguments of \eqref{eq: interpolation error 2} to estimate $\norm{\nabla u}_{\LLs^2\left(S_h\right)}$ and $\norm{\nabla y }_{\LLs^2\left(S_h\right)}$. Take $\norm{\nabla u}_{\LLs^2\left(S_h\right)}$ for instance. Note that $u \in \Ws$. There holds the following inequality
    $$
    \begin{aligned}
        \norm{\nabla u}^2_{\LLs^2\left(S_h\right)} = \sum_{K\in \mathcal{T}_h^\ast} \norm{\nabla u}^2_{\LLs^2\left(K\cap S_h\right)} &\le \sum_{K\in \mathcal{T}_h^\ast}\abs{K\cap S_h} \norm{\nabla u}^2_{\LLs^{\infty}\left(K\cap S_h\right)} \\
        &\le C h^{d+2} \max{\left\{\norm{\nabla\left(\Es_1 u\right)}_{\LLs^{\infty}\left(Q_T\right)}^2, \norm{\nabla\left(\Es_2 u\right)}_{\LLs^{\infty}\left(Q_T\right)}^2\right\}} \left(\sum_{K\in \mathcal{T}_h^\ast} 1\right)\\
        &\le C h^2 \max{\left\{\norm{\nabla\left(\Es_1 u\right)}_{\HHs^{s-1}\left(Q_T\right)}^2, \norm{\nabla\left(\Es_2 u\right)}_{\HHs^{s-1}\left(Q_T\right)}^2\right\}}\\
        &\le C h^2 \norm{u}^2_{\Hs^s\left(Q_1\cup Q_2\right)},
    \end{aligned}
    $$
    where $\Es_1$ and $\Es_2$ are defined in \eqref{eq: extension operator}. This inequality is equivalent to
    $$
    \norm{\nabla u}_{\LLs^2\left(S_h\right)} \le C h\norm{u}_{\Hs^s\left(Q_1\cup Q_2\right)}.
    $$
    Similarly, for $y \in \Ws$, we can prove that
    $$
    \norm{\nabla y }_{\LLs^2\left(S_h\right)} \le C h\norm{y }_{\Hs^s\left(Q_1\cup Q_2\right)}\le Ch,
    $$
    using \eqref{eq: state error estimate L2 T norm 5}. Consequently, we obtain
    \begin{equation}
        \label{eq: state error estimate L2 T norm 4}
        I_1 \le C h^2\norm{u}_{\Hs^s\left(Q_1\cup Q_2\right)}.
    \end{equation}
    
    We substitute \eqref{eq: state error estimate L2 T norm 3} and \eqref{eq: state error estimate L2 T norm 4} into \eqref{eq: state error estimate L2 T norm 1} to arrive at
    $$
    \left(1 - Ch\right)\norm{\left(u-u_h\right)\left(\cdot, T\right)}_{\Ls^2\left(\Omega\right)} \le C h^2 \norm{u}_{\Hs^s\left(Q_1\cup Q_2\right)}.
    $$
    Since $h\in \left(0,h_\ast \right)$ with $h_\ast >0$ be given, it follows that $1-Ch > 1-Ch_\ast$. Therefore, the proof is finished by choosing $h_\ast$ such that $1-Ch_\ast \ge \frac{1}{3}$.
\end{proof}

By using Lemma \ref{lem: state error estimate L2 T norm}, we are now able to estimate the error $\norm{u-u_h}_{\Ls^2\left(Q_T\right)}$, where $u\in \Xs_0$ and $u_h\in \Vs_{h,0}$ are the solutions to Problems \eqref{eq: weak state equation} and \eqref{eq: discrete state equation}, respectively. In the following lemma, we consider Problem \eqref{eq: Nitsche's trick equation} with $\eta = \norm{u-u_h}^{-1}_{\Ls^2\left(Q_T\right)}\left(u-u_h\right) \in \Ls^2\left(Q_T\right)$ and $y_T = 0$.

\begin{lemma}
    \label{lem: state error estimate L2 Q norm}
    Let $u\in \Xs_0$ and $u_h\in \Vs_{h,0}$ be the solutions to Problems \eqref{eq: weak state equation} and \eqref{eq: discrete state equation}, respectively. Assume that $u\in \Ws$ and Assumption \ref{assum: Nitsche's trick assumption} is satisfied. Then, there exists $h_\ast >0$ such that for all $h\in \left(0, h_\ast\right)$, there holds the following estimate
    $$
    \norm{u-u_h}_{\Ls^2\left(Q_T\right)} \le C h^2 \norm{u}_{\Hs^s\left(Q_1\cup Q_2\right)}.
    $$
\end{lemma}

\begin{proof}
    We choose $\varsigma = u-u_h \in \Xs_0$ in \eqref{eq: Nitsche's trick equation}. Then, we apply the integration by parts formula with $y\in \Xs_T$ and $u-u_h\in \Xs_0$, followed by the arguments of \eqref{eq: optimality conditions 2} and Lemma \ref{lem: discrepancy between a and ah}, to arrive at
    \begin{equation}
        \label{eq: state error estimate L2 Q norm 1}
        \begin{aligned}
            \norm{u-u_h}_{\Ls^2\left(Q_T\right)} &=-\inprod{\partial_t y, u-u_h} + \int\limits_0^T\int\limits_\Omega -\left(\vb\cdot \nabla y\right)\left(u-u_h\right) +  \kappa \nabla  y  \cdot \nabla \left(u-u_h\right) \dx \dt\\
            &= a_h\left(u-u_h,  y \right) + \int\limits_{S_h}\left(\kappa-\kappa_h\right)\nabla\left(u-u_h\right)\cdot\nabla  y \dx\dt\\
            &= a_h\left(u-u_h, e\right) + a_h\left(u-u_h, I_h y \right) + \int\limits_{S_h}\left(\kappa-\kappa_h\right)\nabla\left(u-u_h\right)\cdot\nabla  y \dx\dt\\
            &= a_h\left(u-u_h, e\right) + \int\limits_{S_h}\left(\kappa_h-\kappa\right)\left[\nabla u \cdot \nabla \left(I_h  y \right)-\nabla\left(u-u_h\right)\cdot\nabla  y \right]\dx\dt,
        \end{aligned}
    \end{equation}
    where $e = y - I_h y$. For the term $a_h\left(u-u_h, e\right)$, we integrate by parts again with $u-u_h\in \Xs_0$ and $ e \in \Hs^{1}\left(Q_T\right)$. Then, we employ \eqref{eq: compare L2 and the Y norm} and use Lemmas \ref{lem: state error estimate discrete norm} and \ref{lem: interpolation error} to get
    \begin{equation}
        \label{eq: state error estimate L2 Q norm 2}
        \begin{aligned}
            &a_h\left(u-u_h, e\right) = \\
            &= \int\limits_\Omega \left(u-u_h\right)\left(\xb, T\right)e\left(\xb, T\right)\dx + \int\limits_0^T\int\limits_\Omega -\left(u-u_h\right)\left(\partial_t e\right) + \left[\vb\cdot \nabla \left(u-u_h\right) \right]e +  \kappa_h \nabla \left(u-u_h\right) \cdot \nabla e \dx \dt\\
            &\le \int\limits_\Omega \left(u-u_h\right)\left(\xb, T\right)e\left(\xb, T\right)\dx + \norm{u-u_h}_{\Ls^2\left(Q_T\right)}\norm{\Ds e}_{\LLs^2\left(Q_T\right)} + C\norm{\nabla\left(u-u_h\right)}_{\LLs^2\left(Q_T\right)}\left(\norm{e}_{\Ls^2\left(Q_T\right)} + \norm{\Ds e}_{\LLs^2\left(Q_T\right)}\right)\\
            &\le \int\limits_\Omega \left(u-u_h\right)\left(\xb, T\right)e\left(\xb, T\right)\dx + C\norm{\nabla\left(u-u_h\right)}_{\LLs^2\left(Q_T\right)}\left(\norm{e}_{\Ls^2\left(Q_T\right)} + \norm{\Ds e}_{\LLs^2\left(Q_T\right)}\right)\\
            &\le \int\limits_\Omega \left(u-u_h\right)\left(\xb, T\right)e\left(\xb, T\right)\dx + Ch\norm{u}_{\Hs^s\left(Q_1\cup Q_2\right)}h\norm{y}_{\Hs^s\left(Q_1\cup Q_2\right)}.
    \end{aligned}
    \end{equation}
    On the other hand, by following the technique for \eqref{eq: state error estimate L2 T norm 4}, we can prove that
    \begin{equation}
        \label{eq: state error estimate L2 Q norm 3}
        \begin{aligned}
            I_2 &:=\int\limits_{S_h}\left(\kappa_h-\kappa\right)\left[\nabla u \cdot \nabla \left(I_h  y \right)-\nabla\left(u-u_h\right)\cdot\nabla y\right]\dx\dt \\
            &=\int\limits_{S_h}\left(\kappa_h-\kappa\right)\left[\nabla u \cdot \nabla y -\nabla u \cdot \nabla e -\nabla\left(u-u_h\right)\cdot\nabla y\right]\dx\dt\\
            &\le C\left(\norm{\nabla u}_{\LLs^2\left(S_h\right)}\norm{\nabla y}_{\LLs^2\left(S_h\right)} + \norm{\nabla u}_{\LLs^2\left(S_h\right)}\norm{\nabla e}_{\LLs^2\left(Q_T\right)} + \norm{\nabla \left(u-u_h\right)}_{\LLs^2\left(Q_T\right)}\norm{\nabla y}_{\LLs^2\left(S_h\right)}\right)\\
            &\le Ch\norm{u}_{\Hs^s\left(Q_1\cup Q_2\right)}h\norm{y}_{\Hs^s\left(Q_1\cup Q_2\right)}.
        \end{aligned}
    \end{equation}
    By substituting \eqref{eq: state error estimate L2 Q norm 2} and \eqref{eq: state error estimate L2 Q norm 3} into \eqref{eq: state error estimate L2 Q norm 1}, we imply 
    $$
    \begin{aligned}
        \norm{u-u_h}_{\Ls^2\left(Q_T\right)} &\le Ch^2\norm{u}_{\Hs^s\left(Q_1\cup Q_2\right)}\norm{y}_{\Hs^s\left(Q_1\cup Q_2\right)} + \int\limits_\Omega \left(u-u_h\right)\left(\xb, T\right)e\left(\xb, T\right)\dx \\
        &= Ch^2\norm{u}_{\Hs^s\left(Q_1\cup Q_2\right)}\norm{y}_{\Hs^s\left(Q_1\cup Q_2\right)} - \int\limits_\Omega \left(u-u_h\right)\left(\xb,T\right)\left(I_h y \right)\left(\xb,T\right)\dx\\
        &\le Ch^2\norm{u}_{\Hs^s\left(Q_1\cup Q_2\right)}\norm{y}_{\Hs^s\left(Q_1\cup Q_2\right)} + \norm{\left(u-u_h\right)\left(\cdot,T\right)}_{\Ls^2\left(\Omega\right)}\norm{\left(I_h y \right)\left(\cdot,T\right)}_{\Ls^2\left(\Omega\right)}\\
        &\le Ch^2\norm{u}_{\Hs^s\left(Q_1\cup Q_2\right)}\norm{y}_{\Hs^s\left(Q_1\cup Q_2\right)} + Ch^2\norm{u}_{\Hs^s\left(Q_1\cup Q_2\right)}\norm{\left(I_h y \right)\left(\cdot,T\right)}_{\Ls^2\left(\Omega\right)},
    \end{aligned}
    $$
    for all $h\in \left(0,h_\ast\right)$, where $h_\ast>0$ is sufficiently small. Here, we invoked $y\in \Xs_T$ in the second line and Lemma \ref{lem: state error estimate L2 T norm} in the final line. 
    
    The next step is to estimate $\norm{\left(I_h y \right)\left(\cdot,T\right)}_{\Ls^2\left(\Omega\right)}$. By applying the technique in \eqref{eq: state error estimate L2 T norm 2} and the $\Hs^1$-seminorm stability of $I_h$ \cite[Proposition 22.21]{Ern2021}, we obtain 
    $$
    \norm{\left(I_h y \right)\left(\cdot,T\right)}_{\Ls^2\left(\Omega\right)}\le C \norm{I_h y }_{\Xs}\le C\norm{\Ds\left(I_h y \right)}_{\LLs^2\left(Q_T\right)} \le C\norm{\Ds y}_{\LLs^2\left(Q_T\right)}\le C \norm{y}_{\Hs^1\left(Q_T\right)}.
    $$
    When $\eta = \norm{u-u_h}^{-1}_{\Ls^2\left(Q_T\right)}\left(u-u_h\right)$ and $y_T = 0$, Assumption \ref{assum: Nitsche's trick assumption} says that
    $$
    \norm{y}^2_{\Hs^1\left(Q_T\right)} + \norm{y}^2_{\Hs^s\left(Q_1\cup Q_2\right)}\le C\norm{\norm{u-u_h}^{-1}_{\Ls^2\left(Q_T\right)}\left(u-u_h\right)}^2_{\Ls^2\left(Q_T\right)} = C,
    $$
    where $C>0$ is an $h$-independent constant. The proof is completed by combining the last three inequalities.
\end{proof}

\begin{remark}
    \label{rem: error estimate u*}
    Analogous error estimates to those in Lemmas \ref{lem: state error estimate discrete norm}, \ref{lem: state error estimate L2 T norm}, and \ref{lem: state error estimate L2 Q norm} can be proved for Problem \eqref{eq: weak state equation with g}. Specifically, if $u^\ast \in \Ws$ and Assumption \ref{assum: Nitsche's trick assumption} holds, then we obtain similar estimates for the error $u^\ast - u^\ast_h$ with respect to three different norms, where $u^\ast\in \Xs_0$ is defined in \eqref{eq: weak state equation with g}, and $u^\ast_h \in \Vs_{h,0}$ is its finite element approximation.
\end{remark}

Now, we study the stability of the discrete state $u_h\in \Vs_{h,0}$ in \eqref{eq: discrete state equation} with respect to the norms $\norm{\cdot}_{h,0}$ and $\norm{\Ds \cdot}_{\LLs^2\left(Q_T\right)}$, which are fundamental for the subsequent theoretical analyses.

\begin{lemma}
    \label{lem: stability result}
    For $f\in \Ls^2\left(Q_T\right)$, let $u\in \Xs_0$ and $u_h\in \Vs_{h,0}$ be the solutions to Problems \eqref{eq: weak state equation} and \eqref{eq: discrete state equation}, respectively. Let $C_P>0$ and $C_s>0$ be the constants defined in \eqref{eq: compare L2 and the Y norm} and \eqref{eq: stability}, respectively. Then, we have
    $$
    \norm{u_h}_{h,0} \le C_0 \norm{f}_{\Ls^2\left(Q_T\right)},
    $$
    where $C_0:= \left(\norm{\ell}_{\Ls^\infty\left(Q_T\right)}C_P\right)/ \left(C_s\min\left\{\kappa_1, \kappa_2\right\}\right)$. Moreover, if $u\in \Ws$ and Assumption \ref{assum: Nitsche's trick assumption} is satisfied, then there exists $h_\ast >0$ such that for all $h\in \left(0, h_\ast\right)$, the following estimate holds
    $$
    \norm{\Ds u_h}_{\LLs^2\left(Q_T\right)} \le C h\norm{u}_{\Hs^s\left(Q_1\cup Q_2\right)} + \norm{u}_{\Hs^1\left(Q_T\right)}.
    $$
\end{lemma}

\begin{proof}
    For the first inequality, we apply \eqref{eq: stability} and \eqref{eq: compare L2 and the Y norm}. Indeed, we have
    $$
    C_s\norm{u_h}_{h,0} \le \sup_{\vphi_h \in \Vs_{h,0}\setminus \{0\}} \dfrac{a_h\left(u_h,\vphi_h\right)}{\norm{\vphi_h}_{h}} = \sup_{\vphi_h \in \Vs_{h,0}\setminus \{0\}} \dfrac{\left(\ell f, \vphi_h\right)_{\Ls^2\left(Q_T\right)}}{\norm{\vphi_h}_{h}} \le \dfrac{\norm{\ell f}_{\Ls^2\left(Q_T\right)} \norm{\vphi_h}_{\Ls^2\left(Q_T\right)}}{\norm{\vphi_h}_{h}} \le \dfrac{\norm{\ell}_{\Ls^\infty\left(Q_T\right)}\norm{f}_{\Ls^2\left(Q_T\right)} C_P}{\min\left\{\kappa_1, \kappa_2\right\}}.
    $$
    Now, we prove the second inequality. By using the triangle inequality, we obtain
    \begin{equation}
        \label{eq: stability result 3}
        \norm{\Ds u_h}_{\LLs^2\left(Q_T\right)} \le \norm{\Ds\left(u_h-u\right)}_{\LLs^2\left(Q_T\right)} + \norm{\Ds u}_{\LLs^2\left(Q_T\right)} \le \norm{\Ds\left(u_h-u\right)}_{\LLs^2\left(Q_T\right)} + \norm{u}_{\Hs^1\left(Q_T\right)}.
    \end{equation}
    Since $u\in \Ws$, we can use the interpolant $I_hu$. We invoke a global inverse inequality \cite[Example 12.3]{Ern2021} with $u_h -I_h u\in \Vs_{h,0}$, the triangle inequality, and finally, Lemmas \ref{lem: interpolation error} and \ref{lem: state error estimate L2 Q norm}, to arrive at
    $$
    \begin{aligned}
        \norm{\Ds\left(u_h-u\right)}_{\LLs^2\left(Q_T\right)} &\le \norm{\Ds\left(u_h-I_h u\right)}_{\LLs^2\left(Q_T\right)} + \norm{\Ds\left(I_h u-u\right)}_{\LLs^2\left(Q_T\right)}\\
        &\le C h^{-1}\norm{u_h-I_h u}_{\Ls^2\left(Q_T\right)} + Ch \norm{u}_{\Hs^s\left(Q_1\cup Q_2\right)}\\
        &\le C h^{-1}\left(\norm{u_h- u}_{\Ls^2\left(Q_T\right)} + \norm{u -I_h u}_{\Ls^2\left(Q_T\right)}\right) + Ch\norm{u}_{\Hs^s\left(Q_1\cup Q_2\right)}\\
        &\le C h^{-1}\left(h^2\norm{u}_{\Hs^s\left(Q_1\cup Q_2\right)} + h^2\norm{u}_{\Hs^s\left(Q_1\cup Q_2\right)}\right) + Ch\norm{u}_{\Hs^s\left(Q_1\cup Q_2\right)} =  Ch\norm{u}_{\Hs^s\left(Q_1\cup Q_2\right)},
    \end{aligned}
    $$
    for all $h\in \left(0,h_\ast\right)$, where $h_\ast>0$ is sufficiently small. Substituting this into \eqref{eq: stability result 3} completes the proof.
\end{proof}

Under a smoothing property of the adjoint (cf. Assumption \ref{assum: Nitsche's trick assumption}), the convergence order in Lemma \ref{lem: state error estimate discrete norm} has been enhanced. Specifically, second-order error estimates for the state are provided in Lemmas \ref{lem: state error estimate L2 T norm} and \ref{lem: state error estimate L2 Q norm}. Furthermore, conducting a numerical analysis for Problem \eqref{eq: discrete state optimality conditions}--\eqref{eq: discrete variational inequality} necessitates analogous estimates for the adjoint. To derive these results, we first mention the following remark:

\begin{remark}
    \label{rem: regularity of u}
    It is noteworthy that Problem \eqref{eq: general weak problem} can be obtained from Problem \eqref{eq: Nitsche's trick equation} by reversing the directions of time and the velocity field. This enables the application of Assumption \ref{assum: Nitsche's trick assumption} to this problem. More precisely, under Assumption \ref{assum: Nitsche's trick assumption}, for any $F\in \Ls^2\left(Q_T\right)$ and any $U_0\in \Hs^1_0\left(\Omega\right)$, Problem \eqref{eq: general weak problem} admits a solution $U\in \Ws$. Furthermore, there exists a constant $C>0$, independent of $F$ and $U_0$, such that
    $$
    \norm{U}_{\Ws}\le C\left(\norm{F}_{\Ls^2\left(Q_T\right)} + \norm{U_0}_{\Ls^2\left(\Omega\right)}\right).
    $$
\end{remark}

This remark highlights that the conditions in Lemmas \ref{lem: state error estimate L2 T norm}, \ref{lem: state error estimate L2 Q norm}, and \ref{lem: stability result} can be refined to improve both clarity and efficiency. In particular, under Assumption \ref{assum: Nitsche's trick assumption}, the estimates provided in these lemmas remain valid. Therefore, moving forward, we will assume only that Assumption \ref{assum: Nitsche's trick assumption} holds when applying second-order error estimates for the state or the $\norm{\Ds \cdot}_{\LLs^2\left(Q_T\right)}$-norm stability of the discrete state.

Following the duality arguments, we now establish second-order error estimates for the adjoint. For brevity, we state only the key results.

\begin{lemma}
    \label{lem: adjoint error estimate L2 Q norm}
    Let $p\in \Xs_T$ and $p_h\in \Vs_{h, T}$ be the solutions to Problems \eqref{eq: weak adjoint equation} and \eqref{eq: discrete adjoint equation}, respectively. If $p\in \Ws$, then we have the following estimate
    $$
    \norm{p-p_h}_{h,T}\le C h \norm{p}_{\Hs^s\left(Q_1\cup Q_2\right)}.
    $$
    Under Assumption \ref{assum: Nitsche's trick assumption}, there exists $h_\ast >0$ such that for all $h\in \left(0, h_\ast\right)$, the following estimates hold
    $$
    \norm{\left(p-p_h\right)\left(\cdot, 0\right)}_{\Ls^2\left(\Omega\right)} \le C h^2 \norm{p}_{\Hs^s\left(Q_1\cup Q_2\right)},
    $$ 
    and 
    $$
    \norm{p-p_h}_{\Ls^2\left(Q_T\right)} \le C h^2 \norm{p}_{\Hs^s\left(Q_1\cup Q_2\right)}.
    $$
\end{lemma}

\section{Error estimates and convergence rates for the regularized problem}
\label{sec: error estimates}

Let $f_+ \in F_{+}$ be the exact source, $f^{\varepsilon}_{\lambda, h}\in F^h_{+}$ be the solution to Problem \eqref{eq: discrete state optimality conditions}--\eqref{eq: discrete variational inequality}, and $f_h^\dagger\in F^h_{+}$ be the post-processing solution in \eqref{eq: post-processing solution}. In this section, we estimate the errors $f_+ -f^{\varepsilon}_{\lambda, h}$ and $f_+ -f_h^\dagger$ with respect to the $\Ls^2\left(Q_T\right)$-norm in terms of the parameter $\lambda$, the mesh size $h$, and the noise level $\varepsilon$. Moreover, we suggest a priori choices for $\lambda$, depending on $h$ and $\varepsilon$, such that $f^{\varepsilon}_{\lambda, h}$ and $f_h^\dagger$ strongly converge to $f_+ $ in $\Ls^2\left(Q_T\right)$ as both $h\to 0$ and $\varepsilon\to 0$. 

Let us begin with the following inequalities
\begin{equation}
    \label{eq: overall error 1}
    \norm{f_+ -f^{\varepsilon}_{\lambda, h}}_{\Ls^2\left(Q_T\right)}\le \norm{f_+ -f^{\varepsilon}_{\lambda}}_{\Ls^2\left(Q_T\right)} + \norm{f^{\varepsilon}_{\lambda}-f^{\varepsilon}_{\lambda, h}}_{\Ls^2\left(Q_T\right)},
\end{equation}
and 
\begin{equation}
    \label{eq: overall error 2}
    \norm{f_+ -f_h^\dagger}_{\Ls^2\left(Q_T\right)}\le \norm{f_+ -f^{\varepsilon}_{\lambda}}_{\Ls^2\left(Q_T\right)} + \norm{f^{\varepsilon}_{\lambda}-f_h^\dagger}_{\Ls^2\left(Q_T\right)}.
\end{equation}
The first term on the right-hand side of \eqref{eq: overall error 1} and \eqref{eq: overall error 2} is addressed in Lemma \ref{lem: convergence rate}. The following two subsections estimate the second term in \eqref{eq: overall error 1} concerning two different discretization methods: the variational approach and the element-wise constant discretization. Subsection \ref{subsec: error postprocessing} deals with the second term in \eqref{eq: overall error 2}, which is the error of the post-processing strategy. After that, we return to the overall errors and a priori choices for $\lambda$ in Subsection \ref{subsec: regularization parameter selection}. 

\subsection{Variational approach}
\label{subsec: error variational}

Consider the case where $F_{+}$ is not discretized. The main result of this subsection is stated in Theorem \ref{theo: error estimate variational}. For clarity, we divide it into three technical lemmas. 

To commence, let us denote by $u_h\left(f^{\varepsilon}_{\lambda}\right)\in \Vs_{h,0}$ and $p_h\left(f^{\varepsilon}_{\lambda}\right)\in \Vs_{h,T}$ the solutions to the following problems
\begin{equation}
    \label{eq: auxiliary state variational}
    a_h\left(u_h\left(f^{\varepsilon}_{\lambda}\right), \vphi_h\right) = \left(\ell f^{\varepsilon}_{\lambda}, \vphi_h\right)_{\Ls^2\left(Q_T\right)} \qqqq \forall\vphi_h \in \Vs_{h,0},    
\end{equation}
and
\begin{equation}
    \label{eq: auxiliary adjoint variational}
    a^{\prime}_h\left(p_h\left(f^{\varepsilon}_{\lambda}\right),\phi_h\right) = \left(\chi_{\omega_T}\left(u_h\left(f^{\varepsilon}_{\lambda}\right)-z^{\varepsilon}_{d,h}\right),\phi_h\right)_{\Ls^2\left(Q_T\right)} \qqqq \forall \phi_h\in \Vs_{h,T}.
\end{equation}

\begin{lemma}
    \label{lem: discrete-auxiliary estimate variational}
    Let $\left(u_h\left(f^{\varepsilon}_{\lambda}\right), p_h\left(f^{\varepsilon}_{\lambda}\right), f^{\varepsilon}_{\lambda}\right) \in \Vs_{h,0}\times \Vs_{h,T}\times F_{+}$ and $\left(u^{\varepsilon}_{\lambda, h}, p^{\varepsilon}_{\lambda, h}, f^{\varepsilon}_{\lambda, h}\right)\in \Vs_{h,0}\times \Vs_{h,T}\times F^h_{+}$ be the solutions to Problems \eqref{eq: auxiliary state variational}--\eqref{eq: auxiliary adjoint variational}, \eqref{eq: variational inequality} and \eqref{eq: discrete state optimality conditions}--\eqref{eq: discrete variational inequality} in case of variational discretization, respectively. Let $C_0>0$ be the constant defined in Lemma \ref{lem: stability result}. Then, we have the following inequalities
    $$
    \norm{u_h\left(f^{\varepsilon}_{\lambda}\right)-u^{\varepsilon}_{\lambda, h}}_{h,0} +\norm{p_h\left(f^{\varepsilon}_{\lambda}\right)-p^{\varepsilon}_{\lambda, h}}_{h,T} \le C_1\norm{f^{\varepsilon}_{\lambda}-f^{\varepsilon}_{\lambda, h}}_{\Ls^2\left(Q_T\right)},
    $$
    and 
    $$
    \norm{u_h\left(f^{\varepsilon}_{\lambda}\right)-u^{\varepsilon}_{\lambda, h}}_{\Ls^2\left(Q_T\right)} +\norm{p_h\left(f^{\varepsilon}_{\lambda}\right)-p^{\varepsilon}_{\lambda, h}}_{\Ls^2\left(Q_T\right)} \le C_2\norm{f^{\varepsilon}_{\lambda}-f^{\varepsilon}_{\lambda, h}}_{\Ls^2\left(Q_T\right)},
    $$
    where
    $$
    C_1:= \max{\left\{C_0,\dfrac{C^2_P C_0}{C_s\min^2\left\{\kappa_1, \kappa_2\right\}}\right\}}\qqq\text{and}\qqq C_2:=\dfrac{C_P C_1}{\min\left\{\kappa_1, \kappa_2\right\}}.
    $$
\end{lemma}

\begin{proof}
    The first result is a direct application of Lemma \ref{lem: stability result}. Indeed, we subtract \eqref{eq: discrete state optimality conditions} from \eqref{eq: auxiliary state variational} to obtain
    $$
    a_h\left(u_h\left(f^{\varepsilon}_{\lambda}\right) - u^{\varepsilon}_{\lambda, h}, \vphi_h\right) = \left(\ell f^{\varepsilon}_{\lambda}-\ell f^{\varepsilon}_{\lambda, h}, \vphi_h\right)_{\Ls^2\left(Q_T\right)} \qqqq \forall\vphi_h \in \Vs_{h,0}.
    $$
    Due to Lemma \ref{lem: stability result}, we have
    \begin{equation}
        \label{eq: discrete-auxiliary estimate variational 1}
        \norm{u_h\left(f^{\varepsilon}_{\lambda}\right) - u^{\varepsilon}_{\lambda, h}}_{h,0}\le C_0\norm{f^{\varepsilon}_{\lambda}-f^{\varepsilon}_{\lambda, h}}_{\Ls^2\left(Q_T\right)}.
    \end{equation}
    We continue by subtracting \eqref{eq: discrete adjoint optimality conditions} from \eqref{eq: auxiliary adjoint variational} to get
    $$
    a^{\prime}_h\left(p_h\left(f^{\varepsilon}_{\lambda}\right) - p^{\varepsilon}_{\lambda, h},\phi_h\right) = \left(\chi_{\omega_T}\left(u_h\left(f^{\varepsilon}_{\lambda}\right) - u^{\varepsilon}_{\lambda, h}\right),\phi_h\right)_{\Ls^2\left(Q_T\right)} \qqqq \forall \phi_h\in \Vs_{h,T}.
    $$
    We invoke \eqref{eq: stability adjoint}, the technique in Lemma \ref{lem: stability result}, \eqref{eq: compare L2 and the Y norm}, and \eqref{eq: discrete-auxiliary estimate variational 1} to obtain
    \begin{equation}
        \label{eq: discrete-auxiliary estimate variational 3}
        \begin{aligned}
            \norm{p_h\left(f^{\varepsilon}_{\lambda}\right) - p^{\varepsilon}_{\lambda, h}}_{h,T} \le \dfrac{C_P}{C_s\min\left\{\kappa_1, \kappa_2\right\}}\norm{u_h\left(f^{\varepsilon}_{\lambda}\right) - u^{\varepsilon}_{\lambda, h}}_{\Ls^2\left(\omega_T\right)} &\le \dfrac{C^2_P}{C_s\min^2\left\{\kappa_1, \kappa_2\right\}}\norm{u_h\left(f^{\varepsilon}_{\lambda}\right) - u^{\varepsilon}_{\lambda, h}}_{h,0} \\
            &\le \dfrac{C^2_P C_0}{C_s\min^2\left\{\kappa_1, \kappa_2\right\}}\norm{f^{\varepsilon}_{\lambda}-f^{\varepsilon}_{\lambda, h}}_{\Ls^2\left(Q_T\right)}.
        \end{aligned}
    \end{equation}
    The first result follows by combining \eqref{eq: discrete-auxiliary estimate variational 1} and \eqref{eq: discrete-auxiliary estimate variational 3}. The second one is a consequence of the first, due to \eqref{eq: compare L2 and the Y norm}.
\end{proof}

\begin{lemma}
    \label{lem: exact-discrete control estimate variational}
    Let $f^{\varepsilon}_{\lambda}\in F_{+}$ be the solution to Problem \eqref{eq: problem formulation} and $p^\varepsilon_\lambda\in \Xs_T$ be its corresponding adjoint. Let $p_h\left(f^{\varepsilon}_{\lambda}\right)\in \Vs_{h,T}$ and $f^{\varepsilon}_{\lambda, h}\in F^h_{+}$ be the solutions to Problems \eqref{eq: auxiliary adjoint variational} and \eqref{eq: discrete variational inequality} in case of variational discretization, respectively. Then, we have
    $$
    \norm{f^{\varepsilon}_{\lambda} - f^{\varepsilon}_{\lambda, h}}_{\Ls^2\left(Q_T\right)}\le \dfrac{C}{\lambda^\tau}\norm{p^{\varepsilon}_{\lambda}-p_h\left(f^{\varepsilon}_{\lambda}\right)}_{\Ls^2\left(Q_T\right)} \qqqq \text{with } \tau = 
    \begin{cases}
        \frac{3}{2} & \text{if } \lambda \le 8 \norm{\ell}_{\Ls^\infty\left(Q_T\right)}C_2, \\ 
        0 & \text{otherwise},
    \end{cases}
    $$
    where the constant $C_2>0$ is defined in Lemma \ref{lem: discrete-auxiliary estimate variational}.
\end{lemma}

\begin{proof}
    We choose $f=f^{\varepsilon}_{\lambda, h}\in F_{+}$ in \eqref{eq: variational inequality} and $f_h=f^{\varepsilon}_{\lambda}\in F^h_{+}$ in \eqref{eq: discrete variational inequality}, then add the corresponding inequalities to get
    \begin{equation}
        \label{eq: exact-discrete control estimate variational 1}
        \begin{aligned}
            \lambda \norm{f^{\varepsilon}_{\lambda}-f^{\varepsilon}_{\lambda, h}}_{\Ls^2\left(Q_T\right)}^2 &\le \left(\ell p^{\varepsilon}_{\lambda}-\ell p^{\varepsilon}_{\lambda, h}, f^{\varepsilon}_{\lambda, h} - f^{\varepsilon}_{\lambda}\right)_{\Ls^2\left(Q_T\right)} \\
            &= \left(\ell p^{\varepsilon}_{\lambda}-\ell p_h\left(f^{\varepsilon}_{\lambda}\right), f^{\varepsilon}_{\lambda, h} - f^{\varepsilon}_{\lambda}\right)_{\Ls^2\left(Q_T\right)} + \left(\ell p_h\left(f^{\varepsilon}_{\lambda}\right)-\ell p^{\varepsilon}_{\lambda, h}, f^{\varepsilon}_{\lambda, h} - f^{\varepsilon}_{\lambda}\right)_{\Ls^2\left(Q_T\right)}.
        \end{aligned}
    \end{equation}
    By using the Cauchy inequality, we estimate the first term on the right-hand side of \eqref{eq: exact-discrete control estimate variational 1} as follows
    \begin{equation}
        \label{eq: exact-discrete control estimate variational 2}
        \begin{aligned}
            \left(\ell p^{\varepsilon}_{\lambda}-\ell p_h\left(f^{\varepsilon}_{\lambda}\right), f^{\varepsilon}_{\lambda, h} - f^{\varepsilon}_{\lambda}\right)_{\Ls^2\left(Q_T\right)}&\le \norm{\ell}_{\Ls^\infty\left(Q_T\right)}\norm{p^{\varepsilon}_{\lambda}-p_h\left(f^{\varepsilon}_{\lambda}\right)}_{\Ls^2\left(Q_T\right)}\norm{f^{\varepsilon}_{\lambda, h} - f^{\varepsilon}_{\lambda}}_{\Ls^2\left(Q_T\right)} \\
            &\le \dfrac{1}{2\lambda}\norm{\ell}^2_{\Ls^\infty\left(Q_T\right)}\norm{p^{\varepsilon}_{\lambda}-p_h\left(f^{\varepsilon}_{\lambda}\right)}_{\Ls^2\left(Q_T\right)}^2 + \dfrac{\lambda}{2}\norm{f^{\varepsilon}_{\lambda, h} - f^{\varepsilon}_{\lambda}}_{\Ls^2\left(Q_T\right)}^2.
        \end{aligned}
    \end{equation}
    For the second term on the right-hand side of \eqref{eq: exact-discrete control estimate variational 1}, we apply Lemma \ref{lem: discrete-auxiliary estimate variational} to have
    \begin{equation}
        \label{eq: exact-discrete control estimate variational 3}
        \begin{aligned}
            \left(\ell p_h\left(f^{\varepsilon}_{\lambda}\right)-\ell p^{\varepsilon}_{\lambda, h}, f^{\varepsilon}_{\lambda, h} - f^{\varepsilon}_{\lambda}\right)_{\Ls^2\left(Q_T\right)} &\le \norm{\ell}_{\Ls^\infty\left(Q_T\right)}\norm{p_h\left(f^{\varepsilon}_{\lambda}\right)- p^{\varepsilon}_{\lambda, h}}_{\Ls^2\left(Q_T\right)}\norm{f^{\varepsilon}_{\lambda, h} - f^{\varepsilon}_{\lambda}}_{\Ls^2\left(Q_T\right)}\\
            &\le \norm{\ell}_{\Ls^\infty\left(Q_T\right)}C_2\norm{f^{\varepsilon}_{\lambda, h} - f^{\varepsilon}_{\lambda}}^2_{\Ls^2\left(Q_T\right)}.
        \end{aligned}
    \end{equation}
    By combining \eqref{eq: exact-discrete control estimate variational 1}, \eqref{eq: exact-discrete control estimate variational 2}, and \eqref{eq: exact-discrete control estimate variational 3}, we obtain
    $$
    \left(\dfrac{\lambda}{2} - \norm{\ell}_{\Ls^\infty\left(Q_T\right)}C_2\right)\norm{f^{\varepsilon}_{\lambda} - f^{\varepsilon}_{\lambda, h}}^2_{\Ls^2\left(Q_T\right)} \le \dfrac{1}{2\lambda}\norm{\ell}^2_{\Ls^\infty\left(Q_T\right)}\norm{p^{\varepsilon}_{\lambda}-p_h\left(f^{\varepsilon}_{\lambda}\right)}_{\Ls^2\left(Q_T\right)}^2.
    $$
    If $\lambda \le 8 \norm{\ell}_{\Ls^\infty\left(Q_T\right)}C_2$, it follows that $\frac{\lambda}{2} - \norm{\ell}_{\Ls^\infty\left(Q_T\right)}C_2 \ge \frac{3\lambda^2}{64 \norm{\ell}_{\Ls^\infty\left(Q_T\right)}C_2}$, which implies
    $$
    \norm{f^{\varepsilon}_{\lambda} - f^{\varepsilon}_{\lambda, h}}_{\Ls^2\left(Q_T\right)} \le \sqrt{\dfrac{32 \norm{\ell}^3_{\Ls^\infty\left(Q_T\right)}C_2}{3\lambda^3}}\norm{p^{\varepsilon}_{\lambda}-p_h\left(f^{\varepsilon}_{\lambda}\right)}_{\Ls^2\left(Q_T\right)}.
    $$
    On the other hand, if $\lambda > 8 \norm{\ell}_{\Ls^\infty\left(Q_T\right)}C_2$, then $\frac{\lambda}{2} - \norm{\ell}_{\Ls^\infty\left(Q_T\right)}C_2 > 3 \norm{\ell}_{\Ls^\infty\left(Q_T\right)}C_2$. Therefore, we arrive at
    $$
    \norm{f^{\varepsilon}_{\lambda} - f^{\varepsilon}_{\lambda, h}}_{\Ls^2\left(Q_T\right)} < \sqrt{\dfrac{\norm{\ell}_{\Ls^\infty\left(Q_T\right)}}{6\lambda C_2}}\norm{p^{\varepsilon}_{\lambda}-p_h\left(f^{\varepsilon}_{\lambda}\right)}_{\Ls^2\left(Q_T\right)} < \dfrac{1}{4\sqrt{3}C_2} \norm{p^{\varepsilon}_{\lambda}-p_h\left(f^{\varepsilon}_{\lambda}\right)}_{\Ls^2\left(Q_T\right)}.
    $$
    The proof is complete.
\end{proof}

We next estimate the right-hand side of the inequality in Lemma \ref{lem: exact-discrete control estimate variational}. To do so, let us introduce $\widetilde{p_h}\left(f^{\varepsilon}_{\lambda}\right)\in \Vs_{h,T}$ as the solution to the problem
\begin{equation}
    \label{eq: auxiliary adjoint equation with continuous state variational}
    a^{\prime}_h\left(\widetilde{p_h}\left(f^{\varepsilon}_{\lambda}\right),\phi_h\right) = \left(\chi_{\omega_T}\left(u^{\varepsilon}_{\lambda}-z^{\varepsilon}_d\right),\phi_h\right)_{\Ls^2\left(Q_T\right)} \qqqq \forall \phi_h\in \Vs_{h,T}.
\end{equation}

We also need an $\Ls^2\left(\omega_T\right)$-norm estimate of the noisy data $z^\varepsilon_d \in \Ls^2\left(\omega_T\right)$. By using the triangle inequality, \eqref{eq: priori estimate}, and \eqref{eq: noise level}, we get
\begin{equation}
    \label{eq: stability of the noisy data}
    \begin{aligned}
        \norm{z^\varepsilon_d}_{\Ls^2\left(\omega_T\right)} \le \norm{z_d}_{\Ls^2\left(\omega_T\right)} + \norm{z^\varepsilon_d - z_d}_{\Ls^2\left(\omega_T\right)}  &\le \norm{U_d}_{\Ls^2\left(\omega_T\right)} + \norm{u^\ast}_{\Ls^2\left(\omega_T\right)} + \norm{U^\varepsilon_d - U_d}_{\Ls^2\left(\omega_T\right)} \\
        &\le \norm{U_d}_{\Ls^2\left(\omega_T\right)} +C\norm{g}_{\Ls^2\left(Q_T\right)} + \varepsilon \\
        &\le C\left(\norm{U_d}_{\Ls^2\left(\omega_T\right)} +\norm{g}_{\Ls^2\left(Q_T\right)}+\varepsilon\right).
    \end{aligned}
\end{equation}
Consequently, we conclude that the solution $f^\varepsilon_\lambda\in F_+$ to Problem \eqref{eq: problem formulation} satisfies the following inequality
\begin{equation}
    \label{eq: stability of the source}
    \begin{aligned}
        \norm{f^{\varepsilon}_{\lambda}}_{\Ls^2\left(Q_T\right)} \le \sqrt{\frac{2}{\lambda}J^{\varepsilon}_{\lambda}\left(f^{\varepsilon}_{\lambda}\right)}\le \sqrt{\frac{2}{\lambda}J^{\varepsilon}_{\lambda}\left(0\right)} &= \dfrac{1}{\sqrt{\lambda}}\norm{z^\varepsilon_d}_{\Ls^2\left(\omega_T\right)}\\
        &\le \dfrac{C}{\sqrt{\lambda}}\left(\norm{U_d}_{\Ls^2\left(\omega_T\right)} +\norm{g}_{\Ls^2\left(Q_T\right)}+\varepsilon\right).
    \end{aligned}
\end{equation}

\begin{lemma}
    \label{lem: auxiliary-discrete adjoint estimate variational}
    Let $p^{\varepsilon}_{\lambda}\in \Xs_T$ and $p_h\left(f^{\varepsilon}_{\lambda}\right)\in \Vs_{h,T}$ be the solutions to Problems \eqref{eq: weak adjoint equation optimality conditions} and \eqref{eq: auxiliary adjoint variational}, respectively. Under Assumption \ref{assum: Nitsche's trick assumption}, we have the following estimates
    $$
    \norm{p^{\varepsilon}_{\lambda}-p_h\left(f^{\varepsilon}_{\lambda}\right)}_{h,T} \le Ch\left(1+\dfrac{1}{\sqrt{\lambda}}\right)\left(\norm{U_d}_{\Ls^2\left(\omega_T\right)} +\norm{g}_{\Ls^2\left(Q_T\right)}+\varepsilon\right),
    $$
    and
    $$
    \norm{p^{\varepsilon}_{\lambda}-p_h\left(f^{\varepsilon}_{\lambda}\right)}_{\Ls^2\left(Q_T\right)} \le Ch^2\left(1+\dfrac{1}{\sqrt{\lambda}}\right)\left(\norm{U_d}_{\Ls^2\left(\omega_T\right)} +\norm{g}_{\Ls^2\left(Q_T\right)}+\varepsilon\right).
    $$
\end{lemma}

\begin{proof}
    To prove the first estimate, we begin with the triangle inequality
    \begin{equation}
        \label{eq: auxiliary-discrete adjoint estimate variational 1}
        \norm{p^{\varepsilon}_{\lambda}-p_h\left(f^{\varepsilon}_{\lambda}\right)}_{h,T}\le \norm{p^{\varepsilon}_{\lambda}-\widetilde{p_h}\left(f^{\varepsilon}_{\lambda}\right)}_{h,T} + \norm{\widetilde{p_h}\left(f^{\varepsilon}_{\lambda}\right)-p_h\left(f^{\varepsilon}_{\lambda}\right)}_{h,T}.
    \end{equation}
    Thanks to Lemma \ref{lem: adjoint error estimate L2 Q norm}, there holds the following estimate
    $$
    \norm{p^{\varepsilon}_{\lambda}-\widetilde{p_h}\left(f^{\varepsilon}_{\lambda}\right)}_{h,T}\le Ch \norm{p^{\varepsilon}_{\lambda}}_{\Hs^s\left(Q_1\cup Q_2\right)}.
    $$
    By using Assumption \ref{assum: Nitsche's trick assumption}, \eqref{eq: priori estimate}, \eqref{eq: stability of the noisy data}, and \eqref{eq: stability of the source}, we obtain
    \begin{equation}
        \label{eq: auxiliary-discrete adjoint estimate variational 4}
        \begin{aligned}
            \norm{p^{\varepsilon}_{\lambda}}_{\Hs^s\left(Q_1\cup Q_2\right)}\le C\left(\norm{u^{\varepsilon}_{\lambda}}_{\Ls^2\left(\omega_T\right)} + \norm{z^\varepsilon_d}_{\Ls^2\left(\omega_T\right)}\right)&\le C\left(\norm{f^{\varepsilon}_{\lambda}}_{\Ls^2\left(Q_T\right)} + \norm{z^\varepsilon_d}_{\Ls^2\left(\omega_T\right)}\right)\\
            &\le C\left(1+\dfrac{1}{\sqrt{\lambda}}\right)\left(\norm{U_d}_{\Ls^2\left(\omega_T\right)} +\norm{g}_{\Ls^2\left(Q_T\right)}+\varepsilon\right).
        \end{aligned}
    \end{equation}
    Thus, we imply
    \begin{equation}
        \label{eq: auxiliary-discrete adjoint estimate variational 2}
        \norm{p^{\varepsilon}_{\lambda}-\widetilde{p_h}\left(f^{\varepsilon}_{\lambda}\right)}_{h,T}\le Ch\left(1+\dfrac{1}{\sqrt{\lambda}}\right)\left(\norm{U_d}_{\Ls^2\left(\omega_T\right)} +\norm{g}_{\Ls^2\left(Q_T\right)}+\varepsilon\right).
    \end{equation}
    For the second term on the right-hand side of \eqref{eq: auxiliary-discrete adjoint estimate variational 1}, we subtract \eqref{eq: auxiliary adjoint variational} from \eqref{eq: auxiliary adjoint equation with continuous state variational} to obtain
    $$
    a^{\prime}_h\left(\widetilde{p_h}\left(f^{\varepsilon}_{\lambda}\right)-p_h\left(f^{\varepsilon}_{\lambda}\right),\phi_h\right) = \left(\chi_{\omega_T}\left(u^{\varepsilon}_{\lambda}-u_h\left(f^{\varepsilon}_{\lambda}\right)  - z^\varepsilon_d + z^\varepsilon_{d,h}\right),\phi_h\right)_{\Ls^2\left(Q_T\right)} \qqqq \forall \phi_h\in \Vs_{h,T}.
    $$
    By invoking the technique in \eqref{eq: discrete-auxiliary estimate variational 3}, Lemma \ref{lem: state error estimate L2 Q norm}, and subsequently Remarks \ref{rem: error estimate u*} and \ref{rem: regularity of u}, we have
    $$
    \begin{aligned}
        \norm{\widetilde{p_h}\left(f^{\varepsilon}_{\lambda}\right)-p_h\left(f^{\varepsilon}_{\lambda}\right)}_{h,T}\le C \left(\norm{u^{\varepsilon}_{\lambda} - u_h\left(f^{\varepsilon}_{\lambda}\right)}_{\Ls^2\left(\omega_T\right)} + \norm{u^\ast - u^\ast_h}_{\Ls^2\left(\omega_T\right)}\right) &\le Ch^2\left(\norm{u^{\varepsilon}_{\lambda}}_{\Hs^s\left(Q_1\cup Q_2\right)} + \norm{u^\ast}_{\Hs^s\left(Q_1\cup Q_2\right)}\right)\\
        &\le Ch^2\left(\norm{f^{\varepsilon}_{\lambda}}_{\Ls^2\left(Q_T\right)} + \norm{g}_{\Ls^2\left(Q_T\right)}\right).
    \end{aligned}
    $$
    Together with \eqref{eq: stability of the source}, we arrive at
    \begin{equation}
        \label{eq: auxiliary-discrete adjoint estimate variational 3}
        \norm{\widetilde{p_h}\left(f^{\varepsilon}_{\lambda}\right)-p_h\left(f^{\varepsilon}_{\lambda}\right)}_{h,T}\le \dfrac{Ch^2}{\sqrt{\lambda}}\left(\norm{U_d}_{\Ls^2\left(\omega_T\right)}+\varepsilon\right)+ Ch^2\left(1+\dfrac{1}{\sqrt{\lambda}}\right)\norm{g}_{\Ls^2\left(Q_T\right)}.
    \end{equation}
    The first result follows by combining \eqref{eq: auxiliary-discrete adjoint estimate variational 1}, \eqref{eq: auxiliary-discrete adjoint estimate variational 2}, and \eqref{eq: auxiliary-discrete adjoint estimate variational 3}. 
    
    Next, we prove the second result. We also start with the triangle inequality
    $$
    \norm{p^{\varepsilon}_{\lambda}-p_h\left(f^{\varepsilon}_{\lambda}\right)}_{\Ls^2\left(Q_T\right)}\le \norm{p^{\varepsilon}_{\lambda}-\widetilde{p_h}\left(f^{\varepsilon}_{\lambda}\right)}_{\Ls^2\left(Q_T\right)} + \norm{\widetilde{p_h}\left(f^{\varepsilon}_{\lambda}\right)-p_h\left(f^{\varepsilon}_{\lambda}\right)}_{\Ls^2\left(Q_T\right)}.
    $$
    On the first hand, Lemma \ref{lem: adjoint error estimate L2 Q norm} and \eqref{eq: auxiliary-discrete adjoint estimate variational 4} give us
    $$
    \norm{p^{\varepsilon}_{\lambda}-\widetilde{p_h}\left(f^{\varepsilon}_{\lambda}\right)}_{\Ls^2\left(Q_T\right)} \le Ch^2\norm{p^{\varepsilon}_{\lambda}}_{\Hs^s\left(Q_1\cup Q_2\right)}\le Ch^2\left(1+\dfrac{1}{\sqrt{\lambda}}\right)\left(\norm{U_d}_{\Ls^2\left(\omega_T\right)} +\norm{g}_{\Ls^2\left(Q_T\right)}+\varepsilon\right).
    $$
    On the other hand, by applying \eqref{eq: compare L2 and the Y norm} and \eqref{eq: auxiliary-discrete adjoint estimate variational 3}, we obtain
    $$
    \norm{\widetilde{p_h}\left(f^{\varepsilon}_{\lambda}\right)-p_h\left(f^{\varepsilon}_{\lambda}\right)}_{\Ls^2\left(Q_T\right)} \le C\norm{\widetilde{p_h}\left(f^{\varepsilon}_{\lambda}\right)-p_h\left(f^{\varepsilon}_{\lambda}\right)}_{h,T} \le \dfrac{Ch^2}{\sqrt{\lambda}}\left(\norm{U_d}_{\Ls^2\left(\omega_T\right)}+\varepsilon\right)+ Ch^2\left(1+\dfrac{1}{\sqrt{\lambda}}\right)\norm{g}_{\Ls^2\left(Q_T\right)}.
    $$
    The proof is complete.
\end{proof}

We now arrive at the main result of this subsection by combining Lemmas \ref{lem: discrete-auxiliary estimate variational}, \ref{lem: exact-discrete control estimate variational}, and \ref{lem: auxiliary-discrete adjoint estimate variational} with the triangle inequality.

\begin{theorem}
    \label{theo: error estimate variational}
    Let $\left(u^{\varepsilon}_{\lambda},p^{\varepsilon}_{\lambda},f^{\varepsilon}_{\lambda}\right)\in \Xs_0\times \Xs_T\times F_{+}$ and $\left(u^{\varepsilon}_{\lambda, h},p^{\varepsilon}_{\lambda, h},f^{\varepsilon}_{\lambda, h}\right)\in \Vs_{h,0}\times \Vs_{h,T}\times F^h_{+}$ be the solutions to Problems \eqref{eq: weak state equation optimality conditions}--\eqref{eq: variational inequality} and \eqref{eq: discrete state optimality conditions}--\eqref{eq: discrete variational inequality} in case of variational discretization, respectively. Let $C_2>0$ be the constant defined in Lemma \ref{lem: discrete-auxiliary estimate variational}. Assume that Assumption \ref{assum: Nitsche's trick assumption} is satisfied. Then, the following estimates hold
    $$
    \norm{u^{\varepsilon}_{\lambda} - u^{\varepsilon}_{\lambda, h}}_{h,0} + \norm{p^{\varepsilon}_{\lambda} - p^{\varepsilon}_{\lambda, h}}_{h,T} + \norm{f^{\varepsilon}_{\lambda}-f^{\varepsilon}_{\lambda, h}}_{\Ls^2\left(Q_T\right)} \le Ch\left(1+ \dfrac{1}{\lambda^\tau}\right)\left(\norm{U_d}_{\Ls^2\left(\omega_T\right)} +\norm{g}_{\Ls^2\left(Q_T\right)}+\varepsilon\right) = \mathcal{O}\left(h\right),
    $$
    and
    $$
    \norm{u^{\varepsilon}_{\lambda} - u^{\varepsilon}_{\lambda, h}}_{\Ls^2\left(Q_T\right)} + \norm{p^{\varepsilon}_{\lambda} - p^{\varepsilon}_{\lambda, h}}_{\Ls^2\left(Q_T\right)} + \norm{f^{\varepsilon}_{\lambda}-f^{\varepsilon}_{\lambda, h}}_{\Ls^2\left(Q_T\right)} \le Ch^2\left(1+ \dfrac{1}{\lambda^\tau}\right)\left(\norm{U_d}_{\Ls^2\left(\omega_T\right)} +\norm{g}_{\Ls^2\left(Q_T\right)}+\varepsilon\right) = \mathcal{O}\left(h^2\right),
    $$
    where
    $$
    \tau = 
    \begin{cases}
        2 & \text{if } \lambda \le 8 \norm{\ell}_{\Ls^\infty\left(Q_T\right)}C_2, \\ 
        \frac{1}{2} & \text{otherwise}.
    \end{cases}
    $$
\end{theorem}

\begin{proof}
    For the first estimate, we use the triangle inequality and apply Lemmas \ref{lem: state error estimate discrete norm}, \ref{lem: discrete-auxiliary estimate variational}, and \ref{lem: exact-discrete control estimate variational}. It follows that
    $$
    \begin{aligned}
        &\norm{u^{\varepsilon}_{\lambda} - u^{\varepsilon}_{\lambda, h}}_{h,0} + \norm{p^{\varepsilon}_{\lambda} - p^{\varepsilon}_{\lambda, h}}_{h,T} + \norm{f^{\varepsilon}_{\lambda}-f^{\varepsilon}_{\lambda, h}}_{\Ls^2\left(Q_T\right)} \le\\
        &\le \norm{u^{\varepsilon}_{\lambda} - u_h\left(f^\varepsilon_\lambda\right)}_{h,0} + \norm{u_h\left(f^\varepsilon_\lambda\right) - u^{\varepsilon}_{\lambda, h}}_{h,0} + \norm{p^{\varepsilon}_{\lambda} - p_h\left(f^\varepsilon_\lambda\right)}_{h,T} + \norm{p_h\left(f^\varepsilon_\lambda\right) - p^{\varepsilon}_{\lambda, h}}_{h,T} + \norm{f^{\varepsilon}_{\lambda}-f^{\varepsilon}_{\lambda, h}}_{\Ls^2\left(Q_T\right)}\\
        &\le Ch\norm{u^\varepsilon_\lambda}_{\Hs^s\left(Q_1\cup Q_2\right)} + C\norm{f^\varepsilon_\lambda - f^\varepsilon_{\lambda, h}}_{\Ls^2\left(Q_T\right)} + \norm{p^{\varepsilon}_{\lambda} - p_h\left(f^\varepsilon_\lambda\right)}_{h,T}\\
        &\le Ch\norm{u^\varepsilon_\lambda}_{\Hs^s\left(Q_1\cup Q_2\right)} + C\left(1+\dfrac{1}{\lambda^\tau}\right)\norm{p^{\varepsilon}_{\lambda} - p_h\left(f^\varepsilon_\lambda\right)}_{h,T}\qqqq\q\text{with } \tau = 
        \begin{cases}
            \frac{3}{2} & \text{if } \lambda \le 8 \norm{\ell}_{\Ls^\infty\left(Q_T\right)}C_2, \\ 
            0 & \text{otherwise}.
        \end{cases}
    \end{aligned}
    $$
    By using Remark \ref{rem: regularity of u} and \eqref{eq: stability of the source}, we have
    $$
    \norm{u^\varepsilon_\lambda}_{\Hs^s\left(Q_1\cup Q_2\right)} \le C \norm{f^{\varepsilon}_{\lambda}}_{\Ls^2\left(Q_T\right)} \le \dfrac{C}{\sqrt{\lambda}}\left(\norm{U_d}_{\Ls^2\left(\omega_T\right)} +\norm{g}_{\Ls^2\left(Q_T\right)}+\varepsilon\right).
    $$
    Together with Lemma \ref{lem: auxiliary-discrete adjoint estimate variational}, we hence obtain
    $$
    \begin{aligned}
        &\norm{u^{\varepsilon}_{\lambda} - u^{\varepsilon}_{\lambda, h}}_{h,0} + \norm{p^{\varepsilon}_{\lambda} - p^{\varepsilon}_{\lambda, h}}_{h,T} + \norm{f^{\varepsilon}_{\lambda}-f^{\varepsilon}_{\lambda, h}}_{\Ls^2\left(Q_T\right)} \le\\
        &\le Ch\left(1+\dfrac{1}{\sqrt{\lambda}} + \dfrac{1}{\lambda^\tau} + \dfrac{1}{\lambda^{\frac{1}{2}+\tau}}\right)\left(\norm{U_d}_{\Ls^2\left(\omega_T\right)} +\norm{g}_{\Ls^2\left(Q_T\right)}+\varepsilon\right) \\
        &\le Ch\left(1+ \dfrac{1}{\lambda^{\frac{1}{2}+\tau}}\right)\left(\norm{U_d}_{\Ls^2\left(\omega_T\right)} +\norm{g}_{\Ls^2\left(Q_T\right)}+\varepsilon\right) \qqqq\q\text{with } \tau = 
        \begin{cases}
            \frac{3}{2} & \text{if } \lambda \le 8 \norm{\ell}_{\Ls^\infty\left(Q_T\right)}C_2, \\ 
            0 & \text{otherwise}.
        \end{cases}
    \end{aligned}
    $$
    The conclusion follows. The second inequality is proved similarly. 
\end{proof}

\subsection{Element-wise constant discretization}
\label{subsec: error constant}

We choose $F^d$ as the space of step functions, as described in \eqref{eq: step functions space}. First, we discretize the regularized source while leaving the advection-diffusion equations unchanged. Mathematically, we consider a pure state problem: For $f_h\in F^h_{+}= F^d\cap F_{+}$, determine $u\left(f_h\right)\in \Xs_0$ such that
\begin{equation}
    \label{eq: weak pure state equation constant}
    a\left(u\left(f_h\right), \vphi\right) = \left(\ell f_h, \vphi\right)_{\Ls^2\left(Q_T\right)}\qqqq \forall\vphi \in \Ys.
\end{equation}
The well-posedness of this problem is proved in Subsection \ref{subsec: the weak advection-diffusion problem}. We then introduce the following problem
\begin{equation}
    \label{eq: pure functional constant}
    \begin{aligned}
        & \min\limits_{f_h \in F^h_{+}} J^\varepsilon_\lambda\left(f_h\right)=\dfrac{1}{2} \norm{u\left(f_h\right)-z^{\varepsilon}_{d}}^2_{\Ls^2\left(\omega_T\right)} +\dfrac{\lambda}{2}\norm{f_h}^2_{\Ls^2\left(Q_T\right)}, \\
        & \text{ subject to \eqref{eq: weak pure state equation constant}}.
    \end{aligned}
\end{equation}
Clearly, Problem \eqref{eq: pure functional constant} admits a unique solution $\widehat{f_h} \in F^h_{+}$. Furthermore, the following result holds:

\begin{lemma}
    Let $\widehat{f}_h\in F^h_{+}$ be the solution to Problem \eqref{eq: pure functional constant}, and let $\widehat{u}\in \Xs_0$ and $\widehat{p}\in \Xs_T$ be its corresponding state and adjoint, respectively. Then, the following system is satisfied
    \begin{equation}
        \label{eq: weak pure state equation constant optimality conditions}
        a\left(\widehat{u}, \vphi\right) = \left(\ell \widehat{f}_h, \vphi\right)_{\Ls^2\left(Q_T\right)}\qqqq \forall\vphi \in \Ys,
    \end{equation}
    and
    \begin{equation}
        \label{eq: weak pure adjoint equation constant optimality conditions}
        a^\prime\left(\widehat{p},\phi\right) = \left(\chi_{\omega_T}\left(\widehat{u}-z^{\varepsilon}_{d}\right),\phi\right)_{\Ls^2\left(Q_T\right)} \qqqq \forall \phi\in \Ys,
    \end{equation}
    and the variational inequality
    \begin{equation}
        \label{eq: pure variational inequality constant optimality conditions}
        \left(\ell \widehat{p} + \lambda \widehat{f}_h, f_h- \widehat{f}_h\right)_{\Ls^2\left(Q_T\right)} \ge 0 \qqqq \forall f_h\in F^h_{+}.
    \end{equation}
\end{lemma}

Regarding an error analysis, we now discuss the regularity of the regularized source $f^{\varepsilon}_{\lambda}\in F_{+}$ in \eqref{eq: problem formulation}, along with its stability estimate. Specifically, we have the following lemma:

\begin{lemma}
    \label{lem: control regularity}
    Let $f^{\varepsilon}_{\lambda}\in F_{+}$ be the solution to Problem \eqref{eq: problem formulation} and $p^\varepsilon_\lambda\in \Xs_T$ be its corresponding adjoint. Assume that Assumption \ref{assum: Nitsche's trick assumption} holds and $\ell \in \Ws^{1,\infty}\left(Q_T\right)$. Then, $\ell p^{\varepsilon}_{\lambda} \in \Ws^{1,\infty}\left(Q_T\right)$, and consequently, $f^{\varepsilon}_{\lambda} \in \Ws^{1,\infty}\left(Q_T\right)$. Moreover, the following stability estimate holds
    $$
    \norm{\Ds f^\varepsilon_\lambda}_{\LLs^\infty\left(Q_T\right)}\le \dfrac{1}{\lambda} \norm{\Ds \left(\ell p^\varepsilon_\lambda\right)}_{\LLs^\infty\left(Q_T\right)}.
    $$
\end{lemma}

\begin{proof}
    For $p^{\varepsilon}_{\lambda}\in \Hs^s\left(Q_1\cup Q_2\right)$ with $s>\frac{d+3}{2}$, the Sobolev embeddings $\Hs^{s}\left(Q_i\right)\hookrightarrow \Ws^{1,\infty}\left(Q_i\right)\ (i=1,2)$ \cite[Theorem 2.31]{Ern2021} give us $p^{\varepsilon}_{\lambda}\in \Ws^{1,\infty}\left(Q_1\cup Q_2\right)$. On the other hand, since $p^\varepsilon_\lambda\in \Hs^{1}\left(Q_T\right)$, we imply $\gamma_1p^{\varepsilon}_{\lambda}-\gamma_2 p^{\varepsilon}_{\lambda}=0$ by \cite[Theorem 18.8]{Ern2021}. Here, $\gamma_1$ and $\gamma_2$ are the trace operators introduced in Subsection \ref{subsec: functional setting}. By invoking \cite[Theorem 18.8]{Ern2021} once more, we obtain $p^{\varepsilon}_{\lambda}\in \Ws^{1,\infty}\left(Q_T\right)$. Combine this with $\ell \in \Ws^{1,\infty}\left(Q_T\right)$, we conclude that $\ell p^{\varepsilon}_{\lambda}$ is an element of $\Ws^{1,\infty}\left(Q_T\right)$. 
    
    The regularity of $f^\varepsilon_\lambda$ follows from \eqref{eq: projection formula of f} and Stampacchia's lemma \cite[Theorem A.1]{KS1980}. Together with Ziemer's representation \cite[Corollary 2.1.8]{Ziemer1989}, we get the desired stability estimate.
\end{proof}

Denote by $\Pi_d :\Ls^2\left(Q_T\right)\rightarrow F^d$ the $\Ls^2$-orthogonal projection operator, defined element-wise as
\begin{equation}
    \label{eq: L2 orthogonal projection}
    \left(\Pi_d v \right)_{\mid \, K} := \dfrac{1}{\abs{K}}\int\limits_{K}v \dx\dt\qqqq \forall K\in \mathcal{T}_h \in \Ls^2\left(Q_T\right),
\end{equation}
for all $v  \in \Ls^2\left(Q_T\right)$. Recall from \cite[Section 18.4]{Ern2021} that for all $v  \in \Hs^1\left(\Lambda\right)$, we have
\begin{equation}
    \label{eq: error estimates for projection operators ver 2}
    \norm{v  -\Pi_d v }_{\Ls^2\left(\Lambda\right)} \le C h\norm{\Ds v }_{\LLs^2\left(\Lambda\right)},
\end{equation}
and for all $v  \in \Ws^{1,\infty}\left(\Lambda\right)$, the following inequality holds
\begin{equation}
    \label{eq: error estimates for projection operators}
    \norm{v  -\Pi_d v }_{\Ls^\infty\left(\Lambda\right)} \le C h\norm{\Ds v }_{\LLs^\infty\left(\Lambda\right)}.
\end{equation}
Here, $\Lambda$ denotes an arbitrary element $K\in \mathcal{T}_h$ or the entire space-time domain $Q_T$, and the constant $C>0$ does not depend on $v $ and $h$.

\begin{lemma}
    \label{lem: pure control estimate constant}
    Let $f^{\varepsilon}_{\lambda}\in F_{+}$ be the solution to Problem \eqref{eq: problem formulation} and $p^\varepsilon_\lambda\in \Xs_T$ be its corresponding adjoint. Let $\widehat{f_h}\in F^h_{+}$ be the solution to Problem \eqref{eq: pure functional constant}. Assume that Assumption \ref{assum: Nitsche's trick assumption} is satisfied and $\ell \in \Ws^{1,\infty}\left(Q_T\right)$. Then, the following estimate holds
    $$
    \norm{f^{\varepsilon}_{\lambda}-\widehat{f_h}}_{\Ls^2\left(Q_T\right)} \le \dfrac{Ch}{\lambda} \sqrt{\left(1+\dfrac{1}{\sqrt{\lambda}}\right)\left(\norm{U_d}_{\Ls^2\left(\omega_T\right)} +\norm{g}_{\Ls^2\left(Q_T\right)}+\varepsilon\right)}\norm{\Ds \left(\ell p^{\varepsilon}_{\lambda}\right)}^{\frac{1}{2}}_{\LLs^\infty\left(Q_T\right)}. 
    $$
\end{lemma}

\begin{proof}
    By inserting $f=\widehat{f_h}\in F_{+}$ into \eqref{eq: variational inequality}, we find
    $$
    \left(\ell p^{\varepsilon}_{\lambda}+\lambda f^{\varepsilon}_{\lambda},\widehat{f_h}-f^{\varepsilon}_{\lambda}\right)_{\Ls^2\left(Q_T\right)} \geq 0.
    $$
    On the other hand, note that $\Pi_d  f^{\varepsilon}_{\lambda}\in F^h_{+}$. We choose $f_h=\Pi_d  f^{\varepsilon}_{\lambda}\in F^h_{+}$ in \eqref{eq: pure variational inequality constant optimality conditions} to get
    $$
    \left(\ell \widehat{p}+\lambda \widehat{f_h},\Pi_d  f^{\varepsilon}_{\lambda}-\widehat{f_h}\right)_{\Ls^2\left(Q_T\right)} \geq 0,
    $$
    which is equivalent to
    $$
    \left(\ell \widehat{p}+\lambda \widehat{f_h}, f^{\varepsilon}_{\lambda}-\widehat{f_h}\right)_{\Ls^2\left(Q_T\right)}+\left(\ell \widehat{p}+\lambda \widehat{f_h},\Pi_d  f^{\varepsilon}_{\lambda}-f^{\varepsilon}_{\lambda}\right)_{\Ls^2\left(Q_T\right)} \geq 0.
    $$
    We add this inequality to the first one to obtain
    \begin{equation}
        \label{eq: pure control estimate constant 1}
        \lambda\norm{f^{\varepsilon}_{\lambda}-\widehat{f_h}}_{\Ls^2\left(Q_T\right)}^2\le \left(\ell p^{\varepsilon}_{\lambda}-\ell \widehat{p},\widehat{f_h} - f^{\varepsilon}_{\lambda}\right)_{\Ls^2\left(Q_T\right)} +\left(\ell \widehat{p}+\lambda \widehat{f_h},\Pi_d  f^{\varepsilon}_{\lambda} - f^{\varepsilon}_{\lambda}\right)_{\Ls^2\left(Q_T\right)}.
    \end{equation}
    To deal with the first term on the right-hand side of \eqref{eq: pure control estimate constant 1}, we invoke \eqref{eq: optimality conditions 3} twice. It follows that
    $$
    \begin{aligned}
        \left(\ell p^{\varepsilon}_{\lambda}-\ell \widehat{p}, \widehat{f_h} - f^{\varepsilon}_{\lambda}\right)_{\Ls^2\left(Q_T\right)} &= \left(\ell p^{\varepsilon}_{\lambda}, \widehat{f_h} - f^{\varepsilon}_{\lambda}\right)_{\Ls^2\left(Q_T\right)} - \left(\ell \widehat{p}, \widehat{f_h} - f^{\varepsilon}_{\lambda}\right)_{\Ls^2\left(Q_T\right)} \\
        &= \left(\widehat{u}-u^{\varepsilon}_{\lambda},u^{\varepsilon}_{\lambda}-z^{\varepsilon}_{d}\right)_{\Ls^2\left(\omega_T\right)} - \left(\widehat{u}-u^{\varepsilon}_{\lambda},\widehat{u}-z^{\varepsilon}_{d}\right)_{\Ls^2\left(\omega_T\right)} \\
        &= - \norm{\widehat{u}-u^{\varepsilon}_{\lambda}}^2_{\Ls^2\left(\omega_T\right)} \le 0.
    \end{aligned}   
    $$
    In the second term on the right-hand side of \eqref{eq: pure control estimate constant 1}, since $\widehat{f_h}$ and $\Pi_d  f^{\varepsilon}_{\lambda} - f^{\varepsilon}_{\lambda}$ belong to two orthogonal subspaces of $\Ls^2\left(Q_T\right)$, we have
    $$
    \left(\widehat{f_h}, \Pi_d  f^{\varepsilon}_{\lambda} - f^{\varepsilon}_{\lambda}\right)_{\Ls^2\left(Q_T\right)} = 0.
    $$
    Therefore, \eqref{eq: pure control estimate constant 1} reduces to 
    $$
    \lambda\norm{f^{\varepsilon}_{\lambda}-\widehat{f_h}}_{\Ls^2\left(Q_T\right)}^2\le \left(\ell \widehat{p},\Pi_d  f^{\varepsilon}_{\lambda} - f^{\varepsilon}_{\lambda}\right)_{\Ls^2\left(Q_T\right)}.
    $$
    By using the orthogonal arguments once more, \eqref{eq: error estimates for projection operators ver 2}, \eqref{eq: error estimates for projection operators}, and Lemma \ref{lem: control regularity}, we have
    $$
    \begin{aligned}
        \left(\ell \widehat{p},\Pi_d  f^{\varepsilon}_{\lambda} - f^{\varepsilon}_{\lambda}\right)_{\Ls^2\left(Q_T\right)} = \left(\ell\widehat{p}-\Pi_d \left(\ell\widehat{p}\right), \Pi_d  f^{\varepsilon}_{\lambda} - f^{\varepsilon}_{\lambda}\right)_{\Ls^2\left(Q_T\right)} &\le \norm{\ell \widehat{p}-\Pi_d  \left(\ell\widehat{p}\right)}_{\Ls^2\left(Q_T\right)} \norm{\Pi_d  f^{\varepsilon}_{\lambda} - f^{\varepsilon}_{\lambda}}_{\Ls^2\left(Q_T\right)} \\
        &\le \norm{\ell\widehat{p}-\Pi_d  \left(\ell\widehat{p}\right)}_{\Ls^2\left(Q_T\right)} \abs{Q_T}^{\frac{1}{2}}\norm{\Pi_d  f^{\varepsilon}_{\lambda} - f^{\varepsilon}_{\lambda}}_{\Ls^\infty\left(Q_T\right)}\\
        &\le Ch \norm{\Ds \left(\ell\widehat{p}\right)}_{\LLs^2\left(Q_T\right)} h\norm{\Ds f^{\varepsilon}_{\lambda}}_{\LLs^\infty\left(Q_T\right)}\\
        &\le Ch \norm{\Ds \left(\ell\widehat{p}\right)}_{\LLs^2\left(Q_T\right)} \dfrac{h}{\lambda}\norm{\Ds \left(\ell p^{\varepsilon}_{\lambda}\right)}_{\LLs^\infty\left(Q_T\right)}.
    \end{aligned}
    $$
    Therefore, we imply
    \begin{equation}
        \label{eq: pure control estimate constant 3}
        \norm{f^{\varepsilon}_{\lambda}-\widehat{f_h}}_{\Ls^2\left(Q_T\right)}\le \dfrac{Ch}{\lambda} \norm{\Ds \left(\ell\widehat{p}\right)}^{\frac{1}{2}}_{\LLs^2\left(Q_T\right)} \norm{\Ds \left(\ell p^{\varepsilon}_{\lambda}\right)}^{\frac{1}{2}}_{\LLs^\infty\left(Q_T\right)}.
    \end{equation}
    Finally, let us proceed with the first norm on the right-hand side of \eqref{eq: pure control estimate constant 3}. We invoke $\ell\in \Ws^{1,\infty}\left(Q_T\right)$, Assumption \ref{assum: Nitsche's trick assumption}, \eqref{eq: priori estimate}, and the arguments of \eqref{eq: auxiliary-discrete adjoint estimate variational 4} to get
    $$
    \begin{aligned}
        \norm{\Ds \left(\ell\widehat{p}\right)}_{\LLs^2\left(Q_T\right)} \le C\norm{\widehat{p}}_{\Hs^1\left(Q_T\right)}\le C\left(\norm{\widehat{u}}_{\Ls^2\left(\omega_T\right)} + \norm{z^\varepsilon_d}_{\Ls^2\left(\omega_T\right)}\right)&\le C\left(\norm{\widehat{f_h}}_{\Ls^2\left(Q_T\right)}+ \norm{z^\varepsilon_d}_{\Ls^2\left(\omega_T\right)}\right)\\
        &\le C\left(1+\dfrac{1}{\sqrt{\lambda}}\right)\left(\norm{U_d}_{\Ls^2\left(\omega_T\right)} +\norm{g}_{\Ls^2\left(Q_T\right)}+\varepsilon\right).
    \end{aligned}
    $$
   We obtain the desired estimate by substituting this into \eqref{eq: pure control estimate constant 3}. 
\end{proof}

Next, we employ the space-time interface-fitted method \cite{NLPT2024} to discretize both Problems \eqref{eq: weak pure state equation constant optimality conditions} and \eqref{eq: weak pure adjoint equation constant optimality conditions}. More precisely, let us introduce $u_h\left(\widehat{f_h}\right) \in \Vs_{h,0}$ and $ p_h\left(\widehat{f_h}\right) \in \Vs_{h,T}$ as the solutions to the following problems
\begin{equation}
    \label{eq: auxiliary discrete for pure state constant}
    a_h\left(u_h\left(\widehat{f_h}\right), \vphi_h\right) = \left(\ell\widehat{f_h}, \vphi_h\right)_{\Ls^2\left(Q_T\right)} \qqqq \forall\vphi_h \in \Vs_{h,0},
\end{equation}
and
\begin{equation}
    \label{eq: auxiliary discrete for pure adjoint constant 2}
    a^{\prime}_h\left( p_h\left(\widehat{f_h}\right),\phi_h\right) = \left(\chi_{\omega_T}\left(\widehat{u}-z^{\varepsilon}_d\right),\phi_h\right)_{\Ls^2\left(Q_T\right)} \qqqq \forall \phi_h\in \Vs_{h,T}.
\end{equation}
Furthermore, denote by $\widehat{p_h}\left(\widehat{f_h}\right)\in \Vs_{h,T}$ the solution to the problem
\begin{equation}
    \label{eq: auxiliary discrete for pure adjoint constant 1}
    a^{\prime}_h\left(\widehat{p_h}\left(\widehat{f_h}\right),\phi_h\right) = \left(\chi_{\omega_T}\left(u_h\left(\widehat{f_h}\right)-z^{\varepsilon}_{d,h}\right),\phi_h\right)_{\Ls^2\left(Q_T\right)} \qqqq \forall \phi_h\in \Vs_{h,T}.
\end{equation}

\begin{lemma}
    \label{lem: discrete the pure control constant}
    Let $\widehat{f_h}\in F^h_{+}$ and $f^{\varepsilon}_{\lambda, h}\in F^h_{+}$ be the solutions to Problems \eqref{eq: pure functional constant} and \eqref{eq: discrete variational inequality} in case of element-wise constant discretization, respectively. Let $C_2>0$ be the constant defined in Lemma \ref{lem: discrete-auxiliary estimate variational}. If Assumption \ref{assum: Nitsche's trick assumption} holds, then we have the following estimate
    $$
    \norm{f^{\varepsilon}_{\lambda, h}-\widehat{f_h}}_{\Ls^2\left(Q_T\right)} \le \dfrac{Ch^2}{\lambda^\tau}\left(1+\dfrac{1}{\sqrt{\lambda}}\right)\left(\norm{U_d}_{\Ls^2\left(\omega_T\right)} +\norm{g}_{\Ls^2\left(Q_T\right)}+\varepsilon\right) \qqqq \text{with } \tau = 
    \begin{cases}
        2 & \text{if } \lambda \le 4 \norm{\ell}_{\Ls^\infty\left(Q_T\right)}C_2, \\ 
        0 & \text{otherwise}.
    \end{cases}
    $$
\end{lemma}

\begin{proof}
    We begin with inserting $f_h=\widehat{f_h}\in F^h_{+}$ into \eqref{eq: discrete variational inequality} and $f_h= f^{\varepsilon}_{\lambda, h}\in F^h_{+}$ into \eqref{eq: pure variational inequality constant optimality conditions}, respectively. It follows that
    $$
    \left(\ell p^{\varepsilon}_{\lambda, h}+\lambda f^{\varepsilon}_{\lambda, h}, \widehat{f_h}-f^{\varepsilon}_{\lambda, h}\right)_{\Ls^2\left(Q_T\right)} \geq 0\qqqq\text{and}\qqqq\left(\ell\widehat{p}+\lambda \widehat{f_h}, f^{\varepsilon}_{\lambda, h}-\widehat{f_h}\right)_{\Ls^2\left(Q_T\right)} \geq 0.
    $$
    We add these two inequalities to obtain 
    \begin{equation}
        \label{eq: discrete the pure control constant 1}
        \lambda\norm{f^{\varepsilon}_{\lambda, h}-\widehat{f_h}}_{\Ls^2\left(Q_T\right)}^2\le \left(\ell\widehat{p}-\ell p^{\varepsilon}_{\lambda, h},\ f^{\varepsilon}_{\lambda, h} - \widehat{f_h}\right)_{\Ls^2\left(Q_T\right)} = J_1 + J_2 + J_3,
    \end{equation}
    where $J_1, J_2$, and $J_3$ are defined as
    $$
    J_1 := \left(\ell\widehat{p}-\ell  p_h\left(\widehat{f_h}\right),\ f^{\varepsilon}_{\lambda, h} - \widehat{f_h}\right)_{\Ls^2\left(Q_T\right)}\qqq \text{and}\qqq J_2 := \left(\ell  p_h\left(\widehat{f_h}\right)-\ell \widehat{p_h}\left(\widehat{f_h}\right),\ f^{\varepsilon}_{\lambda, h} - \widehat{f_h}\right)_{\Ls^2\left(Q_T\right)},
    $$
    and 
    $$
    J_3 := \left(\ell \widehat{p_h}\left(\widehat{f_h}\right)-\ell p^{\varepsilon}_{\lambda, h},\ f^{\varepsilon}_{\lambda, h} - \widehat{f_h}\right)_{\Ls^2\left(Q_T\right)}.
    $$
    We use Lemma \ref{lem: adjoint error estimate L2 Q norm} and the technique in \eqref{eq: auxiliary-discrete adjoint estimate variational 4} to bound $J_1$ as follows
    \begin{equation}
        \label{eq: discrete the pure control constant 2}
        \begin{aligned}
            J_1 \le C\norm{\widehat{p}- p_h\left(\widehat{f_h}\right)}_{\Ls^2\left(Q_T\right)}\norm{f^{\varepsilon}_{\lambda, h} - \widehat{f_h}}_{\Ls^2\left(Q_T\right)}&\le Ch^2\norm{\widehat{p}}_{\Hs^s\left(Q_1\cup Q_2\right)}\norm{f^{\varepsilon}_{\lambda, h} - \widehat{f_h}}_{\Ls^2\left(Q_T\right)}\\
            &\le Ch^2\left(1+\dfrac{1}{\sqrt{\lambda}}\right)\left(\norm{U_d}_{\Ls^2\left(\omega_T\right)} +\norm{g}_{\Ls^2\left(Q_T\right)}+\varepsilon\right)\norm{f^{\varepsilon}_{\lambda, h} - \widehat{f_h}}_{\Ls^2\left(Q_T\right)}.
        \end{aligned}
    \end{equation}
    For handling $J_2$, we subtract \eqref{eq: auxiliary discrete for pure adjoint constant 1} from \eqref{eq: auxiliary discrete for pure adjoint constant 2} to arrive at
    $$
    a^{\prime}_h\left( p_h\left(\widehat{f_h}\right)-\widehat{p_h}\left(\widehat{f_h}\right),\phi_h\right) = \left(\chi_{\omega_T}\left(\widehat{u}- u_h\left(\widehat{f_h}\right) - z^\varepsilon_d + z^\varepsilon_{d,h}\right),\phi_h\right)_{\Ls^2\left(Q_T\right)} \qqqq \forall \phi_h\in \Vs_{h,T}.
    $$
    We invoke \eqref{eq: compare L2 and the Y norm} and the arguments of \eqref{eq: auxiliary-discrete adjoint estimate variational 3} to get
    $$
    \begin{aligned}
        \norm{p_h\left(\widehat{f_h}\right)-\widehat{p_h}\left(\widehat{f_h}\right)}_{\Ls^2\left(Q_T\right)}\le C\norm{p_h\left(\widehat{f_h}\right)-\widehat{p_h}\left(\widehat{f_h}\right)}_{h,T} &\le C\left(\norm{\widehat{u}- u_h\left(\widehat{f_h}\right)}_{\Ls^2\left(\omega_T\right)} + \norm{u^\ast - u^\ast_h}_{\Ls^2\left(\omega_T\right)}\right)\\
        &\le Ch^2\left(\norm{\widehat{u}}_{\Hs^s\left(Q_1\cup Q_2\right)} + \norm{u^\ast}_{\Hs^s\left(Q_1\cup Q_2\right)}\right)\\
        &\le \dfrac{Ch^2}{\sqrt{\lambda}}\left(\norm{U_d}_{\Ls^2\left(\omega_T\right)}+\varepsilon\right)+ Ch^2\left(1+\dfrac{1}{\sqrt{\lambda}}\right)\norm{g}_{\Ls^2\left(Q_T\right)},
    \end{aligned}
    $$
    which leads to
    \begin{equation}
        \label{eq: discrete the pure control constant 3}
        \begin{aligned}
            J_2&\le C\norm{ p_h\left(\widehat{f_h}\right)-\widehat{p_h}\left(\widehat{f_h}\right)}_{\Ls^2\left(Q_T\right)}\norm{f^{\varepsilon}_{\lambda, h} - \widehat{f_h}}_{\Ls^2\left(Q_T\right)}\\
            &\le C\left[\dfrac{h^2}{\sqrt{\lambda}}\left(\norm{U_d}_{\Ls^2\left(\omega_T\right)}+\varepsilon\right)+ h^2\left(1+\dfrac{1}{\sqrt{\lambda}}\right)\norm{g}_{\Ls^2\left(Q_T\right)}\right]\norm{f^{\varepsilon}_{\lambda, h} - \widehat{f_h}}_{\Ls^2\left(Q_T\right)}.
        \end{aligned}
    \end{equation}
    Finally, we estimate $J_3$ by employing the technique in \eqref{eq: exact-discrete control estimate variational 3}. We have
    $$
    \begin{aligned}
        J_3=\left(\ell \widehat{p_h}\left(\widehat{f_h}\right)-\ell p^{\varepsilon}_{\lambda, h},f^{\varepsilon}_{\lambda, h} - \widehat{f_h}\right)_{\Ls^2\left(Q_T\right)} &\le \norm{\ell}_{\Ls^\infty\left(Q_T\right)} \norm{\widehat{p_h}\left(\widehat{f_h}\right)- p^{\varepsilon}_{\lambda, h}}_{\Ls^2\left(Q_T\right)}\norm{f^{\varepsilon}_{\lambda, h} - \widehat{f_h}}_{\Ls^2\left(Q_T\right)}\\
        &\le \norm{\ell}_{\Ls^\infty\left(Q_T\right)}C_2 \norm{f^{\varepsilon}_{\lambda, h}-\widehat{f_h}}_{\Ls^2\left(Q_T\right)}^2.
    \end{aligned}
    $$
    We insert \eqref{eq: discrete the pure control constant 2}, \eqref{eq: discrete the pure control constant 3}, and this into \eqref{eq: discrete the pure control constant 1} to deduce that
    $$
    \left(\lambda - \norm{\ell}_{\Ls^\infty\left(Q_T\right)}C_2\right)\norm{f^{\varepsilon}_{\lambda, h}-\widehat{f_h}}_{\Ls^2\left(Q_T\right)}^2 \le Ch^2\left(1+\dfrac{1}{\sqrt{\lambda}}\right)\left(\norm{U_d}_{\Ls^2\left(\omega_T\right)} +\norm{g}_{\Ls^2\left(Q_T\right)}+\varepsilon\right)\norm{f^{\varepsilon}_{\lambda, h} - \widehat{f_h}}_{\Ls^2\left(Q_T\right)}.
    $$
    Together with the inequality
    $$
    \lambda - \norm{\ell}_{\Ls^\infty\left(Q_T\right)}C_2 \ge
    \begin{cases}
        \dfrac{3\lambda^2}{16 \norm{\ell}_{\Ls^\infty\left(Q_T\right)}C_2} &\text{if }\lambda \le 4 \norm{\ell}_{\Ls^\infty\left(Q_T\right)}C_2,\\
        3 \norm{\ell}_{\Ls^\infty\left(Q_T\right)}C_2 & \text{if } \lambda \ge 4 \norm{\ell}_{\Ls^\infty\left(Q_T\right)}C_2,
    \end{cases}
    $$
    we hence derive the desired result.
\end{proof}

We are now able to present the main result of this subsection. By combining Lemmas \ref{lem: pure control estimate constant} and \ref{lem: discrete the pure control constant} with the triangle inequality, we obtain the following result:

\begin{theorem}
    \label{theo: error estimate constant}
    Let $\left(u^{\varepsilon}_{\lambda},p^{\varepsilon}_{\lambda},f^{\varepsilon}_{\lambda}\right)\in \Xs_0\times \Xs_T\times F_{+}$ and $\left(u^{\varepsilon}_{\lambda, h},p^{\varepsilon}_{\lambda, h},f^{\varepsilon}_{\lambda, h}\right)\in \Vs_{h,0}\times \Vs_{h,T}\times F^h_{+}$ be the solutions to Problems \eqref{eq: weak state equation optimality conditions}--\eqref{eq: variational inequality} and \eqref{eq: discrete state optimality conditions}--\eqref{eq: discrete variational inequality} in case of element-wise constant discretization, respectively. Let $C_2>0$ be the constant defined in Lemma \ref{lem: discrete-auxiliary estimate variational}. Assume that Assumption \ref{assum: Nitsche's trick assumption} holds and $\ell \in \Ws^{1,\infty}\left(Q_T\right)$. Then, we have
    $$
    \begin{aligned}
        &\norm{u^{\varepsilon}_{\lambda} - u^{\varepsilon}_{\lambda, h}}_{h,0} + \norm{p^{\varepsilon}_{\lambda} - p^{\varepsilon}_{\lambda, h}}_{h,T} + \norm{f^{\varepsilon}_{\lambda}-f^{\varepsilon}_{\lambda, h}}_{\Ls^2\left(Q_T\right)} \le\\
        &\le \dfrac{Ch}{\lambda} \sqrt{\left(1+\dfrac{1}{\sqrt{\lambda}}\right)\left(\norm{U_d}_{\Ls^2\left(\omega_T\right)} +\norm{g}_{\Ls^2\left(Q_T\right)}+\varepsilon\right)}\norm{\Ds \left(\ell p^{\varepsilon}_{\lambda}\right)}^{\frac{1}{2}}_{\LLs^\infty\left(Q_T\right)} + Ch\left(1+\dfrac{1}{\lambda^\tau}\right)\left(\norm{U_d}_{\Ls^2\left(\omega_T\right)} +\norm{g}_{\Ls^2\left(Q_T\right)}+\varepsilon\right),
    \end{aligned}
    $$
    and
    $$
    \begin{aligned}
        &\norm{u^{\varepsilon}_{\lambda} - u^{\varepsilon}_{\lambda, h}}_{\Ls^2\left(Q_T\right)} + \norm{p^{\varepsilon}_{\lambda} - p^{\varepsilon}_{\lambda, h}}_{\Ls^2\left(Q_T\right)} + \norm{f^{\varepsilon}_{\lambda}-f^{\varepsilon}_{\lambda, h}}_{\Ls^2\left(Q_T\right)} \le \\
        &\le \dfrac{Ch}{\lambda} \sqrt{\left(1+\dfrac{1}{\sqrt{\lambda}}\right)\left(\norm{U_d}_{\Ls^2\left(\omega_T\right)} +\norm{g}_{\Ls^2\left(Q_T\right)}+\varepsilon\right)}\norm{\Ds \left(\ell p^{\varepsilon}_{\lambda}\right)}^{\frac{1}{2}}_{\LLs^\infty\left(Q_T\right)} + Ch^2\left(1+\dfrac{1}{\lambda^\tau}\right)\left(\norm{U_d}_{\Ls^2\left(\omega_T\right)} +\norm{g}_{\Ls^2\left(Q_T\right)}+\varepsilon\right).
    \end{aligned}
    $$
    In both estimates, the convergence order is linear, i.e., of $\mathcal{O}\left(h\right)$, and the exponent $\tau$ is given by
    $$
    \tau = 
    \begin{cases}
        \frac{5}{2} & \text{if } \lambda \le 4 \norm{\ell}_{\Ls^\infty\left(Q_T\right)}C_2, \\ 
        \frac{1}{2} & \text{otherwise}.
    \end{cases}
    $$
\end{theorem}

\begin{proof}
    We begin with the first estimate. For the state error, we apply the triangle inequality to obtain 
    \begin{equation}
        \label{eq: error estimate constant 1}
        \norm{u^{\varepsilon}_{\lambda} - u^{\varepsilon}_{\lambda, h}}_{h,0}\le \norm{u^{\varepsilon}_{\lambda} - \widehat{u}}_{h,0} + \norm{\widehat{u} - u_h\left(\widehat{f_h}\right)}_{h, 0} + \norm{u_h\left(\widehat{f_h}\right) - u^{\varepsilon}_{\lambda, h}}_{h,0}.
    \end{equation}
    The first term on the right-hand side of \eqref{eq: error estimate constant 1} can be bounded by using \eqref{eq: compare L2 and the X norm} and Remark \ref{rem: regularity of u}, noting that $u^\varepsilon_\lambda - \widehat{u}\in \Hs^1\left(Q_T\right)$. We have
    $$
    \begin{aligned}
        \norm{u^{\varepsilon}_{\lambda} - \widehat{u}}^2_{h,0} = \norm{u^{\varepsilon}_{\lambda} - \widehat{u}}_{h}^2 + \norm{\xi_h\left(u^{\varepsilon}_{\lambda} - \widehat{u}\right)}_{h}^2 &\le \norm{u^{\varepsilon}_{\lambda} - \widehat{u}}_{h}^2 + C \norm{\partial_t\left(u^{\varepsilon}_{\lambda} - \widehat{u}\right)}^2_{\Ls^2\left(Q_T\right)} \\
        &\le C \norm{u^{\varepsilon}_{\lambda} - \widehat{u}}^2_{\Hs^1\left(Q_T\right)}\le C \norm{f^{\varepsilon}_{\lambda} - \widehat{f_h}}^2_{\Ls^2\left(Q_T\right)}.
    \end{aligned}
    $$
    Next, we estimate the second term on the right-hand side of \eqref{eq: error estimate constant 1}. It follows from Lemma \ref{lem: state error estimate discrete norm}, Remark \ref{rem: regularity of u}, and the technique in \eqref{eq: stability of the source} that
    $$
    \norm{\widehat{u} - u_h\left(\widehat{f_h}\right)}_{h,0}\le Ch\norm{\widehat{u}}_{\Hs^s\left(Q_1\cup Q_2\right)}\le Ch\norm{\widehat{f_h}}_{\Ls^2\left(Q_T\right)} \le \dfrac{Ch}{\sqrt{\lambda}}\left(\norm{U_d}_{\Ls^2\left(\omega_T\right)} +\norm{g}_{\Ls^2\left(Q_T\right)}+\varepsilon\right).
    $$
    Finally, we invoke the arguments of \eqref{eq: discrete-auxiliary estimate variational 1} to get 
    $$
    \norm{u_h\left(\widehat{f_h}\right) - u^{\varepsilon}_{\lambda, h}}_{h,0} \le C\norm{\widehat{f_h} - f^{\varepsilon}_{\lambda, h}}_{\Ls^2\left(Q_T\right)}.
    $$
    By substituting the last three inequalities into \eqref{eq: error estimate constant 1}, we arrive at
    $$
    \norm{u^{\varepsilon}_{\lambda} - u^{\varepsilon}_{\lambda, h}}_{h,0} \le C\left[\norm{f^{\varepsilon}_{\lambda} - \widehat{f_h}}_{\Ls^2\left(Q_T\right)} + \dfrac{h}{\sqrt{\lambda}}\left(\norm{U_d}_{\Ls^2\left(\omega_T\right)} +\norm{g}_{\Ls^2\left(Q_T\right)}+\varepsilon\right) + \norm{\widehat{f_h} - f^{\varepsilon}_{\lambda, h}}_{\Ls^2\left(Q_T\right)} \right].
    $$
    We now turn to estimate the adjoint error. Following a similar approach as for the state, we proceed as follows
    $$
    \begin{aligned}
        &\norm{p^\varepsilon_\lambda - p^\varepsilon_{\lambda, h}}_{h,T} \le\\
        &\le \norm{p^\varepsilon_\lambda - \widehat{p}}_{h,T} + \norm{\widehat{p} - p_h\left(\widehat{f_h}\right)}_{h,T} + \norm{p_h\left(\widehat{f_h}\right) - \widehat{p_h}\left(\widehat{f_h}\right)}_{h,T} + \norm{\widehat{p_h}\left(\widehat{f_h}\right) - p^\varepsilon_{\lambda, h}}_{h,T}\\
        &\le C\left(\norm{u^{\varepsilon}_{\lambda} - \widehat{u}}_{\Ls^2\left(\omega_T\right)} + h\norm{\widehat{p}}_{\Hs^s\left(Q_1\cup Q_2\right)} + \norm{\widehat{u} - u_h\left(\widehat{f_h}\right)}_{\Ls^2\left(\omega_T\right)} + \norm{u^\ast - u^\ast_h}_{\Ls^2\left(\omega_T\right)} + \norm{u_h\left(\widehat{f_h}\right) - u^{\varepsilon}_{\lambda, h}}_{\Ls^2\left(\omega_T\right)}\right)\\
        &\le C \left(\norm{f^{\varepsilon}_{\lambda} - \widehat{f_h}}_{\Ls^2\left(Q_T\right)} + h\norm{\widehat{p}}_{\Hs^s\left(Q_1\cup Q_2\right)} + h^2\norm{\widehat{u}}_{\Hs^s\left(Q_1\cup Q_2\right)} + h^2\norm{u^\ast}_{\Hs^s\left(Q_1\cup Q_2\right)} + \norm{\widehat{f_h} - f^{\varepsilon}_{\lambda, h}}_{\Ls^2\left(Q_T\right)}\right)\\
        &\le C \left[\norm{f^{\varepsilon}_{\lambda} - \widehat{f_h}}_{\Ls^2\left(Q_T\right)} + h\left(1+\dfrac{1}{\sqrt{\lambda}}\right)\left(\norm{U_d}_{\Ls^2\left(\omega_T\right)} +\norm{g}_{\Ls^2\left(Q_T\right)}+\varepsilon\right) + \norm{\widehat{f_h} - f^{\varepsilon}_{\lambda, h}}_{\Ls^2\left(Q_T\right)}\right].
    \end{aligned}
    $$
    Together with Lemmas \ref{lem: pure control estimate constant} and \ref{lem: discrete the pure control constant}, noting that $h^2< Ch$ for all $h\in \left(0,h_\ast\right)$ with a given $h_\ast>0$, we arrive at the first result. The proof of the second estimate follows simlilarly.
\end{proof}

\subsection{Post-processing strategy}
\label{subsec: error postprocessing}

We first discretize the regularized source $f^\varepsilon_\lambda\in F_{+}$ in \eqref{eq: problem formulation} by using element-wise constant functions associated with the mesh $\mathcal{T}_h$. Subsequently, we construct the post-processing solution $f_h^\dagger \in F^h_{+}$, as defined in \eqref{eq: post-processing solution}. Recall that $F^h_{+} = F^d \cap F_{+}$ denotes the discrete admissible set, where $F^d$ is defined in \eqref{eq: step functions space}. This subsection aims to estimate the error $f^\varepsilon_\lambda - f_h^\dagger$. 

Let us begin by decomposing the mesh $\mathcal{T}_h$ into three subsets $\mathcal{T}_h = \mathcal{M}_h^1 \cup \mathcal{M}_h^2 \cup \mathcal{M}_h^3$, where
$$
\mathcal{M}_h^1 := \left\{K\in \mathcal{T}_h\mid f^\varepsilon_\lambda = 0 \text{ at every point in } K \right\}\qqq \text{and}\qqq \mathcal{M}_h^2 := \left\{K\in \mathcal{T}_h \mid f^\varepsilon_\lambda > 0 \text{ at every point in } K \right\},
$$
and $\mathcal{M}_h^3 = \mathcal{T}_h \setminus \left(\mathcal{M}_h^1 \cup \mathcal{M}_h^2\right)$. The sets $\mathcal{M}_h^1$ and $\mathcal{M}_h^2$ are referred to as the active and inactive sets, respectively, while $\mathcal{M}_h^3$ contains elements intersecting the free boundary between $\mathcal{M}_h^1$ and $\mathcal{M}_h^2$. Notably, the sets are pairwise disjoint, i.e., $\mathcal{M}_h^i \cap \mathcal{M}_h^j = \varnothing$ for all distinct $i,j \in \left\{1,2,3\right\}$. 

Next, we deduce from \eqref{eq: projection formula of f} the regularity of $f^\varepsilon_\lambda$. Clearly, it depends on the regularity of its corresponding adjoint $p^\varepsilon_\lambda\in \Xs_T$ and the function $\ell$. Assume that $p^\varepsilon_\lambda\in \Hs^{1}\left(Q_T\right)\cap\Hs^s\left(Q_1\cup Q_2\right)$, where $s>\frac{d+3}{2}$ is given (cf. Assumption \ref{assum: Nitsche's trick assumption}). Additionally, if $\ell\in \Ws^{1,\infty}\left(Q_T\right)$, Lemma \ref{lem: control regularity} asserts that $f^\varepsilon_\lambda\in \Ws^{1,\infty}\left(Q_T\right)$. Furthermore, if $\ell \in \Hs^2\left(Q_1\cup Q_2\right)$, then on certain elements $K\in \mathcal{T}_h$, we also have $f^\varepsilon_\lambda\in \Hs^2\left(K\right)$. Indeed, the following three scenarios arise:

\begin{enumerate}
    \item $K\in \mathcal{M}_h^1$: Since $f^\varepsilon_\lambda\equiv 0$ on $K$, it follows directly that $f^\varepsilon_\lambda\in \Hs^2\left(K\right)$.
    \item $K\in \mathcal{M}_h^2$: Since $f^\varepsilon_\lambda = -\frac{1}{\lambda}\ell p^\varepsilon_\lambda$ on $K$, the regularity of $f^\varepsilon_\lambda$ matches that of $\ell p^\varepsilon_\lambda$. If $K\notin \mathcal{T}_h^\ast$, then $\ell p^\varepsilon_\lambda\in \Hs^2\left(K\right)$, and hence $f^\varepsilon_\lambda \in \Hs^2\left(K\right)$. Conversely, if $K\in \mathcal{T}_h^\ast$, then $\ell p^\varepsilon_\lambda\in \Ws^{1,\infty}\left(K\right)$, but not in $\Hs^2\left(K\right)$. The same applies to $f^\varepsilon_\lambda$. 
    \item $K\in \mathcal{M}_h^3$: Due to the kink in \eqref{eq: projection formula of f}, we only have $f^\varepsilon_\lambda\in \Ws^{1,\infty}\left(K\right)$.
\end{enumerate}

In general, $f^\varepsilon_\lambda$ exhibits lower regularity than $p^\varepsilon_\lambda$, which consequently complicates the derivation of a high-order estimate for the error $f^\varepsilon_\lambda - f_h^\dagger$. To overcome this difficulty and achieve the desired high-order convergence, an additional condition on the mesh structure is required. In \cite[Assumption 2]{BV2007} and \cite[Assumption 1]{MV2008}, the following assumption was introduced
\begin{equation}
    \label{eq: old mesh structure assumption}
    \sum_{K\in \mathcal{M}_h^3}\abs{K}\le Ch,
\end{equation}
where $C>0$ is an $h$-independent constant. This assumption holds if the boundary of the set $\left\{\left(\xb, t\right)\in Q_T\mid f^\varepsilon_\lambda\left(\xb, t\right) = 0\right\}$ consists of a finite number of curves (or surfaces) having a finite length (or area) \cite[Section 6]{RV2006}. 

In this paper, we weaken \eqref{eq: old mesh structure assumption} by exploiting \eqref{eq: Th^ast area}. Based on $\mathcal{M}_h^1, \mathcal{M}_h^2$, and $\mathcal{M}_h^3$, we define 
$$
\mathcal{T}_h^1 := \left(\mathcal{M}_h^1 \cup \mathcal{M}_h^2\right)\setminus \mathcal{T}_h^\ast\qqqq \text{and}\qqqq \mathcal{T}_h^2 := \mathcal{M}_h^3\setminus \mathcal{T}_h^\ast,
$$
and form a new partition
$$
\mathcal{T}_h = \mathcal{T}_h^1 \cup \mathcal{T}_h^2 \cup \mathcal{T}_h^\ast.
$$
It is evident that $\mathcal{T}_h^1\cap \mathcal{T}_h^2 = \mathcal{T}_h^2\cap \mathcal{T}_h^\ast = \mathcal{T}_h^\ast\cap \mathcal{T}_h^1 =\varnothing$. To address the mesh structure, the following alternative assumption is sufficient:

\begin{assumption}
    \label{assum: mesh superconvergence}
    Assume that there exists an $h$-independent constant $C>0$ such that
    $$\sum_{K\in \mathcal{T}_h^2}\abs{K}\le Ch.$$
\end{assumption}

Due to \eqref{eq: Th^ast area}, we observe that if Assumption \ref{assum: mesh superconvergence} is satisfied, then \eqref{eq: old mesh structure assumption} also holds. Denote by $S_K$ the barycenter of an element $K\in \mathcal{T}_h$. We define the barycentric interpolant $\pi_d: \Cs\left(\overline{Q_T}\right)\to F^d$ element-wise as follows
\begin{equation}
    \label{eq: barycentric interpolation}
    \left(\pi_d v \right)_{\mid \, K} := v  \left(S_K\right)\qqqq \forall K\in \mathcal{T}_h,
\end{equation}
for all $v  \in \Cs\left(\overline{Q_T}\right)$. From Lemma \ref{lem: control regularity} and the compact embedding $\Ws^{1,\infty}\left(Q_T\right)\hookrightarrow \Cs\left(\overline{Q_T}\right)$ \cite[Theorem 2.35]{Ern2021}, it follows that $f^\varepsilon_\lambda \in \Cs\left(\overline{Q_T}\right)$. Consequently, we have $\pi_d f^\varepsilon_\lambda \in F^h_{+}$. 

Regarding the approximability of $\pi_d$, we recall the following lemma:

\begin{lemma}
    \label{lem: numerical integration}
    Let $K$ be an arbitrary element of $\mathcal{T}_h$ and $\pi_d$ be the function defined in \eqref{eq: barycentric interpolation}. For all $v \in \Hs^2\left(K\right)$, the following inequality holds
    $$
    \abs{\int\limits_{K} \left(v  - \pi_d v \right)\dx\dt} \le Ch^2\abs{K}^{\frac{1}{2}}\norm{\Ds^2 v }_{\mathbb{L}^2\left(K\right)}.
    $$
    For all $v  \in \Ws^{1,\infty}\left(K\right)$, we have
    $$
    \norm{v  - \pi_d v }_{\Ls^\infty\left(K\right)}\le Ch\norm{\Ds v }_{\LLs^\infty\left(K\right)}.
    $$
    The constant $C>0$ is independent of $v $ and $h$ in these estimates.
\end{lemma}

\begin{proof}
    When $K\subset \mathbb{R}^2$, the proof of the first inequality is presented in \cite[Lemma 3.2]{MR2004}. The technique from that work can be extended to the case $K\subset \mathbb{R}^3$ without significant modifications and can also be used to prove the second inequality; see \cite[Assumption A.6]{RV2006}. 
\end{proof}

Now, we are in a position to present the first technical result. Let us introduce $u_h^\pi\left(f^\varepsilon_\lambda\right) \in \Vs_{h,0}$ and $p_h^\pi\left(f^\varepsilon_\lambda\right) \in \Vs_{h,T}$ as the solutions to the following problems
\begin{equation}
    \label{eq: auxiliary state postprocessing}
    a_h\left(u^\pi_h\left(f^\varepsilon_\lambda\right), \vphi_h\right) = \left(\ell \pi_d f^\varepsilon_\lambda, \vphi_h\right)_{\Ls^2\left(Q_T\right)}\qqqq\forall \vphi_h \in \Vs_{h,0},
\end{equation}
and 
\begin{equation}
    \label{eq: auxiliary adjoint postprocessing}
    a_h^\prime\left(p_h^\pi\left(f^\varepsilon_\lambda\right), \phi_h\right) = \left(\chi_{\omega_T}\left(u_h^\pi\left(f^\varepsilon_\lambda\right) - z^\varepsilon_{d,h}\right), \phi_h\right)_{\Ls^2\left(Q_T\right)}\qqqq\forall \phi_h \in \Vs_{h,T}.
\end{equation}

\begin{lemma}
    \label{lem: exact-postprocess states and adjoint estimate}
    Let $\left(u_h\left(f^\varepsilon_\lambda\right), p_h\left(f^\varepsilon_\lambda\right)\right) \in \Vs_{h,0}\times \Vs_{h,T}$ and $\left(u_h^\pi\left(f^\varepsilon_\lambda\right), p_h^\pi\left(f^\varepsilon_\lambda\right)\right)\in \Vs_{h,0}\times \Vs_{h,T}$ be the solutions to Problems \eqref{eq: auxiliary state variational}--\eqref{eq: auxiliary adjoint variational} and \eqref{eq: auxiliary state postprocessing}--\eqref{eq: auxiliary adjoint postprocessing}, respectively. Let $p^\varepsilon_\lambda\in \Xs_T$ be the solution to Problem \eqref{eq: weak adjoint equation optimality conditions}. Assume that Assumptions \ref{assum: Nitsche's trick assumption} and \ref{assum: mesh superconvergence} hold, and $\ell \in \Ws^{1,\infty}\left(Q_T\right)\cap \Hs^2\left(Q_1\cup Q_2\right)$. Then, we have the following inequality
    $$
    \begin{aligned}
        &\norm{u_h\left(f^\varepsilon_\lambda\right) - u_h^\pi\left(f^\varepsilon_\lambda\right)}_{\Ls^2\left(Q_T\right)} + \norm{p_h\left(f^\varepsilon_\lambda\right) - p_h^\pi\left(f^\varepsilon_\lambda\right)}_{\Ls^2\left(Q_T\right)} \le\\
        &\le \dfrac{C}{\lambda}\left[h^{\frac{3}{2}} \norm{\Ds \left(\ell p^\varepsilon_\lambda\right)}_{\LLs^\infty\left(Q_T\right)} + h^2\left(1+\dfrac{1}{\sqrt{\lambda}}\right)\left(\norm{U_d}_{\Ls^2\left(\omega_T\right)} +\norm{g}_{\Ls^2\left(Q_T\right)}+\varepsilon\right)\right].
    \end{aligned}
    $$
\end{lemma}

\begin{proof}
    We subtract \eqref{eq: auxiliary state postprocessing} from \eqref{eq: auxiliary state variational} to obtain
    $$
    a_h\left(u_h\left(f^\varepsilon_\lambda\right) - u_h^\pi\left(f^\varepsilon_\lambda\right), \vphi_h\right) = \left(\ell f^\varepsilon_\lambda - \ell \pi_d f^\varepsilon_\lambda , \vphi_h\right)_{\Ls^2\left(Q_T\right)}\qqqq\forall \vphi_h \in \Vs_{h,0}.
    $$
    By applying \eqref{eq: compare L2 and the Y norm} and \eqref{eq: coercivity of ah}, and then choosing $\vphi_h = u_h\left(f^\varepsilon_\lambda\right) - u_h^\pi\left(f^\varepsilon_\lambda\right)\in \Vs_{h,0}$ in the above equation, we arrive at
    $$
    \begin{aligned}
        C\norm{u_h\left(f^\varepsilon_\lambda\right) - u_h^\pi\left(f^\varepsilon_\lambda\right)}^2_{\Ls^2\left(Q_T\right)}\le \norm{u_h\left(f^\varepsilon_\lambda\right) - u_h^\pi\left(f^\varepsilon_\lambda\right)}^2_{h} &\le a_h\left(u_h\left(f^\varepsilon_\lambda\right) - u_h^\pi\left(f^\varepsilon_\lambda\right), u_h\left(f^\varepsilon_\lambda\right) - u_h^\pi\left(f^\varepsilon_\lambda\right)\right) \\
        &= \left(\ell u_h\left(f^\varepsilon_\lambda\right) - \ell u_h^\pi\left(f^\varepsilon_\lambda\right), f^\varepsilon_\lambda - \pi_d f^\varepsilon_\lambda \right)_{\Ls^2\left(Q_T\right)}.
    \end{aligned}
    $$
    Let us denote $\sigma_h:= \ell u_h\left(f^\varepsilon_\lambda\right) - \ell u_h^\pi\left(f^\varepsilon_\lambda\right)$ for convenience. Since $\ell\in \Ws^{1,\infty}\left(Q_T\right)$ and $u_h\left(f^\varepsilon_\lambda\right) - u_h^\pi\left(f^\varepsilon_\lambda\right)\in \Vs_{h,0}$, it follows that $\sigma_h\in \Hs^1\left(Q_T\right)$. We split the right-hand side of the last inequality into three integrals to get
    \begin{align}
        C\norm{u_h\left(f^\varepsilon_\lambda\right) - u_h^\pi\left(f^\varepsilon_\lambda\right)}^2_{\Ls^2\left(Q_T\right)} &\le \int\limits_{\mathcal{T}_h^1}\sigma_h\left(f^\varepsilon_\lambda - \pi_d f^\varepsilon_\lambda \right) \dx\dt + \int\limits_{\mathcal{T}_h^2}\sigma_h\left(f^\varepsilon_\lambda - \pi_d f^\varepsilon_\lambda \right) \dx\dt + \int\limits_{\mathcal{T}_h^\ast}\sigma_h\left(f^\varepsilon_\lambda - \pi_d f^\varepsilon_\lambda \right) \dx\dt \notag\\
        &=: J_4 + J_5 + J_6. \label{eq: exact-postprocess states and adjoint estimate 1}
    \end{align}
    
    First, we estimate $J_4$. Let $K$ be an arbitrary element of $\mathcal{T}_h^1$. If $K\in \mathcal{M}_h^1$, then $f^\varepsilon_\lambda\equiv 0$ on $K$, and hence the integral over $K$ is zero. Consider the nontrivial case $K\in \mathcal{T}_h^1\setminus\mathcal{M}_h^1$. By applying \eqref{eq: L2 orthogonal projection} with $v  = \Pi_d f^\varepsilon_\lambda \in F^h_+$, we deduce the following
    $$
    \int\limits_K \Pi_d f^\varepsilon_\lambda \dx\dt = \int\limits_K \left(\dfrac{1}{\abs{K}}\int\limits_{K}f^\varepsilon_\lambda\dx\dt\right)\dx\dt = \left(\dfrac{1}{\abs{K}}\int\limits_{K}f^\varepsilon_\lambda\dx\dt\right)\left(\int\limits_K 1 \dx\dt\right) = \int\limits_K f^\varepsilon_\lambda \dx\dt,
    $$
    which means that
    $$
    \int\limits_K \left(f^\varepsilon_\lambda - \Pi_d f^\varepsilon_\lambda \right)\dx\dt = 0.
    $$
    Consequently, we have
    $$
    \left(\int\limits_K \Pi_d\sigma_h \dx\dt\right)\left[\int\limits_K \left(f^\varepsilon_\lambda - \Pi_d f^\varepsilon_\lambda \right)\dx\dt\right] = \int\limits_K \left(\Pi_d \sigma_h\right)\left(f^\varepsilon_\lambda - \Pi_d f^\varepsilon_\lambda \right)\dx\dt=0,
    $$
    since $\Pi_d\sigma_h$ is a constant on $K$. Together with the fact that $\Pi_d f^\varepsilon_\lambda -\pi_d f^\varepsilon_\lambda $ is also a constant on $K$, we thus imply
    $$
    \begin{aligned}
        \int\limits_{K}\sigma_h\left(f^\varepsilon_\lambda - \pi_d f^\varepsilon_\lambda \right)\dx\dt&= \int\limits_{K}\sigma_h\left(f^\varepsilon_\lambda - \Pi_d f^\varepsilon_\lambda \right)\dx\dt + \int\limits_{K}\sigma_h\left(\Pi_d f^\varepsilon_\lambda  - \pi_d f^\varepsilon_\lambda \right)\dx\dt\\
        &=\int\limits_{K}\sigma_h\left(f^\varepsilon_\lambda - \Pi_d f^\varepsilon_\lambda \right)\dx\dt + \dfrac{1}{\abs{K}}\left(\int\limits_{K} \sigma_h \dx\dt\right)\left[ \int\limits_{K}\left(\Pi_d f^\varepsilon_\lambda  - \pi_d f^\varepsilon_\lambda \right)\dx\dt\right]\\
        &= J_4^a+J_4^b+J_4^c,
    \end{aligned}
    $$
    where $J_4^a, I_b^b$, and $J_4^c$ are defined as
    $$
    J_4^a :=\int\limits_{K}\left(\sigma_h - \Pi_d \sigma_h\right)\left(f^\varepsilon_\lambda - \Pi_d f^\varepsilon_\lambda \right)\dx\dt\qqq \text{and}\qqq J_4^b := \dfrac{1}{\abs{K}}\left[\int\limits_K \left(\sigma_h - \Pi_d\sigma_h\right)\dx\dt\right]\left[\int\limits_K \left(\Pi_d f^\varepsilon_\lambda  - f^\varepsilon_\lambda\right)\dx\dt\right],
    $$
    and
    $$
    J_4^c := \dfrac{1}{\abs{K}}\left(\int\limits_K \sigma_h\dx\dt\right)\left[\int\limits_K \left(f^\varepsilon_\lambda - \pi_d f^\varepsilon_\lambda \right)\dx\dt\right].
    $$
    On any element $K\in \mathcal{T}_h^1\setminus\mathcal{M}_h^1$, we have $f^\varepsilon_\lambda = -\frac{1}{\lambda}\ell p^\varepsilon_\lambda$. Moreover, by Assumption \ref{assum: Nitsche's trick assumption} and $\ell \in \Ws^{1,\infty}\left(Q_T\right)\cap \Hs^2\left(Q_1\cup Q_2\right)$, it follows that $f^\varepsilon_\lambda\in \Ws^{1,\infty}\left(K\right)\cap \Hs^2\left(K\right)$. \\
    We first estimate $J_4^a$. By applying \eqref{eq: error estimates for projection operators ver 2} and \eqref{eq: error estimates for projection operators}, we obtain
    $$
    \begin{aligned}
        \abs{J_4^a} \le \norm{\sigma_h - \Pi_d \sigma_h}_{\Ls^2\left(K\right)}\norm{f^\varepsilon_\lambda - \Pi_d f^\varepsilon_\lambda }_{\Ls^2\left(K\right)} &\le \norm{\sigma_h - \Pi_d \sigma_h}_{\Ls^2\left(K\right)}\abs{K}^{\frac{1}{2}}\norm{f^\varepsilon_\lambda - \Pi_d f^\varepsilon_\lambda }_{\Ls^\infty\left(K\right)} \\
        &\le Ch \norm{\Ds \sigma_h}_{\LLs^2\left(K\right)}\abs{K}^{\frac{1}{2}}h\norm{\Ds f^\varepsilon_\lambda}_{\LLs^\infty\left(K\right)}\\
        &=C h\norm{\Ds \sigma_h}_{\LLs^2\left(K\right)}\abs{K}^{\frac{1}{2}}\dfrac{h}{\lambda}\norm{\Ds \left(\ell p^\varepsilon_\lambda\right)}_{\LLs^\infty\left(K\right)}.
    \end{aligned}
    $$
    Next, consider $J_4^b$. By construction, we observe that $J_4^b=0$. \\
    Finally, we estimate $J_4^c$. Since $K\in \mathcal{T}_h^1\setminus\mathcal{M}_h^1$, we have either $K\subset\overline{Q_1}$ or $K\subset\overline{Q_2}$ (please refer to Section \ref{sec: discretization}). Without loss of generality, assume that $K\subset\overline{Q_1}$. We use Lemma \ref{lem: numerical integration} to get
    $$
    \abs{J_4^c} \le \dfrac{C}{\abs{K}} \left(\abs{K}^{\frac{1}{2}}\norm{\sigma_h}_{\Ls^2\left(K\right)}\right)\left(h^2\abs{K}^{\frac{1}{2}}\norm{\Ds^2 f^\varepsilon_\lambda}_{\mathbb{L}^2\left(K\right)}\right) = C\norm{\sigma_h}_{\Ls^2\left(K\right)}\dfrac{h^2}{\lambda}\norm{\Ds^2 \left(\ell p^\varepsilon_\lambda\right)}_{\mathbb{L}^2\left(K\right)}.
    $$
    Owing to the triangle inequality, the H{\" o}lder inequality, and \eqref{eq: extension operator}, we have
    \begin{equation}
        \label{eq: exact-postprocess states and adjoint estimate 6}
        \begin{aligned}
            \norm{\Ds^2 \left(\ell p^\varepsilon_\lambda\right)}_{\mathbb{L}^2\left(K\right)} &\le C\left(\norm{\left(\Ds^2\ell\right)p^\varepsilon_\lambda}^2_{\mathbb{L}^2\left(K\right)} + \norm{\Ds\ell \Ds p^\varepsilon_\lambda}^2_{\mathbb{L}^2\left(K\right)} + \norm{\Ds p^\varepsilon_\lambda \Ds\ell}^2_{\mathbb{L}^2\left(K\right)} + \norm{\ell \left(\Ds^2 p^\varepsilon_\lambda\right)}^2_{\mathbb{L}^2\left(K\right)}\right)^{\frac{1}{2}}\\
            &\le C \norm{\ell}_{\Hs^2\left(K\right)}\norm{p^\varepsilon_\lambda}_{\Hs^2\left(K\right)} \\
            &\le C\norm{\ell}_{\Hs^2\left(Q_1\cup Q_2\right)}\norm{\Es_1 p^\varepsilon_\lambda}_{\Hs^2\left(K\right)}.
        \end{aligned}
    \end{equation}
    Therefore, we imply
    $$
    \abs{J^c_4} \le C\norm{\sigma_h}_{\Ls^2\left(K\right)}\dfrac{h^2}{\lambda}\norm{\Es_1 p^\varepsilon_\lambda}_{\Hs^2\left(K\right)},
    $$
    and thus
    $$
    \begin{aligned}
        \abs{\int\limits_{K}\sigma_h\left(f^\varepsilon_\lambda - \pi_d f^\varepsilon_\lambda \right)\dx\dt} &\le \abs{J^a_4} + \abs{J_b^4}+\abs{J_c^4} \\
        &\le \dfrac{Ch^2}{\lambda}\left(\norm{\Ds \sigma_h}_{\LLs^2\left(K\right)}\abs{K}^{\frac{1}{2}}\norm{\Ds\left(\ell p^\varepsilon_\lambda\right)}_{\LLs^\infty\left(K\right)} + \norm{\sigma_h}_{\Ls^2\left(K\right)}\norm{\Es_1 p^\varepsilon_\lambda}_{\Hs^2\left(K\right)}\right).
    \end{aligned}
    $$
    Additionally, applying the Cauchy-Schwarz inequality yields
    $$
    \sum_{K\in \mathcal{T}_h^1\setminus \mathcal{M}_h^1}\norm{\Ds \sigma_h}_{\LLs^2\left(K\right)}\abs{K}^{\frac{1}{2}}\norm{\Ds\left(\ell p^\varepsilon_\lambda\right)}_{\LLs^\infty\left(K\right)} \le \left(\sum_{K\in \mathcal{T}_h^1\setminus \mathcal{M}_h^1} \norm{\Ds \sigma_h}^2_{\LLs^2\left(K\right)}\right)^{\frac{1}{2}}\left(\sum_{K\in \mathcal{T}_h^1\setminus \mathcal{M}_h^1} \abs{K}\right)^{\frac{1}{2}}\norm{\Ds\left(\ell p^\varepsilon_\lambda\right)}_{\LLs^\infty\left(Q_T\right)},
    $$
    and
    $$
    \sum_{K\in \mathcal{T}_h^1\setminus \mathcal{M}_h^1}\norm{\sigma_h}_{\Ls^2\left(K\right)}\norm{\Es_i p^\varepsilon_\lambda}_{\Hs^2\left(K\right)} \le \left(\sum_{K\in \mathcal{T}_h^1\setminus \mathcal{M}_h^1} \norm{\sigma_h}^2_{\Ls^2\left(K\right)}\right)^{\frac{1}{2}}\left(\sum_{i=1}^2\norm{\Es_i p^\varepsilon_\lambda}^2_{\Hs^2\left(Q_T\right)}\right)^{\frac{1}{2}},
    $$
    where $i\in \left\{1,2\right\}$. We hence sum over all elements $K\in \mathcal{T}_h^1$ and employ \eqref{eq: extension operator} once more to arrive at
    \begin{equation}
        \label{eq: exact-postprocess states and adjoint estimate 3}
        J_4 \le \sum_{K\in \mathcal{T}_h^1}\abs{\int\limits_{K}\sigma_h\left(f^\varepsilon_\lambda - \pi_d f^\varepsilon_\lambda \right)\dx\dt}\le \dfrac{Ch^2}{\lambda}\left(\norm{\Ds \sigma_h}_{\LLs^2\left(Q_T\right)}\norm{\Ds\left(\ell p^\varepsilon_\lambda\right)}_{\LLs^\infty\left(Q_T\right)} + \norm{\sigma_h}_{\Ls^2\left(Q_T\right)}\norm{ p^\varepsilon_\lambda}_{\Hs^2\left(Q_1\cup Q_2\right)}\right).
    \end{equation}
    We next estimate $\norm{\Ds \sigma_h}_{\LLs^2\left(Q_T\right)}$ and $\norm{\sigma_h}_{\Ls^2\left(Q_T\right)}$. On the one hand, let $\overline{u}\left(f^\varepsilon_\lambda\right)\in \Xs_0$ be the solution to the problem 
    $$
    a\left(\overline{u}\left(f^\varepsilon_\lambda\right),\vphi\right) = \left(\ell f^\varepsilon_\lambda - \ell \pi_d f^\varepsilon_\lambda ,\vphi\right)_{\Ls^2\left(Q_T\right)} \qqqq \forall \vphi\in \Ys.
    $$
    Then, by invoking $\ell \in \Ws^{1,\infty}\left(Q_T\right)$, Lemma \ref{lem: stability result}, and Remark \ref{rem: regularity of u}, we obtain
    $$
    \begin{aligned}
        \norm{\Ds \sigma_h}_{\LLs^2\left(Q_T\right)}\le C\norm{\Ds \left(u_h\left(f^\varepsilon_\lambda\right) - u_h^\pi\left(f^\varepsilon_\lambda\right)\right)}_{\LLs^2\left(Q_T\right)}&\le Ch\norm{\overline{u}\left(f^\varepsilon_\lambda\right)}_{\Hs^s\left(Q_1\cup Q_2\right)} + \norm{\overline{u}\left(f^\varepsilon_\lambda\right)}_{\Hs^1\left(Q_T\right)}\\
        &\le Ch\norm{\ell f^\varepsilon_\lambda - \ell \pi_d f^\varepsilon_\lambda }_{\Ls^2\left(Q_T\right)}+C\norm{\ell f^\varepsilon_\lambda - \ell \pi_d f^\varepsilon_\lambda }_{\Ls^2\left(Q_T\right)}\\
        &\le C\norm{ f^\varepsilon_\lambda - \pi_d f^\varepsilon_\lambda }_{\Ls^2\left(Q_T\right)},
    \end{aligned}
    $$
    for all $h\in \left(0,h_\ast\right)$, where $h_\ast>0$ is given. Moreover, from Lemmas \ref{lem: numerical integration} and \ref{lem: control regularity}, it follows that
    $$
    \begin{aligned}
        \norm{f^\varepsilon_\lambda - \pi_d f^\varepsilon_\lambda }^2_{\Ls^2\left(Q_T\right)} \le \sum_{K\in \mathcal{T}_h} \abs{K} \norm{f^\varepsilon_\lambda - \pi_d f^\varepsilon_\lambda }^2_{\Ls^\infty\left(K\right)} \le C \sum_{K\in \mathcal{T}_h} \left(\abs{K} h^2\norm{\Ds f^\varepsilon_\lambda}^2_{\LLs^\infty\left(K\right)}\right)&\le C\left(\sum_{K\in \mathcal{T}_h} \abs{K}\right)h^2\norm{\Ds f^\varepsilon_\lambda}^2_{\LLs^\infty\left(Q_T\right)}\\
        &\le \dfrac{Ch^2}{\lambda^2} \norm{\Ds \left(\ell p^\varepsilon_\lambda\right)}^2_{\LLs^\infty\left(Q_T\right)},
    \end{aligned}
    $$
    which implies
    $$
    \norm{\Ds \sigma_h}_{\LLs^2\left(Q_T\right)}\le C\norm{f^\varepsilon_\lambda - \pi_d f^\varepsilon_\lambda }_{\Ls^2\left(Q_T\right)} \le \dfrac{Ch}{\lambda} \norm{\Ds \left(\ell p^\varepsilon_\lambda\right)}_{\LLs^\infty\left(Q_T\right)}.
    $$
    On the other hand, we clearly have
    $$
    \norm{\sigma_h}_{\Ls^2\left(Q_T\right)} \le  C\norm{u_h\left(f^\varepsilon_\lambda\right) - u_h^\pi\left(f^\varepsilon_\lambda\right)}_{\Ls^2\left(Q_T\right)}.
    $$
    By substituting the last two inequalities into \eqref{eq: exact-postprocess states and adjoint estimate 3}, we obtain 
    \begin{equation}
        \label{eq: exact-postprocess states and adjoint estimate 2}
        J_4\le \dfrac{Ch^3}{\lambda^2}\norm{\Ds\left(\ell p^\varepsilon_\lambda\right)}^2_{\LLs^\infty\left(Q_T\right)} + \dfrac{Ch^2}{\lambda}\norm{u_h\left(f^\varepsilon_\lambda\right) - u_h^\pi\left(f^\varepsilon_\lambda\right)}_{\Ls^2\left(Q_T\right)}\norm{p^\varepsilon_\lambda}_{\Hs^2\left(Q_1\cup Q_2\right)}.
    \end{equation}
    
    Next, we estimate $J_5$. By Assumption \ref{assum: Nitsche's trick assumption} and $\ell \in \Ws^{1,\infty}\left(Q_T\right)$, it follows that $f^\varepsilon_\lambda\in \Ws^{1,\infty}\left(K\right)$ for any $K\in \mathcal{T}_h^2$. Thanks to Lemma \ref{lem: numerical integration}, the following estimate holds
    $$
    \begin{aligned}
        \abs{\int\limits_{K}\sigma_h\left(f^\varepsilon_\lambda - \pi_d f^\varepsilon_\lambda \right)\dx\dt} \le \norm{\sigma_h}_{\Ls^2\left(K\right)}\norm{f^\varepsilon_\lambda - \pi_d f^\varepsilon_\lambda }_{\Ls^2\left(K\right)}&\le \norm{\sigma_h}_{\Ls^2\left(K\right)}\abs{K}^{\frac{1}{2}}\norm{f^\varepsilon_\lambda - \pi_d f^\varepsilon_\lambda }_{\Ls^\infty\left(K\right)}\\
        &\le C\norm{\sigma_h}_{\Ls^2\left(K\right)}\abs{K}^{\frac{1}{2}}h\norm{\Ds f^\varepsilon_\lambda}_{\LLs^\infty\left(K\right)}.
    \end{aligned}
    $$
    We then sum over all elements $K\in \mathcal{T}_h^2$ and subsequently apply the Cauchy-Schwarz inequality, Assumption \ref{assum: mesh superconvergence}, and Lemma \ref{lem: control regularity} to bound $J_5$. It follows that
    \begin{equation}
        \label{eq: exact-postprocess states and adjoint estimate 4}
        \begin{aligned}
            J_5 \le \sum_{K\in \mathcal{T}_h^2} \abs{\int\limits_{K}\sigma_h\left(f^\varepsilon_\lambda - \pi_d f^\varepsilon_\lambda \right)\dx\dt}&\le C \left(\sum_{K\in \mathcal{T}_h^2}\norm{\sigma_h}^2_{\Ls^2\left(K\right)}\right)^{\frac{1}{2}}\left(\sum_{K\in \mathcal{T}_h^2}\abs{K}\right)^{\frac{1}{2}}h\norm{\Ds f^\varepsilon_\lambda}_{\LLs^\infty\left(Q_T\right)}\\
            &\le C\norm{\sigma_h}_{\Ls^2\left(Q_T\right)}h^{\frac{3}{2}}\norm{\Ds f^\varepsilon_\lambda}_{\LLs^\infty\left(Q_T\right)}\\
            &\le C\norm{u_h\left(f^\varepsilon_\lambda\right) - u_h^\pi\left(f^\varepsilon_\lambda\right)}_{\Ls^2\left(Q_T\right)}\dfrac{h^{\frac{3}{2}}}{\lambda}\norm{\Ds \left(\ell p^\varepsilon_\lambda\right)}_{\LLs^\infty\left(Q_T\right)}.
        \end{aligned}
    \end{equation}

    By following the same steps as in \eqref{eq: exact-postprocess states and adjoint estimate 4}, except that the inequality in Assumption \ref{assum: mesh superconvergence} is replaced by \eqref{eq: Th^ast area}, one achieves the estimate for $J_6$. We end up with
    \begin{equation}
        \label{eq: exact-postprocess states and adjoint estimate 5}
        J_6\le C\norm{u_h\left(f^\varepsilon_\lambda\right) - u_h^\pi\left(f^\varepsilon_\lambda\right)}_{\Ls^2\left(Q_T\right)}\dfrac{h^{\frac{3}{2}}}{\lambda}\norm{\Ds \left(\ell p^\varepsilon_\lambda\right)}_{\LLs^\infty\left(Q_T\right)}.
    \end{equation}
    
    We now substitute \eqref{eq: exact-postprocess states and adjoint estimate 2}, \eqref{eq: exact-postprocess states and adjoint estimate 4}, and \eqref{eq: exact-postprocess states and adjoint estimate 5} into \eqref{eq: exact-postprocess states and adjoint estimate 1} to obtain
    $$
    \begin{aligned}
        &\norm{u_h\left(f^\varepsilon_\lambda\right) - u_h^\pi\left(f^\varepsilon_\lambda\right)}^2_{\Ls^2\left(Q_T\right)}\le\\
        &\le \dfrac{Ch^3}{\lambda^2}\norm{\Ds\left(\ell p^\varepsilon_\lambda\right)}^2_{\LLs^\infty\left(Q_T\right)} + \dfrac{C}{\lambda}\norm{u_h\left(f^\varepsilon_\lambda\right) - u_h^\pi\left(f^\varepsilon_\lambda\right)}_{\Ls^2\left(Q_T\right)}\left(h^{\frac{3}{2}}\norm{\Ds \left(\ell p^\varepsilon_\lambda\right)}_{\LLs^\infty\left(Q_T\right)} + h^2 \norm{p^\varepsilon_\lambda}_{\Hs^2\left(Q_1\cup Q_2\right)}\right),
    \end{aligned}
    $$
    which is a second-order inequation with respect to $\norm{u_h\left(f^\varepsilon_\lambda\right) - u_h^\pi\left(f^\varepsilon_\lambda\right)}_{\Ls^2\left(Q_T\right)}$. By using its discriminant, we find
    $$
    \begin{aligned}
        &\norm{u_h\left(f^\varepsilon_\lambda\right) - u_h^\pi\left(f^\varepsilon_\lambda\right)}_{\Ls^2\left(Q_T\right)}\\
        &\le \dfrac{C}{\lambda}\left[h^{\frac{3}{2}}\norm{\Ds \left(\ell p^\varepsilon_\lambda\right)}_{\LLs^\infty\left(Q_T\right)} + h^2 \norm{p^\varepsilon_\lambda}_{\Hs^2\left(Q_1\cup Q_2\right)} + \sqrt{\left(h^{\frac{3}{2}}\norm{\Ds \left(\ell p^\varepsilon_\lambda\right)}_{\LLs^\infty\left(Q_T\right)} + h^2 \norm{p^\varepsilon_\lambda}_{\Hs^2\left(Q_1\cup Q_2\right)}\right)^2 + h^3\norm{\Ds\left(\ell p^\varepsilon_\lambda\right)}^2_{\LLs^\infty\left(Q_T\right)}}\right]\\
        &\le \dfrac{C}{\lambda}\left[h^{\frac{3}{2}}\norm{\Ds \left(\ell p^\varepsilon_\lambda\right)}_{\LLs^\infty\left(Q_T\right)} + h^2 \norm{p^\varepsilon_\lambda}_{\Hs^2\left(Q_1\cup Q_2\right)} + \sqrt{\left(h^{\frac{3}{2}}\norm{\Ds \left(\ell p^\varepsilon_\lambda\right)}_{\LLs^\infty\left(Q_T\right)} + h^2 \norm{p^\varepsilon_\lambda}_{\Hs^2\left(Q_1\cup Q_2\right)} + h^{\frac{3}{2}}\norm{\Ds\left(\ell p^\varepsilon_\lambda\right)}_{\LLs^\infty\left(Q_T\right)}\right)^2}\right]\\
        &\le \dfrac{C}{\lambda}\left(h^{\frac{3}{2}}\norm{\Ds \left(\ell p^\varepsilon_\lambda\right)}_{\LLs^\infty\left(Q_T\right)} + h^2 \norm{p^\varepsilon_\lambda}_{\Hs^2\left(Q_1\cup Q_2\right)}\right).
    \end{aligned}
    $$
    Hence, we arrive at
    $$
    \begin{aligned}
        \norm{u_h\left(f^\varepsilon_\lambda\right) - u_h^\pi\left(f^\varepsilon_\lambda\right)}_{\Ls^2\left(Q_T\right)}&\le \dfrac{C}{\lambda}\left(h^{\frac{3}{2}}\norm{\Ds \left(\ell p^\varepsilon_\lambda\right)}_{\LLs^\infty\left(Q_T\right)} + h^2 \norm{p^\varepsilon_\lambda}_{\Hs^2\left(Q_1\cup Q_2\right)}\right) \\
        &\le \dfrac{C}{\lambda}\left[h^{\frac{3}{2}} \norm{\Ds \left(\ell p^\varepsilon_\lambda\right)}_{\LLs^\infty\left(Q_T\right)} + h^2\left(1+\dfrac{1}{\sqrt{\lambda}}\right)\left(\norm{U_d}_{\Ls^2\left(\omega_T\right)} +\norm{g}_{\Ls^2\left(Q_T\right)}+\varepsilon\right)\right],
    \end{aligned}
    $$
    using \eqref{eq: auxiliary-discrete adjoint estimate variational 4}.
    Together with \eqref{eq: compare L2 and the Y norm} and the technique in \eqref{eq: discrete-auxiliary estimate variational 3}, we get
    $$
    \begin{aligned}
        \norm{p_h\left(f^\varepsilon_\lambda\right) - p_h^\pi\left(f^\varepsilon_\lambda\right)}_{\Ls^2\left(Q_T\right)}&\le C\norm{p_h\left(f^\varepsilon_\lambda\right) - p_h^\pi\left(f^\varepsilon_\lambda\right)}_{h,T}\\
        &\le C\norm{u_h\left(f^\varepsilon_\lambda\right) - u_h^\pi\left(f^\varepsilon_\lambda\right)}_{\Ls^2\left(\omega_T\right)}\\
        &\le \dfrac{C}{\lambda}\left[h^{\frac{3}{2}} \norm{\Ds \left(\ell p^\varepsilon_\lambda\right)}_{\LLs^\infty\left(Q_T\right)} + h^2\left(1+\dfrac{1}{\sqrt{\lambda}}\right)\left(\norm{U_d}_{\Ls^2\left(\omega_T\right)} +\norm{g}_{\Ls^2\left(Q_T\right)}+\varepsilon\right)\right].
    \end{aligned}
    $$
    We finish the proof.
\end{proof}

Next, we derive an inequality for the interpolant $\pi_d f^\varepsilon_\lambda \in F^h_{+}$. By inserting $f=f^\varepsilon_{\lambda,h}\in F^h_{+}$ into \eqref{eq: variational inequality}, we end up with the following point-wise inequality, which holds almost everywhere
$$
\left(\ell \left(\xb,t\right)p^\varepsilon_\lambda\left(\xb,t\right) + \lambda f^\varepsilon_\lambda\left(\xb,t\right)\right)\left(f^\varepsilon_{\lambda, h}\left(\xb,t\right) - f^\varepsilon_\lambda\left(\xb,t\right)\right)\ge 0 \qqqq \text{for almost every } \left(\xb, t\right)\in Q_T.
$$
Recall that, due to Lemma \ref{lem: control regularity} and the compact embedding $\Ws^{1,\infty}\left(Q_T\right)\hookrightarrow \Cs\left(\overline{Q_T}\right)$ \cite[Theorem 2.35]{Ern2021}, we have $\ell p^\varepsilon_\lambda\in \Cs\left(\overline{Q_T}\right)$ and $f^\varepsilon_\lambda\in \Cs\left(\overline{Q_T}\right)$. Therefore, we can choose $\left(\xb,t\right)=S_K$ for any $K\in \mathcal{T}_h$ in the last inequality. This yields
$$
\left[\ell \left(S_K\right)p^\varepsilon_\lambda\left(S_K\right) + \lambda f^\varepsilon_\lambda\left(S_K\right)\right]\left[f^\varepsilon_{\lambda, h}\left(S_K\right) - f^\varepsilon_\lambda\left(S_K\right)\right]\ge 0 \qqqq \forall K\in \mathcal{T}_h,
$$
which is equivalent to
$$
\left[\left(\pi_d\left(\ell p^\varepsilon_\lambda\right)\right)_{\mid \, K} + \lambda \left(\pi_d f^\varepsilon_\lambda \right)_{\mid \, K}\right]\left[f^\varepsilon_{\lambda, h}\left(S_K\right) - \left(\pi_d f^\varepsilon_\lambda \right)_{\mid \, K}\right]\ge 0 \qqqq \forall K\in \mathcal{T}_h.
$$
We now integrate the left-hand side over an arbitrary element $K\in \mathcal{T}_h$, and then sum over all elements to get
\begin{equation}
    \label{eq: barycentric variational inequality}
    \left(\pi_d\left(\ell p^\varepsilon_\lambda\right) + \lambda \pi_d f^\varepsilon_\lambda , f^\varepsilon_{\lambda, h} - \pi_d f^\varepsilon_\lambda \right)_{\Ls^2\left(Q_T\right)}\ge 0.
\end{equation}

\begin{lemma}
    \label{lem: discrete control - barycenter control estimate}
    Let $f^\varepsilon_{\lambda, h}\in F^h_{+}$ be the solution to Problem \eqref{eq: discrete variational inequality} in case of element-wise constant discretization and $\pi_d f^\varepsilon_\lambda \in F^h_{+}$ the function defined in \eqref{eq: barycentric interpolation}. Let $p^\varepsilon_\lambda\in \Xs_T$ be the solution to Problem \eqref{eq: weak adjoint equation optimality conditions}. Assume that Assumptions \ref{assum: Nitsche's trick assumption} and \ref{assum: mesh superconvergence} hold, and $\ell \in \Ws^{1,\infty}\left(Q_T\right)\cap \Hs^2\left(Q_1\cup Q_2\right)$. Then, we have the following estimate
    $$
    \begin{aligned}
        &\norm{f^\varepsilon_{\lambda,h} - \pi_d\left(f^\varepsilon_{\lambda}\right)}_{\Ls^2\left(Q_T\right)}\le \\
        &\le \dfrac{C}{\lambda^\tau}\left[h^{\frac{3}{2}}\left(1+\dfrac{1}{\lambda}\right)\norm{\Ds \left(\ell p^\varepsilon_\lambda\right)}_{\LLs^\infty\left(Q_T\right)} + h^2\left(1+\dfrac{1}{\sqrt{\lambda^3}}\right)\left(\norm{U_d}_{\Ls^2\left(\omega_T\right)} +\norm{g}_{\Ls^2\left(Q_T\right)}+\varepsilon\right)\right]\text{ with } \tau = 
        \begin{cases}
            2 & \text{if } \lambda \le 4 \norm{\ell}_{\Ls^\infty\left(Q_T\right)}C_2, \\ 
            0 & \text{otherwise},
        \end{cases}
    \end{aligned}
    $$
    where the constant $C_2>0$ is defined in Lemma \ref{lem: discrete-auxiliary estimate variational}.
\end{lemma}

\begin{proof}
    We first choose $f_h = \pi_d f^\varepsilon_\lambda \in F^h_{+}$ in \eqref{eq: discrete variational inequality}, then add the corresponding inequality to \eqref{eq: barycentric variational inequality}. The result is
    \begin{equation}
        \label{eq: discrete control - barycenter control estimate 1}
        \lambda \norm{f^\varepsilon_{\lambda,h} - \pi_d f^\varepsilon_{\lambda}}^2_{\Ls^2\left(Q_T\right)}\le \left(\pi_d\left(\ell p^\varepsilon_\lambda\right) - \ell p^\varepsilon_{\lambda, h}, f^\varepsilon_{\lambda, h} - \pi_d f^\varepsilon_{\lambda}\right)_{\Ls^2\left(Q_T\right)} = J_7 + J_8 + J_9,
    \end{equation}
    where $J_7, J_8$, and $J_9$ are given by
    $$
    J_7 :=\left(\pi_d\left(\ell p^\varepsilon_\lambda\right)-\ell p^\varepsilon_\lambda, f^\varepsilon_{\lambda, h} - \pi_d f^\varepsilon_\lambda \right)_{\Ls^2\left(Q_T\right)}\qqq \text{and}\qqq J_8 := \left(\ell p^\varepsilon_\lambda - \ell p_h^\pi\left(f^\varepsilon_\lambda\right), f^\varepsilon_{\lambda, h} - \pi_d f^\varepsilon_\lambda \right)_{\Ls^2\left(Q_T\right)},
    $$
    and 
    $$
    J_9 := \left(\ell p_h^\pi\left(f^\varepsilon_\lambda\right) - \ell p^\varepsilon_{\lambda, h}, f^\varepsilon_{\lambda, h} - \pi_d f^\varepsilon_\lambda \right)_{\Ls^2\left(Q_T\right)}.
    $$
    
    Let us begin with $J_7$. First, note that on any element $K\in \mathcal{T}_h$, the function $f^\varepsilon_{\lambda, h}-\pi_d f^\varepsilon_\lambda $ is constant. We now split $J_7$ into two terms as follows
    $$
    J_7 = \int\limits_{\mathcal{T}_h\setminus \mathcal{T}_h^\ast}\left(\pi_d\left(\ell p^\varepsilon_\lambda\right)-\ell p^\varepsilon_\lambda\right)\left(f^\varepsilon_{\lambda, h} - \pi_d f^\varepsilon_\lambda \right)\dx\dt + \int\limits_{\mathcal{T}_h^\ast}\left(\pi_d\left(\ell p^\varepsilon_\lambda\right)-\ell p^\varepsilon_\lambda\right)\left(f^\varepsilon_{\lambda, h} - \pi_d f^\varepsilon_\lambda \right)\dx\dt =: J_7^a + J_7^b.
    $$
    We first estimate $J_7^a$. Clearly, if $K$ be an arbitrary element of $\mathcal{T}_h\setminus\mathcal{T}_h^\ast$, then either $K\subset\overline{Q_1}$ or $K\subset\overline{Q_2}$ (see Section \ref{sec: discretization}). Without loss of generality, assume that $K\subset\overline{Q_1}$. Under Assumption \ref{assum: Nitsche's trick assumption} and $\ell \in \Hs^2\left(Q_1\cup Q_2\right)$, we have $\ell p^\varepsilon_\lambda\in \Hs^2\left(K\right)$. By using Lemma \ref{lem: numerical integration} and \eqref{eq: exact-postprocess states and adjoint estimate 6}, we find
    $$
    \begin{aligned}
        \abs{\int\limits_{K}\left(\pi_d\left(\ell p^\varepsilon_\lambda\right)-\ell p^\varepsilon_\lambda\right)\left(f^\varepsilon_{\lambda, h} - \pi_d f^\varepsilon_\lambda \right)\dx\dt}&= \abs{\int\limits_{K}\left(\pi_d\left(\ell p^\varepsilon_\lambda\right)-\ell p^\varepsilon_\lambda\right)\dx\dt}\abs{f^\varepsilon_{\lambda, h} - \pi_d f^\varepsilon_\lambda }\\
        &\le Ch^2 \abs{K}^{\frac{1}{2}}\norm{\Ds^2\left(\ell p^\varepsilon_\lambda\right)}_{\mathbb{L}^2\left(K\right)}\abs{f^\varepsilon_{\lambda, h} - \pi_d f^\varepsilon_\lambda }\\
        &\le Ch^2 \abs{K}^{\frac{1}{2}}\norm{\Es_1 p^\varepsilon_\lambda}_{\Hs^2\left(K\right)}\abs{f^\varepsilon_{\lambda, h} - \pi_d f^\varepsilon_\lambda }.
    \end{aligned}
    $$
    Hence, we imply
    $$
    \begin{aligned}
        \sum\limits_{K\in \mathcal{T}_h\setminus \mathcal{T}_h^\ast}\abs{\int\limits_{K}\left(\pi_d\left(\ell p^\varepsilon_\lambda\right)-\ell p^\varepsilon_\lambda\right)\left(f^\varepsilon_{\lambda, h} - \pi_d f^\varepsilon_\lambda \right)\dx\dt}&\le Ch^2 \left(\sum_{i=1}^2\norm{\Es_i p^\varepsilon_\lambda}^2_{\Hs^2\left(Q_T\right)}\right)^{\frac{1}{2}}\left(\sum\limits_{K\in \mathcal{T}_h\setminus \mathcal{T}_h^\ast} \abs{K}\abs{f^\varepsilon_{\lambda, h} - \pi_d f^\varepsilon_\lambda }^2\right)^{\frac{1}{2}}\\
        &= Ch^2 \left(\sum_{i=1}^2\norm{\Es_i p^\varepsilon_\lambda}^2_{\Hs^2\left(Q_T\right)}\right)^{\frac{1}{2}}\left[\sum\limits_{K\in \mathcal{T}_h\setminus \mathcal{T}_h^\ast} 
        \int\limits_{K} \left(f^\varepsilon_{\lambda, h} - \pi_d f^\varepsilon_\lambda\right)^2\dx\dt\right]^{\frac{1}{2}}.
    \end{aligned}
    $$
    using the Cauchy-Schwarz inequality and \eqref{eq: extension operator}. Together with \eqref{eq: auxiliary-discrete adjoint estimate variational 4}, we thus arrive at
    $$
    J_7^a  \le Ch^2 \norm{p^\varepsilon_\lambda}_{\Hs^2\left(Q_1\cup Q_2\right)}\norm{f^\varepsilon_{\lambda, h} - \pi_d f^\varepsilon_\lambda }_{\Ls^2\left(Q_T\right)}\le Ch^2\left(1+\dfrac{1}{\sqrt{\lambda}}\right)\left(\norm{U_d}_{\Ls^2\left(\omega_T\right)} +\norm{g}_{\Ls^2\left(Q_T\right)}+\varepsilon\right)\norm{f^\varepsilon_{\lambda, h} - \pi_d f^\varepsilon_\lambda }_{\Ls^2\left(Q_T\right)}.
    $$
    Next, we estimate $J_7^b$. Consider an arbitrary element $K\in \mathcal{T}_h^\ast$. Under Assumption \ref{assum: Nitsche's trick assumption} and $\ell \in \Ws^{1,\infty}\left(Q_T\right)$, we now have $\ell p^\varepsilon_\lambda\in \Ws^{1,\infty}\left(K\right)$. By invoking Lemma \ref{lem: numerical integration} once more, we obtain
    $$
    \begin{aligned}
        \abs{\int\limits_{K}\left(\pi_d\left(\ell p^\varepsilon_\lambda\right)-\ell p^\varepsilon_\lambda\right)\left(f^\varepsilon_{\lambda, h} - \pi_d f^\varepsilon_\lambda \right)\dx\dt}&= \abs{\int\limits_{K}\left(\pi_d\left(\ell p^\varepsilon_\lambda\right)-\ell p^\varepsilon_\lambda\right)\dx\dt}\abs{f^\varepsilon_{\lambda, h} - \pi_d f^\varepsilon_\lambda }\\
        &\le \abs{K}\norm{\pi_d\left(\ell p^\varepsilon_\lambda\right)-\ell p^\varepsilon_\lambda}_{\Ls^\infty\left(K\right)}\abs{f^\varepsilon_{\lambda, h} - \pi_d f^\varepsilon_\lambda }\\
        &\le C \abs{K}h\norm{\Ds \left(\ell p^\varepsilon_\lambda\right)}_{\LLs^\infty\left(K\right)}\abs{f^\varepsilon_{\lambda, h} - \pi_d f^\varepsilon_\lambda }.
    \end{aligned}
    $$
    We then sum over all elements $K\in \mathcal{T}_h^\ast$, apply the Cauchy Schwarz inequality, and \eqref{eq: Th^ast area} to get
    $$
    \begin{aligned}
        J_7^b \le \sum\limits_{K\in \mathcal{T}_h^\ast} \abs{\int\limits_{K}\left(\pi_d\left(\ell p^\varepsilon_\lambda\right)-\ell p^\varepsilon_\lambda\right)\left(f^\varepsilon_{\lambda, h} - \pi_d f^\varepsilon_\lambda \right)\dx\dt}&\le C\left(\sum\limits_{K\in \mathcal{T}_h^\ast}\abs{K}\right)^{\frac{1}{2}}h\norm{\Ds \left(\ell p^\varepsilon_\lambda\right)}_{\LLs^\infty\left(Q_T\right)}\left(\sum\limits_{K\in \mathcal{T}_h^\ast}\abs{K}\abs{f^\varepsilon_{\lambda, h} - \pi_d f^\varepsilon_\lambda }^2\right)^{\frac{1}{2}}\\
        &\le C h^{\frac{3}{2}} \norm{\Ds \left(\ell p^\varepsilon_\lambda\right)}_{\LLs^\infty\left(Q_T\right)}\left[\sum\limits_{K\in \mathcal{T}_h^\ast} 
        \int\limits_{K} \left(f^\varepsilon_{\lambda, h} - \pi_d f^\varepsilon_\lambda\right)^2\dx\dt\right]^{\frac{1}{2}}\\
        &\le Ch^{\frac{3}{2}} \norm{\Ds \left(\ell p^\varepsilon_\lambda\right)}_{\LLs^\infty\left(Q_T\right)}\norm{f^\varepsilon_{\lambda, h} - \pi_d f^\varepsilon_\lambda }_{\Ls^2\left(Q_T\right)}.
    \end{aligned}
    $$
    Therefore, we bound $J_7$ by
    \begin{equation}
        \label{eq: discrete control - barycenter control estimate 2}
        J_7 \le C\left[h^2\left(1+\dfrac{1}{\sqrt{\lambda}}\right)\left(\norm{U_d}_{\Ls^2\left(\omega_T\right)} +\norm{g}_{\Ls^2\left(Q_T\right)}+\varepsilon\right) + h^{\frac{3}{2}}\norm{\Ds \left(\ell p^\varepsilon_\lambda\right)}_{\LLs^\infty\left(Q_T\right)}\right]\norm{f^\varepsilon_{\lambda, h} - \pi_d f^\varepsilon_\lambda }_{\Ls^2\left(Q_T\right)}.
    \end{equation}
    
    Next, we estimate $J_8$. To do so, we first need to bound $\norm{p^\varepsilon_\lambda - p_h^\pi\left(f^\varepsilon_\lambda\right)}_{\Ls^2\left(Q_T\right)}$. By combining the triangle inequality with Lemmas \ref{lem: auxiliary-discrete adjoint estimate variational} and \ref{lem: exact-postprocess states and adjoint estimate}, we have
    $$
    \begin{aligned}
        \norm{p^\varepsilon_\lambda - p_h^\pi\left(f^\varepsilon_\lambda\right)}_{\Ls^2\left(Q_T\right)}&\le \norm{p^\varepsilon_\lambda - p_h\left(f^\varepsilon_\lambda\right)}_{\Ls^2\left(Q_T\right)} + \norm{p_h\left(f^\varepsilon_\lambda\right) - p_h^\pi\left(f^\varepsilon_\lambda\right)}_{\Ls^2\left(Q_T\right)}\\
        &\le Ch^2\left(1+\dfrac{1}{\lambda}\right)\left(1+\dfrac{1}{\sqrt{\lambda}}\right)\left(\norm{U_d}_{\Ls^2\left(\omega_T\right)} +\norm{g}_{\Ls^2\left(Q_T\right)}+\varepsilon\right) + \dfrac{Ch^{\frac{3}{2}}}{\lambda} \norm{\Ds \left(\ell p^\varepsilon_\lambda\right)}_{\LLs^\infty\left(Q_T\right)}.
    \end{aligned}
    $$
    We hence end up with
    \begin{equation}
        \label{eq: discrete control - barycenter control estimate 4}
        \begin{aligned}
            J_8 &\le C\norm{p^\varepsilon_\lambda - p_h^\pi\left(f^\varepsilon_\lambda\right)}_{\Ls^2\left(Q_T\right)}\norm{f^\varepsilon_{\lambda, h} - \pi_d f^\varepsilon_\lambda }_{\Ls^2\left(Q_T\right)}\\
            &\le C\left[h^2\left(1+\dfrac{1}{\lambda}\right)\left(1+\dfrac{1}{\sqrt{\lambda}}\right)\left(\norm{U_d}_{\Ls^2\left(\omega_T\right)} +\norm{g}_{\Ls^2\left(Q_T\right)}+\varepsilon\right) + \dfrac{h^{\frac{3}{2}}}{\lambda} \norm{\Ds \left(\ell p^\varepsilon_\lambda\right)}_{\LLs^\infty\left(Q_T\right)}\right]\norm{f^\varepsilon_{\lambda, h} - \pi_d f^\varepsilon_\lambda }_{\Ls^2\left(Q_T\right)}.
        \end{aligned}
    \end{equation}
    
    For $J_9$, we follow the arguments of \eqref{eq: exact-discrete control estimate variational 3} to get
    $$
    J_9\le \norm{\ell}_{\Ls^\infty\left(Q_T\right)}C_2\norm{f^\varepsilon_{\lambda, h} - \pi_d f^\varepsilon_\lambda }^2_{\Ls^2\left(Q_T\right)}.
    $$
    
    Finally, we substitute \eqref{eq: discrete control - barycenter control estimate 2}, \eqref{eq: discrete control - barycenter control estimate 4}, and the last inequality into \eqref{eq: discrete control - barycenter control estimate 1} to deduce that
    $$
    \begin{aligned}
        &\left(\lambda - \norm{\ell}_{\Ls^\infty\left(Q_T\right)}C_2\right)\norm{f^\varepsilon_{\lambda, h} - \pi_d f^\varepsilon_\lambda }^2_{\Ls^2\left(Q_T\right)} \le \\
        &\le C\left[h^2\left(1+\dfrac{1}{\lambda}\right)\left(1+\dfrac{1}{\sqrt{\lambda}}\right)\left(\norm{U_d}_{\Ls^2\left(\omega_T\right)} +\norm{g}_{\Ls^2\left(Q_T\right)}+\varepsilon\right) + h^{\frac{3}{2}}\left(1+\dfrac{1}{\lambda}\right)\norm{\Ds \left(\ell p^\varepsilon_\lambda\right)}_{\LLs^\infty\left(Q_T\right)}\right]\norm{f^\varepsilon_{\lambda, h} - \pi_d f^\varepsilon_\lambda }_{\Ls^2\left(Q_T\right)}.
    \end{aligned}
    $$
    By applying the inequality 
    $$
    \lambda - \norm{\ell}_{\Ls^\infty\left(Q_T\right)}C_2 \ge
    \begin{cases}
        \dfrac{3\lambda^2}{16 \norm{\ell}_{\Ls^\infty\left(Q_T\right)}C_2} &\text{if }\lambda \le 4 \norm{\ell}_{\Ls^\infty\left(Q_T\right)}C_2,\\
        3 \norm{\ell}_{\Ls^\infty\left(Q_T\right)}C_2 & \text{if } \lambda \ge 4 \norm{\ell}_{\Ls^\infty\left(Q_T\right)}C_2,
    \end{cases}
    $$
    we obtain
    $$
    \begin{aligned}
        \norm{f^\varepsilon_{\lambda, h} - \pi_d f^\varepsilon_\lambda }_{\Ls^2\left(Q_T\right)}&\le \dfrac{C}{\lambda^\tau}\left[h^2\left(1+\dfrac{1}{\lambda}\right)\left(1+\dfrac{1}{\sqrt{\lambda}}\right)\left(\norm{U_d}_{\Ls^2\left(\omega_T\right)} +\norm{g}_{\Ls^2\left(Q_T\right)}+\varepsilon\right) + h^{\frac{3}{2}}\left(1+\dfrac{1}{\lambda}\right)\norm{\Ds \left(\ell p^\varepsilon_\lambda\right)}_{\LLs^\infty\left(Q_T\right)}\right]\\
        &\le \dfrac{C}{\lambda^\tau}\left[h^2\left(1+\dfrac{1}{\sqrt{\lambda^3}}\right)\left(\norm{U_d}_{\Ls^2\left(\omega_T\right)} +\norm{g}_{\Ls^2\left(Q_T\right)}+\varepsilon\right) + h^{\frac{3}{2}}\left(1+\dfrac{1}{\lambda}\right)\norm{\Ds \left(\ell p^\varepsilon_\lambda\right)}_{\LLs^\infty\left(Q_T\right)}\right],
    \end{aligned}
    $$
    where 
    $$
    \tau = 
    \begin{cases}
        2 & \text{if } \lambda \le 4 \norm{\ell}_{\Ls^\infty\left(Q_T\right)}C_2, \\ 
        0 & \text{otherwise}.
    \end{cases}
    $$
    We finish the proof.
\end{proof}

From Theorem \ref{theo: error estimate constant} and Lemma \ref{lem: discrete control - barycenter control estimate}, we observe that the discrete source $f^\varepsilon_{\lambda, h}$ (in case of element-wise constant discretization) approximates the interpolant $\pi_d f^\varepsilon_\lambda $ better than $f^\varepsilon_\lambda$ itself. By combining Lemmas \ref{lem: exact-postprocess states and adjoint estimate} and \ref{lem: discrete control - barycenter control estimate}, we arrive at the following superlinear convergence:

\begin{theorem}
    \label{theo: error estimate postprocessing}
    Let $f^\varepsilon_\lambda\in F_{+}$ be the solution to Problem \eqref{eq: problem formulation} and $f_h^\dagger$ be the post-processing solution in \eqref{eq: post-processing solution}. Let $p^\varepsilon_\lambda\in \Xs_T$ be the solution to Problem \eqref{eq: weak adjoint equation optimality conditions} and $C_2>0$ be the constant defined in Lemma \ref{lem: discrete-auxiliary estimate variational}. Assume that Assumptions \ref{assum: Nitsche's trick assumption} and \ref{assum: mesh superconvergence} hold, and $\ell \in \Ws^{1,\infty}\left(Q_T\right)\cap \Hs^2\left(Q_1\cup Q_2\right)$. Then, we have
    $$
    \norm{f^\varepsilon_\lambda - f_h^\dagger}_{\Ls^2\left(Q_T\right)}\le \dfrac{C}{\lambda^{2+\tau}}\left[h^{\frac{3}{2}}\norm{\Ds \left(\ell p^\varepsilon_\lambda\right)}_{\LLs^\infty\left(Q_T\right)} + \dfrac{h^2}{\sqrt{\lambda}}\left(\norm{U_d}_{\Ls^2\left(\omega_T\right)} +\norm{g}_{\Ls^2\left(Q_T\right)}+\varepsilon\right)\right] = \mathcal{O}\left(h^{\frac{3}{2}}\right),
    $$
    where 
    $$
    \tau = 
    \begin{cases}
        2 & \text{if } \lambda \le 4 \norm{\ell}_{\Ls^\infty\left(Q_T\right)}C_2, \\ 
        0 & \text{otherwise}.
    \end{cases}
    $$
\end{theorem}

\begin{proof}
    By using \eqref{eq: projection formula of f}, \eqref{eq: post-processing solution}, and the nonexpansivity of $\operatorname{Proj}_{F_{+}}$ in $\Ls^2\left(Q_T\right)$ \cite[Lemma 1.10]{HPU2009}, we have
    $$
    \lambda \norm{f^\varepsilon_\lambda - f_h^\dagger}_{\Ls^2\left(Q_T\right)} = \lambda \norm{\operatorname{Proj}_{F_{+}}\left(-\dfrac{1}{\lambda} \ell p^\varepsilon_\lambda\right) - \operatorname{Proj}_{F_{+}}\left(-\dfrac{1}{\lambda} \ell p^\varepsilon_{\lambda, h}\right)}_{\Ls^2\left(Q_T\right)}\le C\norm{p^\varepsilon_\lambda - p^\varepsilon_{\lambda, h}}_{\Ls^2\left(Q_T\right)}.
    $$
    Before estimating the right-hand side of this inequality, let us deal with $\norm{p_h^\pi\left(f^\varepsilon_\lambda\right)- p^\varepsilon_{\lambda, h}}_{\Ls^2\left(Q_T\right)}$. We subtract \eqref{eq: discrete adjoint optimality conditions} from \eqref{eq: auxiliary adjoint postprocessing} to get
    $$
    a_h^\prime\left(p_h^\pi\left(f^\varepsilon_\lambda\right) - p^\varepsilon_{\lambda, h}, \phi_h\right) = \left(\chi_{\omega_T}\left(u_h^\pi\left(f^\varepsilon_\lambda\right) - u^\varepsilon_{\lambda, h}\right), \phi_h\right)_{\Ls^2\left(Q_T\right)}\qqqq\forall \phi_h \in \Vs_{h,T}.
    $$
    We then invoke \eqref{eq: compare L2 and the Y norm} and the technique in \eqref{eq: discrete-auxiliary estimate variational 3} to arrive at 
    $$
    \norm{p_h^\pi\left(f^\varepsilon_\lambda\right)- p^\varepsilon_{\lambda, h}}_{\Ls^2\left(Q_T\right)} \le C\norm{p_h^\pi\left(f^\varepsilon_\lambda\right)- p^\varepsilon_{\lambda, h}}_{h,T} \le C \norm{u_h^\pi\left(f^\varepsilon_\lambda\right)- u^\varepsilon_{\lambda, h}}_{\Ls^2\left(\omega_T\right)}\le C\norm{\pi_d f^\varepsilon_\lambda - f^\varepsilon_{\lambda, h}}_{\Ls^2\left(Q_T\right)}.
    $$
    Therefore, thanks to the triangle inequality, we obtain
    $$
    \begin{aligned}
        \norm{p^\varepsilon_\lambda - p^\varepsilon_{\lambda, h}}_{\Ls^2\left(Q_T\right)} &\le \norm{p^\varepsilon_\lambda - p_h\left(f^\varepsilon_\lambda\right)}_{\Ls^2\left(Q_T\right)} + \norm{p_h\left(f^\varepsilon_\lambda\right) - p_h^\pi\left(f^\varepsilon_\lambda\right)}_{\Ls^2\left(Q_T\right)} + \norm{p_h^\pi\left(f^\varepsilon_\lambda\right)- p^\varepsilon_{\lambda, h}}_{\Ls^2\left(Q_T\right)}\\
        &\le \norm{p^\varepsilon_\lambda - p_h\left(f^\varepsilon_\lambda\right)}_{\Ls^2\left(Q_T\right)} + \norm{p_h\left(f^\varepsilon_\lambda\right) - p_h^\pi\left(f^\varepsilon_\lambda\right)}_{\Ls^2\left(Q_T\right)} + C\norm{\pi_d f^\varepsilon_\lambda - f^\varepsilon_{\lambda, h}}_{\Ls^2\left(Q_T\right)},
    \end{aligned}
    $$
    and thus
    $$
    \norm{f^\varepsilon_\lambda - f_h^\dagger}_{\Ls^2\left(Q_T\right)}\le \dfrac{C}{\lambda}\left(\norm{p^\varepsilon_\lambda - p_h\left(f^\varepsilon_\lambda\right)}_{\Ls^2\left(Q_T\right)} + \norm{p_h\left(f^\varepsilon_\lambda\right) - p_h^\pi\left(f^\varepsilon_\lambda\right)}_{\Ls^2\left(Q_T\right)} + \norm{\pi_d f^\varepsilon_\lambda - f^\varepsilon_{\lambda, h}}_{\Ls^2\left(Q_T\right)}\right).
    $$
    By combining this with Lemmas \ref{lem: auxiliary-discrete adjoint estimate variational}, \ref{lem: exact-postprocess states and adjoint estimate}, and \ref{lem: discrete control - barycenter control estimate}, we obtain the desired result.
\end{proof}

\begin{remark}
    Under appropriate conditions, we have shown that the post-processing solution $f_h^\dagger$, constructed by \eqref{eq: post-processing solution}, converges to $f^\varepsilon_\lambda$ more rapidly than the element-wise constant solution $f^\varepsilon_{\lambda, h}$ in Theorem \ref{theo: error estimate constant}. In particular, a superlinear convergence with respect to the $\Ls^2\left(Q_T\right)$-norm has been established. 
    
    Nevertheless, the order of $\mathcal{O}\left(h^{\frac{3}{2}}\right)$ in Theorem \ref{theo: error estimate postprocessing} is sub-optimal when compared to the results in \cite[Theorem 2.4]{MR2004}, \cite[Theorem 2.8]{RV2006}, and \cite[Corollary 5.17]{MV2008} for optimal control problems. The authors derived second-order error estimates for the control in these studies. However, they relied on the control exhibiting higher regularity when it is away from the free boundary between the active and inactive sets. 
    
    Our situation differs, as the regularity of $f^\varepsilon_\lambda$ depends not only on the transition between the sets $\mathcal{M}_h^1$ and $\mathcal{M}_h^2$, but also on the position of $\Gamma^\ast$ and the smoothness of $\ell$. To achieve a quadratic convergence rate, we require a similar inequality to the first one in Lemma \ref{lem: numerical integration} in case of $v  \in \Hs^1\left(K\right)\cap \Hs^2\left(K_1\cup K_2\right)$, where $K\in \mathcal{T}_h^\ast$ and $K_i:= K\cap Q_i\ (i=1,2)$. To the best of our knowledge, no studies in the literature have addressed this result. Based on the Bramble-Hilbert lemma \cite[Lemma 11.9]{Ern2021} and \cite[Lemma 3.2]{MR2004}, it is likely that this variant of the result in Lemma \ref{lem: numerical integration} for locally $\Hs^2$-smooth functions may not hold.
\end{remark}

\subsection{Regularization parameter selection}
\label{subsec: regularization parameter selection}

We are now positioned to estimate the error between the exact source $f_+ \in F_{+}$ and either the solution $f^\varepsilon_{\lambda,h}\in F^h_{+}$ to Problem \eqref{eq: discrete state optimality conditions}--\eqref{eq: discrete variational inequality}, or the post-processing solution $f_h^\dagger\in F^h_{+}$ in \eqref{eq: post-processing solution}. By incorporating the results from Lemma \ref{lem: convergence rate}, Theorems \ref{theo: error estimate variational}, \ref{theo: error estimate constant}, and \ref{theo: error estimate postprocessing} into \eqref{eq: overall error 1} and \eqref{eq: overall error 2}, we derive the following theorem:

\begin{theorem}
    \label{theo: overall convergence rate}
    Let $f_+ \in F_{+}$ denote the exact source and $f^\varepsilon_{\lambda,h}\in F^h_{+}$ the solution to Problem \eqref{eq: discrete state optimality conditions}--\eqref{eq: discrete variational inequality}. Let $p^\varepsilon_\lambda\in \Xs_T$ be the solution to Problem \eqref{eq: weak adjoint equation optimality conditions} and $C_2>0$ be the constant defined in Lemma \ref{lem: discrete-auxiliary estimate variational}. Assume that the conditions in Lemma \ref{lem: convergence rate} and Assumption \ref{assum: Nitsche's trick assumption} are satisfied. When $\lambda \le 4 \norm{\ell}_{\Ls^\infty\left(Q_T\right)}C_2$, the following results hold:
    \begin{enumerate}
        \item In case of variational discretization: We have the following estimate
        $$
        \norm{f_+ -f^{\varepsilon}_{\lambda, h}}_{\Ls^2\left(Q_T\right)} \le \sqrt{\lambda}\norm{\zeta}_{\Ls^2\left(\omega_T\right)} + \dfrac{\varepsilon}{\sqrt{\lambda}} + Ch^2\left(1+ \dfrac{1}{\lambda^2}\right)\left(\norm{U_d}_{\Ls^2\left(\omega_T\right)} +\norm{g}_{\Ls^2\left(Q_T\right)}+\varepsilon\right).
        $$
        If $\lambda = \Theta\left(h^{\frac{4}{5}} + \varepsilon^{\frac{4}{5}}\right)$, then $f^\varepsilon_{\lambda,h}\to f_+ $ in $\Ls^2\left(Q_T\right)$ as $h\to 0$ and $\varepsilon\to 0$, with a convergence rate of $\mathcal{O}\left(h^{\frac{2}{5}}+\varepsilon^{\frac{2}{5}}\right)$.
        \item In case of element-wise constant discretization: If $\ell \in \Ws^{1,\infty}\left(Q_T\right)$, then the following estimate holds
        $$
        \begin{aligned}
            &\norm{f_+ -f^{\varepsilon}_{\lambda, h}}_{\Ls^2\left(Q_T\right)} \le \sqrt{\lambda}\norm{\zeta}_{\Ls^2\left(\omega_T\right)} + \dfrac{\varepsilon}{\sqrt{\lambda}}\\
            &+ \dfrac{Ch}{\lambda} \sqrt{\left(1+\dfrac{1}{\sqrt{\lambda}}\right)\left(\norm{U_d}_{\Ls^2\left(\omega_T\right)} +\norm{g}_{\Ls^2\left(Q_T\right)}+\varepsilon\right)}\norm{\Ds \left(\ell p^{\varepsilon}_{\lambda}\right)}^{\frac{1}{2}}_{\LLs^\infty\left(Q_T\right)} + Ch^2\left(1+\dfrac{1}{\sqrt{\lambda^5}}\right)\left(\norm{U_d}_{\Ls^2\left(\omega_T\right)} +\norm{g}_{\Ls^2\left(Q_T\right)}+\varepsilon\right).
        \end{aligned}
        $$
        \item In case of post-processing discretization: Let $f_h^\dagger\in F^h_{+}$ be the post-processing solution in \eqref{eq: post-processing solution}. If Assumption \ref{assum: mesh superconvergence} is satisfied and $\ell \in \Ws^{1,\infty}\left(Q_T\right)\cap \Hs^2\left(Q_1\cup Q_2\right)$, then their holds the following estimate
        $$
        \norm{f_+ -f_h^\dagger}_{\Ls^2\left(Q_T\right)} \le \sqrt{\lambda}\norm{\zeta}_{\Ls^2\left(\omega_T\right)} + \dfrac{\varepsilon}{\sqrt{\lambda}} + \dfrac{C}{\lambda^{4}}\left[h^{\frac{3}{2}}\norm{\Ds \left(\ell p^\varepsilon_\lambda\right)}_{\LLs^\infty\left(Q_T\right)} + \dfrac{h^2}{\sqrt{\lambda}}\left(\norm{U_d}_{\Ls^2\left(\omega_T\right)} +\norm{g}_{\Ls^2\left(Q_T\right)}+\varepsilon\right)\right].
        $$
    \end{enumerate}
    The function $\zeta\in \Ls^2\left(\omega_T\right)$ in these estimates is defined in Lemma \ref{lem: convergence rate}.
\end{theorem}

\begin{proof}
    Each of the three inequalities directly follows from Lemma \ref{lem: convergence rate} and Theorems \ref{theo: error estimate variational}, \ref{theo: error estimate constant}, and \ref{theo: error estimate postprocessing}, respectively. In the case of variational discretization, we observe that
    $$
    \begin{aligned}
        \norm{f_+ -f^{\varepsilon}_{\lambda, h}}_{\Ls^2\left(Q_T\right)} &\le \sqrt{\lambda}\norm{\zeta}_{\Ls^2\left(\omega_T\right)} + \dfrac{\varepsilon}{\sqrt{\lambda}} + Ch^2\left(1+ \dfrac{1}{\lambda^2}\right)\left(\norm{U_d}_{\Ls^2\left(\omega_T\right)} +\norm{g}_{\Ls^2\left(Q_T\right)}+\varepsilon\right)\\
        &\le \sqrt{\lambda}\norm{\zeta}_{\Ls^2\left(\omega_T\right)} + \dfrac{\varepsilon}{\sqrt{\lambda}} + Ch^2\left(1+ \dfrac{1}{\lambda^2}\right)\left(\norm{U_d}_{\Ls^2\left(\omega_T\right)} +\norm{g}_{\Ls^2\left(Q_T\right)}\right) + C\left(1+ \dfrac{1}{\lambda^2}\right)\left(h^4+\varepsilon^2\right),
    \end{aligned}
    $$
    using the Cauchy inequality. Consequently, if $\lambda = \Theta\left(h^{\frac{4}{5}} + \varepsilon^{\frac{4}{5}}\right)$, then $\norm{f_+ -f^{\varepsilon}_{\lambda, h}}_{\Ls^2\left(Q_T\right)}\le C\left(h^{\frac{2}{5}} + \varepsilon^{\frac{2}{5}}\right)$, which guarantees the desired convergence of the discrete source along with the corresponding convergence rate. Here, the constant $C > 0$ is independent of $\lambda$, $h$, and $\varepsilon$.
\end{proof}

To address the convergence behavior, we specifically focused on the case where $\lambda>0$ is sufficiently small, as presented in Theorem \ref{theo: overall convergence rate}. In particular, we proposed an a priori choice for $\lambda$, depending on $h$ and $\varepsilon$, to ensure that $f^\varepsilon_{\lambda,h} \to f_+ $ in $\Ls^2\left(Q_T\right)$ under variational discretization. However, it is important to note that similar results have not yet been established for the other two scenarios: element-wise constant discretization and the post-processing strategy. This limitation arises because the norm $\norm{\Ds \left(\ell p^\varepsilon_\lambda\right)}_{\LLs^\infty\left(Q_T\right)}$ on the right-hand side of these estimates remains dependent on both $\lambda$ and $\varepsilon$. 

To obtain the desired strong convergence, an explicit dependence of this norm on $\lambda$ and $\varepsilon$ is required. Typically, this necessitates a stability estimate of the solution $p^\varepsilon_\lambda\in \Xs_T$ to Problem \eqref{eq: weak adjoint equation optimality conditions} in the norm $\norm{\Ds \cdot}_{\LLs^\infty\left(Q_T\right)}$. Under Assumption \ref{assum: Nitsche's trick assumption}, recall that $p^\varepsilon_\lambda\in \Ws^{1,\infty}\left(Q_T\right)$ (see the proof of Lemma \ref{lem: control regularity}). If there exists a constant $C>0$, independent of $\lambda$ and $\varepsilon$, such that
\begin{equation}
    \label{eq: extra regularity estimate}
    \norm{\Ds p^\varepsilon_\lambda}_{\LLs^\infty\left(Q_T\right)}\le C \norm{u^\varepsilon_\lambda - z^\varepsilon_d}_{\Ls^2\left(\omega_T\right)},
\end{equation}
and if $\ell$ is a positive constant on $Q_T$, then, together with \eqref{eq: auxiliary-discrete adjoint estimate variational 4}, we obtain
$$
\norm{\Ds \left(\ell p^\varepsilon_\lambda\right)}_{\LLs^\infty\left(Q_T\right)} = C\norm{\Ds p^\varepsilon_\lambda}_{\LLs^\infty\left(Q_T\right)} \le C\left(1+\dfrac{1}{\sqrt{\lambda}}\right)\left(\norm{U_d}_{\Ls^2\left(\omega_T\right)} +\norm{g}_{\Ls^2\left(Q_T\right)}+\varepsilon\right).
$$
Consequently, we arrive at the following result:

\begin{corollary}
    \label{coro: a priori choice for lambda}
    Let $f_+ \in F_{+}$ denote the exact source and $C_2>0$ be the constant defined in Lemma \ref{lem: discrete-auxiliary estimate variational}. Assume that the conditions in Lemma \ref{lem: convergence rate}, Assumption \ref{assum: Nitsche's trick assumption}, and \eqref{eq: extra regularity estimate} are satisfied. When $\lambda \le 4 \norm{\ell}_{\Ls^\infty\left(Q_T\right)}C_2$ and $\ell$ is a positive constant on $Q_T$, the following results hold:
    \begin{enumerate}
        \item Let $f^\varepsilon_{\lambda,h}\in F^h_{+}$ be the solution to Problem \eqref{eq: discrete state optimality conditions}--\eqref{eq: discrete variational inequality} under element-wise constant discretization. The following estimate holds
        $$
        \begin{aligned}
            &\norm{f_+ -f^{\varepsilon}_{\lambda, h}}_{\Ls^2\left(Q_T\right)} \le \sqrt{\lambda}\norm{\zeta}_{\Ls^2\left(\omega_T\right)} + \dfrac{\varepsilon}{\sqrt{\lambda}}\\
            &+ \dfrac{Ch}{\lambda} \left(1+\dfrac{1}{\sqrt{\lambda}}\right)\left(\norm{U_d}_{\Ls^2\left(\omega_T\right)} +\norm{g}_{\Ls^2\left(Q_T\right)}+\varepsilon\right) + Ch^2\left(1+\dfrac{1}{\sqrt{\lambda^5}}\right)\left(\norm{U_d}_{\Ls^2\left(\omega_T\right)} +\norm{g}_{\Ls^2\left(Q_T\right)}+\varepsilon\right).
        \end{aligned}
        $$
        If $\lambda = \Theta\left(h^{\frac{1}{2}} + \varepsilon^{\frac{2}{3}}\right)$, then $f^\varepsilon_{\lambda,h}\to f_+ $ in $\Ls^2\left(Q_T\right)$ as $h\to 0$ and $\varepsilon\to 0$, with a convergence rate of $\mathcal{O}\left(h^{\frac{1}{4}}+\varepsilon^{\frac{1}{3}}\right)$.
        \item Let $f_h^\dagger\in F^h_{+}$ be the post-processing solution in \eqref{eq: post-processing solution}. If Assumption \ref{assum: mesh superconvergence} holds, we have the following estimate
        $$
        \norm{f_+ -f_h^\dagger}_{\Ls^2\left(Q_T\right)} \le \sqrt{\lambda}\norm{\zeta}_{\Ls^2\left(\omega_T\right)} + \dfrac{\varepsilon}{\sqrt{\lambda}} + \dfrac{C}{\lambda^{4}}\left[h^{\frac{3}{2}}\left(1+\dfrac{1}{\sqrt{\lambda}}\right) + \dfrac{h^2}{\sqrt{\lambda}}\right]\left(\norm{U_d}_{\Ls^2\left(\omega_T\right)} +\norm{g}_{\Ls^2\left(Q_T\right)}+\varepsilon\right).
        $$
        If $\lambda = \Theta\left(h^{\frac{3}{10}} + \varepsilon^{\frac{2}{5}}\right)$, then $f_h^\dagger \to f_+ $ in $\Ls^2\left(Q_T\right)$ as $h\to 0$ and $\varepsilon\to 0$, with a convergence rate of $\mathcal{O}\left(h^{\frac{3}{20}}+\varepsilon^{\frac{1}{5}}\right)$.
    \end{enumerate}
    Here, $\zeta\in \Ls^2\left(\omega_T\right)$ is the function defined in Lemma \ref{lem: convergence rate}.
\end{corollary}

Finally, we observe that in Theorem \ref{theo: overall convergence rate} and Corollary \ref{coro: a priori choice for lambda}, the convergence rate decreases progressively. This behavior is attributed to the increasing exponents of $\lambda$ in the denominators of the respective estimates. Furthermore, in all cases, $f^\varepsilon_{\lambda,h}$ and $f_h^\dagger$ strongly converge to $f_+$ in $\Ls^2\left(Q_T\right)$ at half the rate of $\lambda$ approaching zero, contrasting with the unconstrained inverse problem; see \cite[Theorem 2.5]{HLO2024}.

\section{Conclusion}
\label{sec:conclusion}

We present an error analysis for the numerical solution of an inverse source problem associated with the advection-diffusion equation featuring moving subdomains. The numerical scheme integrates a space-time method with multiple source discretization strategies. Specifically, the state and adjoint problems are discretized using the space-time interface-fitted method. Under appropriate conditions, we establish two second-order error estimates in the $\Ls^2$-norms. For the regularized source, discretization is performed sequentially using the variational approach, the element-wise constant discretization, and the post-processing strategy. We derive optimal error estimates for the variational approach and the element-wise constant discretization while achieving superlinear convergence for the post-processing strategy. Furthermore, we propose a priori choices for $\lambda$ to ensure that $f^\varepsilon_{\lambda,h}$ and $f_h^\dagger$ strongly converge to $f_+$ in $\Ls^2\left(Q_T\right)$ as $h$ and $\varepsilon$ approach zero. The convergence rate is established in each case. 

For future work, we aim to extend these results to three-dimensional spatial domains and enhance the convergence order of the post-processing strategy. Another direction of interest is to develop additional error estimates where $\lambda$ appears in the numerator of the terms on the right-hand side, as demonstrated in \cite{WZ2010, JZ2021}, to achieve higher-order convergence rates.

\section*{Funding}

This research was supported by Vietnam National Foundation for Science and Technology Development (NAFOSTED) under grant number  101.02-2024.12.



\bibliographystyle{elsarticle-num} 
\bibliography{abrv_ref}





\end{document}

%% file: style1.tex
\newcommand{\R}{\mathbb R}

\DeclareMathOperator{\Ls}{L}

\DeclareMathOperator{\LLs}{\mathbf L}
\DeclareMathOperator{\Hs}{H}
\DeclareMathOperator{\HHs}{\mathbf H}
\DeclareMathOperator{\Ws}{W}

\DeclareMathOperator{\Cs}{C}
\DeclareMathOperator{\CCs}{\mathbf C}
\DeclareMathOperator{\Ds}{D}
\DeclareMathOperator{\Es}{E}

\DeclareMathOperator{\Vs}{V}

\DeclareMathOperator{\Xs}{X}
\DeclareMathOperator{\Ys}{Y}

\newcommand{\brac}[1]{\left\lbrace{#1}\right\rbrace}

\newcommand{\norm}[1]{\left\lVert{#1}\right\rVert}

\newcommand{\abs}[1]{\left\vert{#1}\right\vert}

\newcommand{\inprod}[1]{\left\langle{#1}\right\rangle}

\newcommand{\vertiii}[1]{{\left\vert\kern-0.25ex\left\vert\kern-0.25ex\left\vert #1 
    \right\vert\kern-0.25ex\right\vert\kern-0.25ex\right\vert}}
    
\newcommand{\bm}[1]{\boldsymbol{#1}}

\newcommand{\Om}{\Omega}

\newcommand{\vphi}{\varphi}
\newcommand{\fa}{\forall}

\newcommand{\xb}{\bm{x}}

\newcommand{\vb}{\mathbf{v}}

\newcommand{\dx}{\,\mathrm{d}\xb}
\newcommand{\ds}{\,\mathrm{d}s}
\newcommand{\dt}{\,\mathrm{d}t}

\newcommand{\q}{\quad}
\newcommand{\qq}{\qquad}
\newcommand{\qqq}{\qquad\quad}
\newcommand{\qqqq}{\qquad\qquad}

%% file: style2.tex
\newtheorem{theorem}{Theorem}[section]
\newtheorem{lemma}{Lemma}[section]

\newtheorem{corollary}{Corollary}[section]
\newtheorem{definition}{Definition}[section]

\newtheorem{remark}{Remark}[section]
\newtheorem{assumption}{Assumption}[section]